\documentclass{article}
\usepackage{amsmath}
\usepackage{amsthm}
\usepackage{graphics}
\usepackage{epsfig}

\newtheorem{theorem}{Theorem}[section]

\newtheorem{lem}[theorem]{Lemma}
\newtheorem{prop}[theorem]{Proposition}

\theoremstyle{definition}

\newtheorem{rem}[theorem]{Remark}

\newtheorem{example}[theorem]{Example}

\numberwithin{equation}{section}

\usepackage{amssymb}
\usepackage{amsfonts}
\usepackage{psfrag}
\usepackage{url}
\usepackage{mathtools}

\usepackage[makeroom]{cancel}

\usepackage{stmaryrd}
\DeclareMathOperator*{\bigtimes}{\vartimes}

\newcommand{\K}{\mathbb{K}}
\newcommand{\N}{\mathbb{N}}
\newcommand{\R}{\mathbb{R}}
\newcommand{\T}{\mathbb{T}}
\newcommand{\Z}{\mathbb{Z}}
\newcommand{\Q}{\mathbb{Q}}
\newcommand{\C}{\mathbb{C}}

\newcommand{\cLL}{\mathcal{L}}

\begin{document}

\title{On the classification of\\finite-dimensional linear flows}

\author{Arno Berger and Anthony Wynne\\Mathematical and Statistical
Sciences\\University of Alberta\\Edmonton, Alberta, {\sc Canada}}

\maketitle
\begin{abstract}
\noindent
New elementary, self-contained proofs are presented for the
topological and the smooth classification theorems of linear flows on
finite-dimensional normed spaces. The arguments, and the examples that
accompany them, highlight the fundamental roles of linearity and
smoothness more clearly than does the existing literature.
\end{abstract}
\hspace*{6.6mm}{\small {\bf Keywords.} Linear flow, flow equivalence,
  orbit equivalence, (uniform) core, Kronecker flow.}

\noindent
\hspace*{6.6mm}{\small {\bf MSC2010.} 34A30, 34C41, 37C15.}

\medskip

\section{Introduction}

Let $X$ be a finite-dimensional normed space over $\R$ and $\varphi$ a 
flow on $X$, i.e., $\varphi: \R \times X \to X$ is continuous, with $\varphi \bigl( t , \varphi(s ,x)\bigr) = \varphi(t+s
,x)$ and $\varphi(0,x) = x$ for all $t,s \in \R$ and $x\in X$. A fundamental
problem throughout dynamics is to decide precisely which flows are, in
some sense, {\em essentially the same}. Formally, call two
smooth flows $\varphi, \psi$ on $X, Y$ respectively
$C^{\ell}$-{\bf orbit equivalent}, with $\ell \in \{0,1,\ldots, \infty\}$, if
there exists a $C^{\ell}$-diffeomorphism (or homeomorphism, in case
$\ell =0$) $h:X\to Y$ and a function $\tau : \R \times X \to \R$, with $\tau (\,
\cdot \, , x)$ strictly increasing for each $x\in X$, such that
\begin{equation}\label{eq1}
h\bigl( \varphi(t,x) \bigr) = \psi \bigl( \tau (t,x), h(x)\bigr) \quad
\forall (t,x) \in \R\times X \, .
\end{equation}
If $\tau$ in (\ref{eq1}) can be
chosen to be independent of $x$, and thus simply $\tau(t,x) =
\alpha t$ with some $\alpha \in \R^+$, then $\varphi, \psi$ are $C^{\ell}$-{\bf flow equivalent}; they are {\bf linearly} (orbit or flow)
{\bf equivalent} if $h(x) = H x$ with some linear isomorphism $H :X
\to Y$. Notice that these definitions are tailor-made for the present
article and differ somewhat from terminology in the
literature which, however, is itself not completely unified. Usage
herein of terminology pertaining to the equivalence of flows is
informed by the magisterial text
\cite{KH}, as well as by \cite{Amann, Irwin}. Widely used alternative
terms are ({\bf topologically}) {\bf conjugate} (for {\em flow equivalent}, often understood to include the additional requirement 
that $\alpha = 1$) and ({\bf topologically}) {\bf equivalent} (for {\em orbit
equivalent\/}); see \cite{ArrP, ACK1, CK, DSS, He, KS, KRS,
Meiss, PdM, Perko, R, Wiggins}.

Clearly, linear equivalence implies $C^{\ell}$-equivalence for any
$\ell$, which in turn implies $C^0$-equivalence; also, flow equivalence
implies orbit equivalence. Simple examples show that none of these implications can
be reversed in general, not even when $\mbox{\rm dim}\, X =1$, though the
latter case is somewhat special in that $C^0$-orbit equivalence does imply
$C^0$-flow equivalence. In any case, however, it turns out that all
such examples must involve non-linear flows. In fact, 
the main theme of this article is that for linear flows all the {\em infinitely many\/}
different notions of equivalence do coalesce, rather amazingly, into just {\em two\/} 
notions; see Figures \ref{fig1} and \ref{fig2}.

\begin{figure}[ht] 
\psfrag{tl}[]{\large {\sf arbitrary flows}}
\psfrag{t11}[]{linearly flow equivalent}
\psfrag{t1m}[]{\large $\Longrightarrow$}
\psfrag{t2m}[]{($\! \Longleftarrow \, \mbox{\rm dim}\, X \!\! =\!  1$)}
\psfrag{t1v}[]{\large $\Rightarrow$}
\psfrag{t12}[]{linearly orbit equivalent}
\psfrag{t21}[]{$C^{\infty}$-flow equivalent}
\psfrag{t22}[]{$C^{\infty}$-orbit equivalent}
\psfrag{t31}[]{$C^{\ell}$-flow equivalent}
\psfrag{t32}[]{$C^{\ell}$-orbit equivalent}
\psfrag{t41}[]{$C^1$-flow equivalent}
\psfrag{t42}[]{$C^1$-orbit equivalent}
\psfrag{t51}[]{$C^0$-flow equivalent}
\psfrag{t52}[]{$C^0$-orbit equivalent}
%
%
\begin{center}
\includegraphics{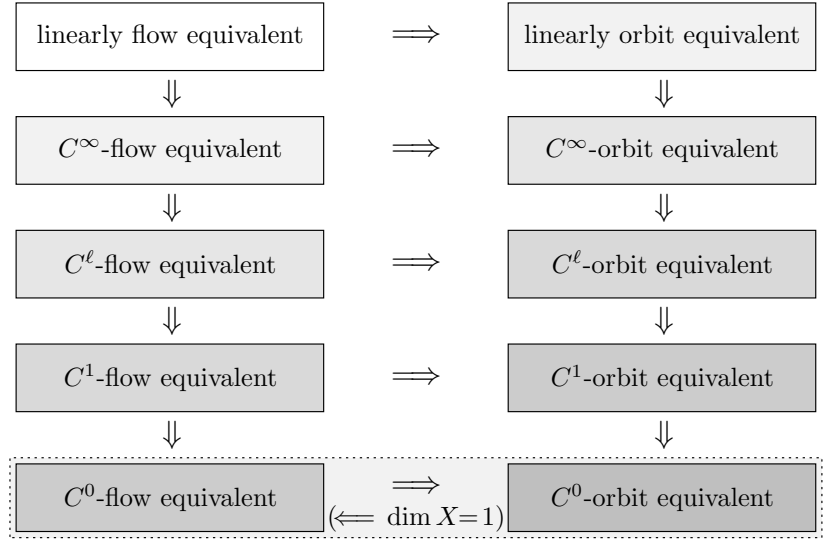}
\end{center}
\caption{Notions of equivalence for flows on normed spaces over $\R$; no conceivable
  implication not shown in the diagram is valid in general.}\label{fig1}
\end{figure}

A flow $\varphi$ on $X$ is {\bf linear} if each homeomorphism (or
time-$t$-map) $\varphi (t, \, \cdot \, ):X\to X$ is linear, or equivalently if $\varphi (t, \, \cdot \, ) = e^{tA^{\varphi}}$ for every $t\in \R$, with a
(unique) linear operator $A^{\varphi}:X\to X$, referred to as the {\bf
  generator} of $\varphi$. Thus a linear flow simply encodes the totality
of all solutions of the linear differential equation $\dot
x = A^{\varphi} x$ on $X$, in that $\varphi(\, \cdot \, , x_0)$ is the
unique solution of that equation satisfying $x(0)=x_0$.
To emphasize the fundamental role played by linearity in all that
follows, linear flows are henceforth denoted exclusively by upper-case
Greek letters $\Phi, \Psi$ etc. 

For linear flows, the weakest form of equivalence,
$C^0$-orbit equivalence, implies the seemingly much stronger $C^0$-flow
equivalence, and both properties can be characterized neatly in terms of linear
algebra. To state the following {\bf topological
  classification theorem}, the main topic of this article, recall
that  every linear flow $\Phi$ on $X$ uniquely
determines a $\Phi$-invariant decomposition $X = X^{\Phi}_{\sf S}
\oplus X^{\Phi}_{\sf C} \oplus
X^{\Phi}_{\sf U}$ into stable, central, and unstable subspaces,
with a corresponding unique decomposition $\Phi \simeq \Phi_{\sf S}
\times \Phi_{\sf C} \times \Phi_{\sf U}$; see Section \ref{sec4} for details.

\begin{theorem}\label{thmB}
Let $\Phi, \Psi$ be linear flows on $X, Y$,
respectively. Then each of the following four statements implies the other three:
\begin{enumerate}
\item $\Phi, \Psi$ are $C^0$-orbit equivalent;
\item $\Phi, \Psi$ are $C^0$-flow equivalent;
\item $\Phi_{\sf S} \times \Phi_{\sf U}, \Psi_{\sf S} \times \Psi_{\sf
  U}$ are $C^0$-flow equivalent, and $\Phi_{\sf C}, \Psi_{\sf C}$ are linearly flow equivalent;
\item $\mbox{\rm dim}\, X^{\Phi}_{\sf S} = \mbox{\rm dim}\,
  Y^{\Psi}_{\sf S}$, $\mbox{\rm dim}\, X^{\Phi}_{\sf U} = \mbox{\rm
    dim}\, Y^{\Psi}_{\sf U}$, and $A^{\Phi_{\sf C}}, \alpha
  A^{\Psi_{\sf C}}$ are similar for some $\alpha\in \R^+$, i.e., $H
  A^{\Phi_{\sf C}}  = \alpha A^{\Psi_{\sf C}}H$
  with some linear isomorphism $H :X^{\Phi}_{\sf C}\to Y^{\Psi}_{\sf C}$.
\end{enumerate}
\end{theorem}

\noindent
In the presence of smoothness, i.e., for
$C^{\ell}$-equivalence with $\ell \ge 1$, the counterpart of Theorem \ref{thmB}
is the following {\bf smooth classification theorem} which shows that in fact the weakest notion ($C^1$-orbit equivalence) implies
the strongest (linear flow equivalence).

\begin{theorem}\label{thmA}
Let $\Phi, \Psi$ be linear flows. Then each of the following four statements implies the other three:
\begin{enumerate}
\item $\Phi, \Psi$ are $C^1$-orbit equivalent;
\item $\Phi, \Psi$ are $C^1$-flow equivalent;
\item $\Phi, \Psi$ are linearly flow equivalent;
\item $A^{\Phi}, \alpha  A^{\Psi}$ are similar for some $\alpha \in \R^+$.
\end{enumerate}
\end{theorem}

\noindent
Taken together, Theorems \ref{thmB} and \ref{thmA} reveal a
remarkable rigidity of finite-dimensional real linear flows: For such flows, there really are only two different notions of
equivalence, informally referred to as {\bf topological} and {\bf
  smooth} equivalence; for central or
one-dimensional flows, even these two notions
coalesce. Moreover, the theorems characterize these equivalences in terms
of elementary properties of the associated generators.

\begin{figure}[ht] 
\psfrag{tl1}[]{{\bf ``smooth''}}
\psfrag{tl1a}[]{(Theorem \ref{thmA})}
\psfrag{tl2}[]{{\bf ``topological''}}
\psfrag{tl3}[]{(Theorem \ref{thmB})}
\psfrag{t11}[]{linearly flow equivalent}
\psfrag{t1m}[]{\large $\Longleftrightarrow$}
\psfrag{t1v}[]{\large $\Leftrightarrow$}
\psfrag{t2v}[]{\Large $\Rightarrow$}
\psfrag{t3v}[l]{($\Uparrow$ $X_{\sf C}^{\Phi}=X$ or
  $\mbox{\rm dim} \, X \! =\! 1$)}
\psfrag{t12}[]{linearly orbit equivalent}
\psfrag{t21}[]{$C^{\infty}$-flow equivalent}
\psfrag{t22}[]{$C^{\infty}$-orbit equivalent}
\psfrag{t31}[]{$C^{\ell}$-flow equivalent}
\psfrag{t32}[]{$C^{\ell}$-orbit equivalent}
\psfrag{t41}[]{$C^1$-flow equivalent}
\psfrag{t42}[]{$C^1$-orbit equivalent}
\psfrag{t51}[]{$C^0$-flow equivalent}
\psfrag{t52}[]{$C^0$-orbit equivalent}
%
%
\begin{center}
\includegraphics{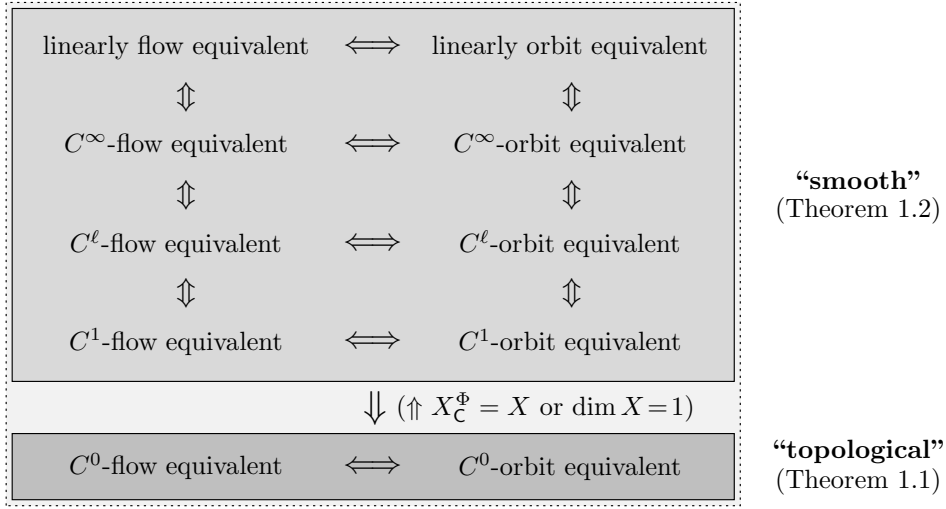}
\end{center}
\caption{By Theorems \ref{thmB} and \ref{thmA}, for real linear flows all notions of equivalence coalesce
  into only two different notions, or even just one if
  $X_{\sf C}^{\Phi} = X $ or $\mbox{\rm dim}\, X =1$.}\label{fig2}
\end{figure}

As far as the authors have been able to ascertain, variants of Theorem \ref{thmB} were first proved,
independently, in \cite{Kuiper} and \cite{Ladis}, though of course for
hyperbolic linear flows the result dates back much further (see, e.g.,
\cite{Amann, Arnold, Irwin}; a detailed discussion of the pertinent
literature is deferred to Section \ref{sec4} when all relevant
technical terms will have been introduced). Given
the clear, definitive nature of Theorem \ref{thmB} and the
fundamental importance of linear differential equations
throughout science, it is striking that the details of \cite{Kuiper, Ladis} have not been disseminated more widely in over four
decades \cite{He}. A main objective of this article, then, is to provide an
elementary, self-contained proof of Theorem \ref{thmB} that hopefully will find its way into future
textbooks on differential equations. In the process, several
inaccuracies and gaps in the classical arguments are addressed as well.
As presented here, Theorem \ref{thmA} is a rather
straightforward consequence of Theorem
\ref{thmB}. Although the result itself seems to have long been part of
dynamical systems folklore \cite{Amann, Arnold, ACK1, CK, Meiss, Wiggins}, the authors are not aware of any reference that
would establish it in its full strength, that is, without imposing
additional (and, as it turns out, unnecessary) assumptions on $\tau$.

To appreciate the difference between Theorem \ref{thmB} and
\ref{thmA}, first note that for $\mbox{\rm dim}\, X = 1$, trivially {\em all\/} notions of
equivalence coincide, yielding exactly three equivalence
classes of real linear flows, represented by $\Phi(t,x) = e^{t a }x$
with $a \in \{-1,0,1\}$. However, already for $\mbox{\rm dim}\, X = 2$ the huge
difference between the theorems becomes apparent: On the one hand, by Theorem \ref{thmB},
there are exactly eight topological ($C^0$) equivalence classes, represented by $\Phi(t,x) = e^{tA}x$ with $A$ being
precisely one of
\begin{equation}\label{eq3}
\left[\begin{array}{cr}
0 & -1 \\1 & 0 
\end{array}
\right] , \enspace
\left[\begin{array}{cc}
0 & 0 \\0 & 0 
\end{array}
\right]  , \enspace
\left[\begin{array}{cc}
0 & 1 \\0 & 0 
\end{array}
\right] , \enspace
\pm \left[\begin{array}{cc}
0 & 0 \\0 & 1 
\end{array}
\right]  ,  \enspace
\left[\begin{array}{cr}
1 & 0 \\0 & -1 
\end{array}
\right] , \enspace
\pm \left[\begin{array}{cc}
1 & 0 \\0 & 1
\end{array}
\right] .
\end{equation}
By Theorem \ref{thmA}, on the other hand, all smooth ($C^1$)
equivalence classes are represented uniquely by the five
left-most matrices in (\ref{eq3}), together with $\pm \left[\begin{array}{cc}
1 & 1 \\0 & 1 \end{array} \right] $ and the five infinite families
$$
\left[\begin{array}{rc}
-1 & 0 \\0 &  a
\end{array}
\right] , \enspace
\pm \left[\begin{array}{cc}
1 & 0 \\0 &  a
\end{array}
\right] ,\enspace
\pm \left[\begin{array}{cr}
1 & -  a \\  a  & 1 
\end{array}
\right] 
\quad
( a \in \R^+ ) \, .
$$

This article is organized as follows. Section \ref{secorb} briefly
reviews the notions of equivalence for flows, as well as a few basic
dynamical concepts. It then introduces {\bf cores}, a new family of invariant objects. Although these objects may
well be useful in more general contexts, their properties
are established here only as far as needed for the subsequent
analysis of flows on finite-dimensional normed spaces. Section
\ref{sec2} specifically identifies cores for real linear flows, and
shows how they can be iterated in a natural way. As it turns out, the
proof of Theorem \ref{thmB} also hinges on a careful analysis of bounded
linear flows, and the latter is carried out in Section \ref{sec3}. With all required
tools finally assembled, proofs of Theorems \ref{thmB} and \ref{thmA}
are presented in Section \ref{sec4}, together with several comments
on related results in the literature that prompted this work.
While, for reasons that will become apparent in Section \ref{secom},
the article focuses mostly on {\em real\/} spaces, the concluding
section shows how the results carry over to {\em complex\/} spaces
in a natural way. To keep the exposition focussed
squarely on the main arguments, several elementary (and, presumably, known) facts of an
auxiliary nature are stated without proof; for details regarding these facts, as well as others
that are mentioned in passing but for which the authors were unable to
identify a precise reference, the interested reader is referred to the
accompanying document \cite{Wsupp}.

Throughout, the familiar symbols $\N, \N_0,\Z, \Q, \R^+, \R$, and $\C$
denote the sets of all positive integers, non-negative integers,
integers, rational, positive real, real, and complex numbers,
respectively; for convenience, $c+\Omega = \{c+\omega : \omega \in
\Omega\}$ and $c\Omega = \{c\omega : \omega \in \Omega\}$ for any
$c\in \C$, $\Omega \subset \C$. Occasionally, for the purpose of
coordinate-dependent arguments, elements of $\Z^m$, $\R^m$, or
$\C^m$, with $m\in \N\setminus\{1\}$, are interpreted as $m\times
1$-column vectors.

\section{Orbit equivalence}\label{secorb}

Let $X, Y$ be two finite-dimensional normed spaces over $\R$, and let
$\varphi, \psi$, respectively, be flows on them; unless specified
further, $\| \cdot  \|$ denotes
any norm on either space. Given two functions
$h:X\to Y$ and $\tau : \R \times X \to \R$, say that $\varphi$ is
$(h,\tau)$-{\bf related} to $\psi$ if $h$ is a homeomorphism, $\tau
(\, \cdot \, , x)$ is strictly increasing for each $x\in X$, and
\begin{equation*}
h \bigl( \varphi (t, x)\bigr) = \psi \bigl( \tau (t,x) ,  h(x)\bigr)
\quad \forall (t,x) \in \R \times X \, . \tag{\ref{eq1}}
\end{equation*}
In what follows, for each $t\in \R$ the homeomorphism $\varphi (t, \,
\cdot \, ):X\to X$ usually is denoted $\varphi_t$, and for each $x\in
\R$ the strictly increasing map $\tau (\, \cdot \, , x):\R \to \R$ is denoted
$\tau_x$. With this, (\ref{eq1}) succinctly reads
$$
h\circ \varphi_t (x) = \psi_{\tau_x (t)} \circ h (x) \quad \forall
(t,x)\in \R \times X \, .
$$
Thus $\varphi$ is $(h,\tau)$-related to $\psi$ precisely if
the homeomorphism $h$ maps each $\varphi$-orbit into a $\psi$-orbit in
an orientation-preserving way. Note that no assumption whatsoever is
made regarding the $x$-dependence of $\tau_x$. Still, utilizing the
flow axioms of $\varphi, \psi$, and the continuity of $h,h^{-1}$,
it is readily deduced from (\ref{eq1}) that the function $\tau$ can be assumed to have several
additional properties; cf.\ \cite{ArrP, ACK1, PdM}. For
convenience, these properties are understood to be part of what it means for $\varphi$ to be
$(h,\tau)$-related to $\psi$ throughout the remainder of this article.

\begin{prop}\label{prop22}
Let $\varphi, \psi$ be flows on $X, Y$, respectively, and assume
that $\varphi$ is $(h,\tau)$-related to $\psi$. Then $\varphi$ is
$(h,\widetilde{\tau}\, )$-related to $\psi$ where $\widetilde{\tau}_x:\R\to \R$ is,
for every $x\in X$, an (increasing) continuous bijection with $\widetilde{\tau}_x (0)=0$.
\end{prop}

Recall from the Introduction that two flows $\varphi, \psi$ are ($C^0$-){\bf orbit equivalent} if $\varphi$ is $(h,\tau)$-related to $\psi$
for some $h$, $\tau$; they are {\bf flow equivalent} if, with the
appropriate constant $\alpha \in \R^+$, the function $\tau$ can be chosen
so that $\tau_x (t) = \alpha t$ for all $(t,x)\in \R\times X$. This terminology is justified.

\begin{prop}\label{prop23}
Orbit equivalence and flow equivalence are equivalence relations in
the class of all flows on finite-dimensional normed spaces.
\end{prop}

A simple, classical example of orbit equivalence, presented in essence
(though not always in name) by many textbooks, is as follows
\cite[Sec.3.1]{Perko}: Assume that two flows $\varphi,\psi$ on $X$ are
generated by the differential equations $\dot x = V(x)$,
$\dot x = W(x)$, respectively, with $C^{\infty}$-vector fields
$V,W$. If $V=wW$ for some (measurable and locally bounded) function $w:X \to \R^+$ then
$\varphi$ is $({\rm id}_X, \tau)$-related to $\psi$, with $\tau_x(t) =
\int_0^t w\bigl( \varphi_s (x)\bigr)\,  {\rm d}s$ for all $(t,x)\in
\R\times X$.

For every $x\in X$, let $T_x^{\varphi} = \inf \{t\in \R^+  : \varphi_t(x) =
x\}$, with the usual convention that $\inf \varnothing = +\infty$. Note that
whenever the set $\{t\in \R^+ : \varphi_t(x) = x\}$ is non-empty, it equals
either $\R^+$ or $\{n T_x^{\varphi} : n \in \N\}$. In the former case,
$T_x^{\varphi} = 0$, and $x$ is a {\bf fixed point} of $\varphi$. In
the latter case, $0< T_x^{\varphi} < +\infty$, and $x$ is $T$-{\bf
  periodic}, i.e., $\varphi_T(x) = x$ with $T\in \R^+$, precisely for
$T \in T_x^{\varphi} \N$; in particular, $T_x^{\varphi}$ is the
minimal $\varphi$-{\bf period} of $x$. Denote by $\mbox{\rm Fix}\,
\varphi$ and $\mbox{\rm Per}_T \varphi$ the sets of all fixed and
$T$-periodic points respectively, and let $\mbox{\rm Per}\, \varphi =
\bigcup_{T\in \R^+ } \mbox{\rm Per}_T \varphi$. Note that $T_{\, {\bf
    \cdot} }^{\varphi}$ is lower semi-continuous, with $T_x^{\varphi}=0$ and
$T_x^{\varphi} < +\infty$ if and only if $x\in \mbox{\rm
  Fix}\,\varphi$ and $x\in \mbox{\rm Per}\, \varphi$, respectively.

The $\varphi$-{\bf orbit} of any $x\in X$ is $\varphi (\R,
x) = \{\varphi_t (x) : t \in \R\}$. Recall that $C\subset X$ is $\varphi$-{\bf invariant} if
$\varphi_t (C) = C$ for all $t\in \R$, or equivalently if $\varphi(\R, x) \subset
C$ for every $x\in C$. Clearly, $\mbox{\rm Fix}\,
\varphi$ and $\mbox{\rm Per}\, \varphi$ are $\varphi$-invariant, and
so is $\mbox{\rm Per}_T \varphi$ for every $T\in \R^+$. Another example of a
$\varphi$-invariant set is
$\mbox{\rm Bnd}\, \varphi := \left\{  x \in X : \sup\nolimits_{t\in \R}
\| \varphi_t (x)\| <+\infty \right\}$, 
which simply is the union of all bounded $\varphi$-orbits. Plainly,
$$
\mbox{\rm Fix}\, \varphi \subset \mbox{\rm Per}_T \varphi \subset
\mbox{\rm Per}\, \varphi \subset \mbox{\rm Bnd}\, \varphi \quad
\forall T\in \R^+  \, .
$$

\begin{prop}\label{prop24}
Let $\varphi, \psi$ be flows on $X, Y$, respectively, and assume
that $\varphi$ is $(h,\tau)$-related to $\psi$. Then $C \subset X$ is
$\varphi$-invariant if and only if $h(C)\subset Y$ is
$\psi$-invariant. Moreover,
$$
h(\mbox{\rm Fix}\, \varphi) = \mbox{\rm Fix}\, \psi \, , \quad
h(\mbox{\rm Per}\, \varphi) = \mbox{\rm Per}\, \psi \, , \quad
h(\mbox{\rm Bnd}\, \varphi) = \mbox{\rm Bnd}\, \psi \, .
$$
\end{prop}

\noindent
A simple observation with far-reaching consequences for the subsequent
analysis is that, under the assumptions of Proposition \ref{prop24}, and for
any $T\in \R^+ $, the $\psi$-invariant set $h(\mbox{\rm Per}_T\varphi)$ may
not be contained in $\mbox{\rm Per}_S \psi$ for any $S\in \R^+$. A
numerical invariant that can be used to address this ``scrambling'' of $\mbox{\rm
  Per}\, \varphi \setminus \mbox{\rm Fix}\, \varphi$ by $h$ is the
$\varphi$-{\bf height} of $x$, defined as
$$
\langle x \rangle^{\varphi} = \limsup\nolimits_{\widetilde{x} \in \mbox{{\scriptsize \rm Per}}\,
  \varphi, \widetilde{x} \to x}
\frac{T_{\widetilde{x}}^{\varphi}}{T_x^{\varphi}} \quad \forall x \in \mbox{\rm
  Per}\, \varphi \setminus \mbox{\rm Fix}\, \varphi \, .
$$
Note that $\langle x \rangle^{\varphi} $ equals either a positive
integer or $+ \infty$, and with $\langle x \rangle^{\varphi}:= +\infty$ for
every $x\in \mbox{\rm Fix}\, \varphi$, the function  $ \langle
\, \cdot \,  \rangle^{\varphi}$ is upper semi-continuous on $\mbox{\rm Per}\,
\varphi$; cf.\ \cite[Def.5]{Ladis}. As is readily confirmed, minimal
periods and heights are well-behaved under orbit equivalence.

\begin{prop}\label{prop25}
Let $\varphi, \psi$ be flows, and assume
that $\varphi$ is $(h,\tau)$-related to $\psi$. Then, for every $x\in \mbox{\rm
  Per}\, \varphi$:
\begin{enumerate}
\item $T_{h(x)}^{\psi} =\tau_x(T_x^{\varphi})$;
\item $\langle h(x)\rangle^{\psi} = \langle x\rangle^{\varphi}$.
\end{enumerate}
\end{prop}

\noindent
The subsequent analysis relies heavily on the properties of certain
invariant sets associated with the flows under
consideration. Specifically, given a flow $\varphi$ on $X$ and any two
points $x^-,x^+ \in X$, define the $(x^-,x^+)$-{\bf core} $C_{x^-,x^+}
(\varphi, X)$ as
\begin{align*}
C_{x^-,x^+} (\varphi, X) = \bigl\{x \in X : \: \: & \mbox{\rm There exist
  sequences $(t_n^{\pm})$ and $(x_n^{\pm})$ with $t_n^{\pm}\to  \pm \infty$}\\
& \mbox{\rm and $x_n^{\pm} \to x$ such that $\varphi_{t_n^{\pm} }( x_n^{\pm})\to x^{\pm}$} \bigr\} \, ;
\end{align*}
here and throughout, expressions containing $\pm$ (or $\mp$) are to be
read as two separate expressions containing only the upper and only
the lower symbols, respectively. Note that $C_{x^-,x^+} (\varphi, X)$ is $\varphi$-invariant and closed, possibly empty. For
linear flows, the $(0,0)$-core $C_{0,0}(\varphi, X)$, henceforth
simply denoted $C_0 (\varphi, X)$, is naturally of particular relevance, and so is
the {\bf core} 
$$
C(\varphi, X) := \bigcup\nolimits_{x^-,x^+ \in X} C_{x^-,x^+ } (\varphi, X)
\enspace \supset \enspace C_0 (\varphi, X) \, .
$$
Clearly, $C(\varphi, X)$ also is $\varphi$-invariant and contains
$\mbox{\rm Bnd}\, \varphi$ as well as all
non-wandering points of $\varphi$. For instance, if $X$ is
one-dimensional then $C(\varphi, X)$ simply is the convex hull of
$\mbox{\rm Fix}\, \varphi$, whereas $C_0(\varphi , X) =\{0\}\cap
\mbox{\rm Fix}\, \varphi$. Most importantly, $C(\varphi, X)$ and
$C_0(\varphi, X)$ both are well-behaved under orbit equivalence.

\begin{lem}\label{prop27}
Let $\varphi, \psi$ be flows on $X, Y$, respectively, and assume
that $\varphi$ is $(h,\tau)$-related to $\psi$. Then 
\begin{equation}\label{eq2l8}
h \bigl( C_{x^-,x^+ } (\varphi, X) \bigr) = C_{h(x^-), h(x^+)} (\psi ,
Y) \quad \forall x^-,x^+  \in X \, .
\end{equation}
Thus $h \bigl( C(\varphi, X)\bigr) = C(\psi, Y)$, and if
$h(0)=0$ then also $h \bigl( C_{0}(\varphi, X)\bigr) = C_{0}(\psi, Y)$.
\end{lem}

\noindent
The proof of Lemma \ref{prop27} is facilitated by an elementary observation
\cite{Wsupp}.

\begin{prop}\label{prop28a}
Let $\varphi$ be a flow on $X$, and $x\in X$. Then the following
are equivalent:
\begin{enumerate}
\item For every $\varepsilon > 0$ there exists an $\widetilde{x}\in X$ such that
  $\|\varphi_t ( \widetilde{x}  ) - x\| < \varepsilon$ for all $0\le t \le
  \varepsilon^{-1}$;
\item $x\in \mbox{\rm Fix}\, \varphi$.
\end{enumerate}
\end{prop}

\begin{proof}[Proof of Lemma \ref{prop27}]
It suffices to prove (\ref{eq2l8}), as all other assertions directly
follow from it. To do this, given
$x^-,x^+\in X$, denote $C_{x^-,x^+}(\varphi, X)$ and $C_{h(x^-),
  h(x^+)}(\psi, Y)$ simply by $C$ and $D$, respectively. From
reversing the roles of $(\varphi, X)$ and $(\psi, Y)$, as well as
$h$ and $h^{-1}$, it is clear that all that needs to be shown is that
$h(C)\subset D$.

Pick any $x\in C$, together with sequences $(t_n^{\pm})$ and
$(x_n^{\pm})$ with $t_n^{\pm} \to \pm \infty$ and $x_n^{\pm} \to x$ such that $\varphi_{t_n^{\pm}}
( x_n^{\pm})\to x^{\pm}$; assume w.l.o.g.\ that $t_n^- < 0 <
t_n^+$ for all $n$. Letting $s_n^{\pm} = \tau_{x_n^{\pm}} (t_n^{\pm})$, note
that $s_n^- < 0 < s_n^+$, and
\begin{equation}\label{eq2l8a}
h(x_n^{\pm}) \to h(x) \, , \quad \psi_{s_n^{\pm}} \bigl( h(x_n^{\pm})\bigr) \to
h(x^{\pm}) \, .
\end{equation}
By considering appropriate subsequences, assume that $s_n^-\to s^- \in [-\infty, 0]$ and $s_n^+ \to
s^+\in [0,+\infty]$. Note that (\ref{eq2l8a}) immediately yields
$h(x)\in D$ if $\{s^-,s^+\} = \{ -\infty, +\infty\}$, so assume for
instance that $s^+ < +\infty$. (The case
of $s^- > - \infty$ is completely analogous.) Then $\psi_{s^+} \bigl(
h(x)\bigr) = h(x^+)$ by (\ref{eq2l8a}), and, as will be shown below, in fact
\begin{equation}\label{eq2l8b}
h(x) \in \mbox{\rm Per}\, \psi \, .
\end{equation}
Assuming (\ref{eq2l8b}), let $T\in \R^+$ be any $\psi$-period of $h(x)$, and
$y_n^+ = h(x), r_n^+ = s^+ + nT$ for every $n\in \N$. With this
clearly $r_n^+ \to +\infty$, and $\psi_{r_n^+}\bigl(
h(x)\bigr) = h(x^+)$ for all $n$. Thus to complete the proof it only remains to verify (\ref{eq2l8b}).

Assume first that $s^+ = 0$, and hence $x= x^+$. For each $n\in \N$,
define a non-negative continuous function $f_n : \R \to \R$ as $f_n
(s) = \|\varphi_{s t_n^+} (x_n^+) - x\|$, and note that $f_n(0) = \|x_n^+ -
x\|\to 0$, but also $f_n (1) = \|\varphi_{t_n^+} (x_n^+) - x\|\to 0$. In
fact, more is true:
\begin{equation}\label{eq2l8c}
\lim\nolimits_{n\to \infty} f_n (s) = 0 \quad \mbox{\rm uniformly on
}[0,1]\, .
\end{equation}
To prove (\ref{eq2l8c}), suppose by way of contradiction that
\begin{equation}\label{eq2l8d}
\varepsilon_0  \le f_{n_k} (s_k) = \|\varphi_{s_k t_{n_k}^+} (x_{n_k}^+) - x\|
\quad \forall k\in \N \, , 
\end{equation}
with appropriate $\varepsilon_0 > 0$, $s_k \in [0,1]$, and integers $n_k
\ge k$. Since $0\le r_k := \tau_{x_{n_k}^+} (s_k t_{n_k}^+)\le
\tau_{x_{n_k}^+} (t_{n_k}^+) = s_{n_k}^+\to 0$, clearly $h \bigl( \varphi_{s_k t_{n_k}^+} (x_{n_k}^+) \bigr) = \psi_{r_k}
\bigl( h(x_{n_k}^+)\bigr) \: \to \: h(x) $, which, together with (\ref{eq2l8d}), contradicts the continuity of
$h^{-1}$ at $h(x)$, and hence establishes
(\ref{eq2l8c}). Deduce that, given any $\varepsilon >
0$, there exists an $N\in \N$ with $\max_{s\in [0,1]}f_N (s)<
\varepsilon$ as well as $t_N^+ > \varepsilon^{-1}$. But then
$\|\varphi_t (x_N^+) - x\|< \varepsilon$ for all $0\le t \le
\varepsilon^{-1}$, and Proposition \ref{prop28a} yields $x\in
\mbox{\rm Fix}\, \varphi$. By Proposition \ref{prop24}, $h(x)\in
\mbox{\rm Fix}\, \psi$, which proves (\ref{eq2l8b}) when $s^+ = 0$.

Finally, assume that $s^+ \in \R^+$, and let $t^+ = \tau_x^{-1}
(s^+)>0$. Then $h \bigl( \varphi_{t^+} (x) \bigr) = \psi_{s^+} \bigl(
h(x) \bigr) = h(x^+)$, and consequently $\varphi_{t^+} (x) = x^+$, as
well as
$$
\psi_{\tau_{x_n^+}(t_n^+ - t^+)}\bigl( h(x_n^+)\bigr) = h \circ
\varphi_{-t^+} \bigl( \varphi_{t_n^+} (x_n^+)\bigr) \: \to \: h \circ
\varphi_{-t^+} (x^+) = h(x) \, .
$$ 
Since $0 \le \tau_{x_n^+} (t_n^+ - t^+) \le s_n^+$ for all 
large $n$, assume w.l.o.g.\ that $\tau_{x_n^+} (t_n^+ - t^+)\to r \in
[0,s^+]$, and hence $\psi_r \bigl( h(x)\bigr) = h(x)$. On the one
hand, if $r\in \R^+$ then clearly $h(x)\in \mbox{\rm Per}\, \psi$. On the
other hand, if $r = 0$ then (\ref{eq2l8c}) holds with $f_n (s) =
\|\varphi_{s (t_n^+ - t^+)} (x_n^+) - x\|$, and the same argument as
above shows that $x\in \mbox{\rm Fix}\, \varphi$. Thus (\ref{eq2l8b})
also holds when $s^+ \in \R^+$.
\end{proof}

A crucial step in the subsequent analysis is the decomposition of 
flows into simpler, well-understood parts. To prepare for this,
recall that two flows $\varphi, \psi$ on $X, Y$, respectively, together induce the product flow
$\varphi \times \psi$ on $X\times Y$, by letting $(\varphi\times
\psi)_t = \varphi_t\times \psi_t $ for all $t\in \R$. Endow $X\times
Y$ with any norm. It is readily seen that
$$
C_{(x^-,y^-), (x^+,y^+)} (\varphi \times \psi, X\times Y) \subset
C_{x^-,x^+} (\varphi, X) \times C_{y^-, y^+} (\psi, Y) \quad \forall
x^-,x^+ \in X, y^-,y^+  \in Y \, ,
$$
and therefore also
\begin{equation}\label{eq23}
C(\varphi \times \psi, X\times Y) \subset C(\varphi , X) \times
C(\psi, Y) \, ;
\end{equation}
the same inclusion is valid with $C_0$ instead of $C$. 
Quite trivially, equality holds in (\ref{eq23}) and its analogue for
$C_0$ if one factor is at most one-dimensional.
As the following example shows, however, equality does not hold in
general if $\min \{\mbox{\rm dim} \, X , \mbox{\rm dim} \, Y\}\ge 2$.

\begin{example}\label{ex25}
Let $X = \R^2$, and write $X^+ = (\R^+)^2$ and ${\bf 1} =
\left[ \! \begin{array}{c} 1\\ 1\end{array} \! \right]$ for
convenience. Consider the flow $\varphi$ on $X$ generated by $\dot x = V(x)$,
with the $C^{\infty}$-vector field
$$
V(x) = \left\{
\begin{array}{ll}
{\displaystyle \frac1{s}} \left[ \begin{array}{c}   f(s)  x_1 \log x_1 - s f(s) x_1 \log x_2 \\
    s f(s) x_2 \log x_1 +  f(s) x_2 \log x_2 \end{array}\right] &
\mbox{\rm if } x \in X^+ \setminus \{{\bf 1}\} \, , \\[5mm]
0 & \mbox{\rm otherwise}\, , 
\end{array}
\right.
$$
where $s=s(x) = \sqrt{(\log x_1)^2 + (\log x_2)^2}$, and $f(s) =
e^{-s-1/s}$ for all $s\in \R^+$. Clearly, $(X\setminus X^+ )\cup \{{\bf
1}\}\subset \mbox{\rm Fix}\, \varphi$. Introducing (exponential) polar coordinates
$x_1 = e^{r \cos \theta}$, $ x_2 = e^{r\sin \theta}$
in $X^+$ transforms $\dot x = V(x)$ into
\begin{equation}\label{eqex1b}
\dot r = \dot \theta = f (r) \, . 
\end{equation}
Deduce from (\ref{eqex1b}) that $\lim_{t\to - \infty} r(t)=0$,
$\lim_{t\to +\infty} r(t) = +\infty$, and $r-\theta$ is
constant. Consequently, $\lim_{t\to -\infty} \varphi_t (x) = {\bf 1}$
for every $x\in X^+$, but also, given any $x\in X^+ \setminus \{{\bf
  1}\}$, there exists a sequence $(t_n^+)$ with $t_n^+ \to
+\infty$ such that $\theta (t_n^+) + \frac34 \pi \in  2\pi \Z$ for
all $n$, and hence $\lim_{n\to \infty} \varphi_{t_n^+} (x) = 0$. Thus
$x\in C_{{\bf 1}, 0} (\varphi, X)$ for every $x\in X^+\setminus \{{\bf
1}\}$, and $C(\varphi, X) = X$; see also Figure \ref{fig3}.

Next, note that $f$ is decreasing on $[1,+\infty[$, and hence any two solutions $(r,\theta)$, $(\widetilde{r},
\widetilde{\theta}\, )$ of (\ref{eqex1b}) with $r(0), \widetilde{r}(0)
\ge 1$ satisfy
\begin{equation}\label{eqex1c}
\big| r(t) - \widetilde{r} (t) \big| \le \big| r(0) - \widetilde{r}
(0) \big| \, , \quad
\big| \theta(t) - \widetilde{\theta} (t) \big| \le \big| r(0) - \widetilde{r}
(0) \big|  + \big| \theta (0) -
\widetilde{\theta }
(0) \big|\quad \forall t \ge 0 \, ;
\end{equation}
moreover, $\theta - \widetilde{\theta}$ is constant whenever $r(0) =
\widetilde{r} (0)$. Pick any $a \ge e^{1/\sqrt{2}}$, and consider
\begin{equation}\label{eqex1d}
u = \left[ \! \begin{array}{c} a \\ a^{-1} \end{array}\! \right], \quad
\widetilde{u} = \left[ \! \begin{array}{c} a^{-1} \\
    a \end{array} \! \right] .
\end{equation}
Then $r(t) = s \bigl( \varphi_t (u)
\bigr) = s \bigl( \varphi_t (\widetilde{u}) \bigr) = \widetilde{r}(t)\ge 1$
and $\theta (t) - \widetilde{\theta} (t) \in \pi + 2\pi \Z$ for all
$t\ge 0$. Also, let
$$
U = \left\{
x\in X^+ : x_2^{\sqrt{3}} > \max \{ x_1, x_1^3\}^{-1}
\right\}  = \Bigl\{ x\in X^+ \setminus \{{\bf 1}\} : \theta \in \: \left]
  -{\textstyle \frac13} \pi , {\textstyle \frac56} \pi \right[ + 2\pi \Z  \Bigr\} \, .
$$
For any $\varepsilon >0$ sufficiently small, it is clear
from (\ref{eqex1c}) that for every $t\ge 0$ at least one of the two
open sets $\varphi_t \bigl( B_{\varepsilon } (u) \bigr) $ and $\varphi_t
\bigl( B_{\varepsilon } (\widetilde{u}) \bigr)$ is entirely contained in $U$.
Note that $B_{\varepsilon } (u)  \times B_{\varepsilon } (\widetilde{u})$ is a neighbourhood of $(u,\widetilde{u})$ in $X\times
X$. Consequently, $\bigl( (\varphi\times \varphi)_{t_n^+} (x_n, \widetilde{x}_n)\bigr)$ is
unbounded whenever $t_n^+ \to + \infty$ and $(x_n,\widetilde{x}_n) \to (u,\widetilde{u})$. Thus, $(u,\widetilde{u}) \not \in C(\varphi\times
\varphi, X\times X)$, whereas clearly $(u,\widetilde{u}) \in C(\varphi
, X) \times C(\varphi, X)$, and so the inclusion
(\ref{eq23}) is strict in this example.
\end{example}

\begin{figure}[ht]
\psfrag{tx1}[]{$x_1$}
\psfrag{txx2}[]{$x_2$}
\psfrag{tv14}[]{$\frac14$}
\psfrag{tv1}[]{$1$}
\psfrag{th1}[]{$1$}
\psfrag{tbf1}[]{${\bf 1}$}
\psfrag{tx}[]{$u$}
\psfrag{tx2}[]{$u$}
\psfrag{txn2}[]{$u_n$}
\psfrag{tptx}[]{$\varphi_t (u) $}
\psfrag{tptxt}[]{$\varphi_t (\widetilde{u})$}
\psfrag{txt}[]{$\widetilde{u}$}
\psfrag{tu}[]{$U$}
\psfrag{ttit1a}[]{$u,\widetilde{u}\in C(\varphi, X)$}
\psfrag{ttit1b}[]{$(u,\widetilde{u})\not \in C(\varphi\times \varphi, X\times X)$}
\psfrag{ttit2a}[]{$u \not \in C^*(\varphi, X)$, $u\in C^*(\psi, X)$}
\psfrag{ttit2b}[]{$ \varphi, \psi$ orbit equivalent, with $h={\rm
    id}_X$}
\psfrag{tpsix}[]{$\psi_t (u_n) $}
\psfrag{tpsiz}[l]{$\psi_t (u) $}
%
%
\begin{center}
\includegraphics{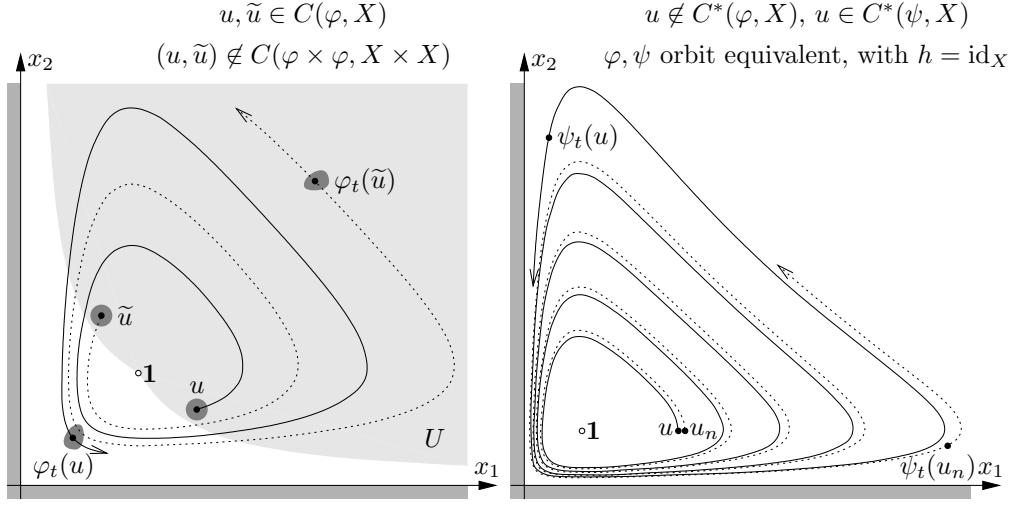}
\end{center}
\caption{In general, (non-uniform) cores are well-behaved under orbit
  equivalence but not under products (left; see Example \ref{ex25}),
  whereas for uniform cores the situation is the exact opposite
  (right; see Example \ref{ex255}).}\label{fig3}
\end{figure}

Good behaviour of certain invariant objects under products is indispensable for the analysis in later sections. Negative examples such as Example
\ref{ex25} therefore suggest that the cores $C(\varphi, X)$ and $C_0 (\varphi,
X)$ be supplanted, or at least supplemented with, similar objects
that are well-behaved under products. To this end, note that 
\begin{align*}
C (\varphi, X) = \bigl\{x \in X : \: \: & \mbox{\rm There exist
  sequences $(t_n^{\pm})$ and $(x_n^{\pm})$ with $t_n^{\pm}\to \pm \infty$}\\
& \mbox{\rm and $x_n^{\pm} \to x$ such that $\bigl( \varphi_{t_n^{\pm} }( x_n^{\pm}) \bigr)$ both are bounded} \bigr\} \, .
\end{align*}
In light of this, define the {\bf uniform core} $C^*(\varphi, X)$ as 
\begin{align*}
C^* (\varphi, X) = \bigl\{x \in X : \: \: & \mbox{\rm For {\em every\/} sequence $(t_n)$ with
  $|t_n|\to +\infty$ there exists a sequence} \\
& \mbox{\rm $(x_n)$ with $x_n \to x$ such that $\bigl( \varphi_{t_n} (x_n)\bigr)$ is bounded} \bigr\} \, ;
\end{align*}
analogously, define the {\bf uniform} $(0,0)$-{\bf core}
$C^*_{0} (\varphi, X)$ as
\begin{align*}
C^*_{0}(\varphi, X) = \bigl\{x \in X : \: \: & \mbox{\rm For {\em every\/} sequence $(t_n)$ with
  $|t_n|\to +\infty$ there exists a sequence} \\
& \mbox{\rm $(x_n)$ with $x_n \to x$ such that $\varphi _{t_n}(x_n)\to
  0$} \bigr\} \enspace \subset \enspace  C^* (\varphi, X) \, .
\end{align*}
Again, $C^*(\varphi, X)$ and $C^*_{0} (\varphi, X)$ are
$\varphi$-invariant, and they obviously are contained in their
non-uniform counterparts, i.e.,
\begin{equation}\label{eq21}
C^* (\varphi, X)\subset C(\varphi, X) \, , \quad C^*_{0} (\varphi,
X)\subset C_{0}(\varphi, X) \, .
\end{equation}
Moreover, $C^*(\varphi, X)\supset \mbox{\rm Bnd}\, \varphi$, just as
for (non-uniform) cores. For the flow
$\varphi$ in Example \ref{ex25}, it is clear that 
$C^*(\varphi , X) = \mbox{\rm Fix}\, \varphi \ne X = C(\varphi, X)$;
see also Example \ref{ex255} below.
Thus the left inclusion in (\ref{eq21}) is strict in general, and
so is the right inclusion.

As alluded to earlier, $C^* (\varphi, X)$ and $C^*_0 (\varphi, X)$ are
useful for the purpose of this article because, unlike their
non-uniform counterparts, they are well-behaved under products.

\begin{lem}\label{prop26}
Let $\varphi,\psi$ be flows on $X,Y$, respectively. Then
$$
C^* (\varphi\times \psi, X\times Y) = C^* (\varphi, X) \times C^* (\psi, Y) 
$$
as well as
$$
C^*_{0}(\varphi\times \psi, X\times Y) = C^*_{0} (\varphi, X) \times C^*_{0} (\psi, Y)
\, . 
$$
\end{lem}

\begin{proof}
The asserted equality for $C^*$ (respectively, $C^*_0$) is an immediate consequence of
the fact that $\bigl( (\varphi \times \psi)_{t_n} (x_n, y_n)\bigr)$ is
bounded (converges to $0$) if and only if $\bigl( \varphi_{t_n}
(x_n)\bigr)$ and $\bigl( \psi_{t_n} (y_n)\bigr)$ both are bounded (converge to $0$).
\end{proof}

Regarding the behaviour of uniform cores under equivalence, it is
readily checked that if $\varphi, \psi$ are {\em flow\/} equivalent
then $h\bigl( C^* (\varphi, X)\bigr) = C^* (\psi, Y)$; moreover,
$h\bigl( C_0^* (\varphi, X)\bigr) = C_0^* (\psi, Y)$ if
$h(0)=0$. These equalities may fail under mere {\em orbit\/}
equivalence, however, so the analogue of Lemma \ref{prop27} for uniform
cores does not hold. The following example demonstrates this.

\begin{example}\label{ex255}
With the identical objects as in Example \ref{ex25}, first deduce from
(\ref{eqex1c}) that, given any $x\in X^+ \setminus \{{\bf 1}\}$ and
sufficiently small $\varepsilon  > 0$, one may chose $(t_n)$ with $t_n \to
+\infty$ such that $\varphi_{t_n} \bigl( B_{\varepsilon } (x)\bigr)\subset
U$ for all $n$. But then clearly $\bigl( \varphi_{t_n} (x_n)\bigr)$ is
unbounded whenever $x_n \to x$, and hence $x\not \in C^* (\varphi,
X)$. Thus, $C^*(\varphi, X) = (X \setminus X^+) \cup \{{\bf 1}\} =
\mbox{\rm Fix}\, \varphi \ne C(\varphi, X)$; see also Figure \ref{fig3}.

Next, fix a decreasing $C^{\infty}$-function $g:\R \to \R$ with
$g(s)=1$ for all $s\le 1$ and $g(s)= 0$ for all $s\ge 2$. Let $\psi$
be the flow on $X$ generated by
$\dot x = v(x) V(x)$, where $v: X\to \R $ is given by
\begin{equation}\label{ex255eq1}
v(x) = \left\{
\begin{array}{ll}
1 + e^{4\pi} g \bigl( (s - \log x_1 \cos s - \log x_2 \sin
s ) s^{-1} e^{s- 1/s} \bigr)  & \mbox{\rm if } x\in X^+
\setminus \{{\bf 1}\} \, , \\
1  & \mbox{\rm otherwise}\, ;
\end{array}
\right.
\end{equation}
note that the vector field $vV$ is $C^{\infty}$.
Similarly to Example \ref{ex25}, $(X\setminus X^+ ) \cup \{{\bf 1}\} = 
\mbox{\rm Fix}\, \psi$, and (exponential) polar coordinates in
$X^+$ transform $\dot x = v(x)V(x)$ into
\begin{equation}\label{ex255eq2}
\dot r = \dot \theta = f(r) + e^{4\pi} f(r) g\Bigl( \bigl(1 - \cos
(\theta - r) \bigr) e^{r-1/r} \Bigr)  \, .
\end{equation}
Note that $r-\theta$ again is constant for every solution of
(\ref{ex255eq2}). Specifically, given any $0\le a \le \frac12 $, let
$(r_a,\theta_a)$ be the solution of (\ref{ex255eq2}) with $r(0) = 2\pi
(1+a)$ and $\theta(0) =0$. Then $r_a(t) - \theta_a(t) = 2\pi (1+a)$
and $r_a(t) - r_0 (t) \le 2\pi a$ for all $t\ge 0$. Notice that $\lim_{t\to
+\infty} r_a(t) = +\infty$. Consequently, for every $0< a \le \frac12 $ there
exists a $t_a\in \R^+$ such that $\dot r_a = f(r_a)$ for all $t\ge t_a$, but
also $e^{-1/r_0} (1+e^{4\pi}) > 1 +e^{3\pi}$. Clearly,
$\lim_{a\downarrow 0 } t_a = +\infty$; assume w.l.o.g.\ that
$a\mapsto t_a$ is decreasing on $]0,\frac12]$. It follows that $\dot r_0 \ge e^{-r_0}
(1 + e^{3\pi}) $ as well as $\dot r_a \le e^{-r_a}$ on $[t_a, +\infty[$,
and therefore also, with $\widetilde{t}_a := t_a + e^{4\pi + r_0 (t_a)}$,
$$
\theta_0 (t) - \theta_a (t) = r_0(t) - r_a (t) + 2\pi a \ge 2\pi a +
\log \frac{e^{r_0 (t_a)} + (t-t_a) (1+e^{3\pi})}{e^{r_a (t_a)} + t -
  t_a} > 3\pi \quad \forall t \ge \widetilde{t}_a \, .
$$
Deduce from this and the continuity of $a\mapsto
\theta_a(t)$, that, given any integer $j\ge 2$ and $t\ge
\widetilde{t}_{1/j}$, there exists $0< a_j(t) \le j^{-1}$ such that
$\theta_{a_j(t)} (t) + \frac34 \pi \in 2\pi \Z$.

With these preparations, consider the point $u = 
\left[
  \!  \begin{array}{c} 
 e^{2\pi} \\ 1
\end{array}\! \right]\not \in C^*(\varphi, X)$, and let $(t_n)$ be
any sequence with $|t_n| \to +\infty$. If $t_n \to -\infty$ then
$\bigl( \psi_{t_n} (u)\bigr)$ is bounded, in fact $\psi_{t_n} (u)\to
{\bf 1}$, so it suffices to assume that $(t_n)$ is
increasing, and $t_1 \ge \widetilde{t}_{1/2}$. Pick a sequence $(j_n)$
with $\widetilde{t}_{1/j_n} \le t_n < \widetilde{t}_{1/j_{n+1}}$ for
all $n$. Note that $j_n \to \infty$, and hence $0<a_{j_n} (t_n) < j_n^{-1} \to
0$. Writing $b_n : = a_{j_n} (t_n) $ for convenience, consider
$$
u_n =  \left[
  \!  \begin{array}{c} 
 e^{2\pi (1 + b_n ) } \\ 1
\end{array}\! \right] \quad \forall n \in \N \, .
$$   
With this, not only $u_n \to u$, but also 
$\psi_{t_n} (u_n) = e^{- r_{b_n} (t_n)/\sqrt{2}} {\bf 1} \: \to \: 0$,
showing that $\bigl( \psi_{t_n} (u_n)\bigr)$ is bounded. In other
words, $u\in C^* (\psi, X)$. Recall that $\varphi$ and $\psi$ are generated
by $\dot x = V(x)$ and $\dot x = v(x)V(x)$, respectively, with $v$ given by (\ref{ex255eq1}), and $1
\le v \le 1 +e^{4\pi}$. As pointed out right after Proposition
\ref{prop23}, the flows $\varphi, \psi$ are orbit equivalent with $h= {\rm id}_X$, and yet $h\bigl( C^* (\varphi, X)\bigr) \ne C^* (\psi, X)$.
\end{example}

\section{Cores of linear flows}\label{sec2}

In a linear flow, naturally an invariant set is of particular interest
if it also is a (linear) subspace. For instance, $\mbox{\rm Fix}\, \Phi$ and $\mbox{\rm Bnd}\, \Phi$ (but not, in
general, $\mbox{\rm Per}\, \Phi$) are $\Phi$-invariant subspaces for any linear flow $\Phi$, and
so are all uniform cores. As seen in the previous section, uniform
cores are well-behaved under products (Lemma \ref{prop26}) but not under orbit equivalence
(Example \ref{ex255}), whereas for (non-uniform)
cores the situation is the exact opposite (Lemma
\ref{prop27} and Example \ref{ex25}). This discrepancy is consistent with a lack of equality in (\ref{eq21})
in general. One main result of this section, Theorem
\ref{thm3xx} below, shows that both inclusions in (\ref{eq21}) are in
fact equalities --- provided that $\varphi$ is linear. As an important
consequence, all cores of linear flows are invariant subspaces that are well-behaved under
orbit equivalence {\em and\/} under products. With regard to the last
assertion in  Lemma \ref{prop27}, the following additional property of
orbit equivalences is useful when dealing with linear flows; again,
for convenience this property is hereafter assumed to be part of what it means for $\Phi$ to be $(h,\tau)$-related to $\Psi$.

\begin{prop}\label{prop30a}
Let $\Phi, \Psi$ be linear flows, and assume
that $\Phi$ is $(h,\tau)$-related to $\Psi$. Then $\Phi$ is
$(\, \widetilde{h},\tau)$-related to $\Psi$ where $\widetilde{h}(0)=0$.
\end{prop}

In a first step towards Theorem \ref{thm3xx}, cores of irreducible linear flows are considered.
Recall that $\Phi$ is {\bf irreducible} if $X =  Z \oplus \widetilde{Z}$,
with $\Phi$-invariant subspaces $Z$, $\widetilde{Z}$, implies that $Z=\{0
\}$ or $\widetilde{Z}=\{0\}$. Plainly, $\Phi$ is irreducible
if and only if, relative to the appropriate basis, $A^{\Phi}$ is a single
real Jordan block. In particular, for irreducible $\Phi$ the spectrum $\sigma
(\Phi):= \sigma (A^{\Phi})$ is either a real
singleton or a non-real complex conjugate pair. In order to
clarify the structure of cores of irreducible linear flows, for
every $s\in \R$ denote by $\lceil s\rceil$ and $\lfloor s \rfloor$ the
smallest integer $\ge s$ and the largest integer $\le s$, respectively.

\begin{lem}\label{lem31}
Let $\Phi$ be an irreducible linear flow on $X$. Then $C^*(\Phi, X) =
C(\Phi, X)$, and
$$
\mbox{\rm dim}\, C^*(\Phi, X) =
\mbox{\rm dim}\, C(\Phi, X) =
\left\{
\begin{array}{ll}
0 & \mbox{if } \sigma(\Phi)\cap \imath \R = \varnothing \, , \\[1mm]
\lceil  \frac12 \, \mbox{\rm dim}\, X\rceil & \mbox{if }
\sigma(\Phi)= \{0\} \, , \\[1mm]
2\lceil  \frac14 \, \mbox{\rm dim}\, X\rceil & \mbox{if }
\sigma(\Phi)\subset \imath \R \setminus  \{0\} \, .
\end{array}
\right.
$$
Similarly, $C_0^* (\Phi, X) = C_0(\Phi, X)$, and
$$
\mbox{\rm dim}\, C_0^*(\Phi, X) =
\mbox{\rm dim}\, C_0(\Phi, X) =
\left\{
\begin{array}{ll}
0 & \mbox{if } \sigma(\Phi)\cap \imath \R = \varnothing \, , \\[1mm]
\lfloor  \frac12 \, \mbox{\rm dim}\, X\rfloor & \mbox{if }
\sigma(\Phi)= \{0\} \, , \\[1mm]
2\lfloor  \frac14 \, \mbox{\rm dim}\, X\rfloor & \mbox{if }
\sigma(\Phi)\subset \imath \R \setminus  \{0\} \, .
\end{array}
\right.
$$
\end{lem}

\noindent
The proof of Lemma \ref{lem31} utilizes explicit
calculations involving several families of special matrices. These
matrices are reviewed beforehand for the reader's convenience.
First, given any $m\in \N$ and $\omega \in \C$, consider the diagonal matrix
$$
D_m(\omega) = \mbox{\rm diag}\, [1,\omega , \ldots , \omega^{m-1}] \in \C^{m\times m} \, , 
$$
for which $D_m (\omega) \in \R^{m\times m}$ whenever $\omega \in
\R$, as well as the nilpotent Jordan block of size $m$,
$$
J_m = \left[ 
\begin{array}{ccccc}
0 & 1 & 0 & \cdots & 0 \\
\vdots & \ddots & \ddots &   &  \vdots \\
&  &  & \ddots & 0 \\
\vdots  & & & \ddots  & 1\\
0 & \cdots & & \cdots & 0
\end{array}
\right]\in \R^{m\times m} .
$$
Clearly, $D_m (1)= \mbox{\rm id}_{\R^m} =:I_m$, and $D_m(\omega )^{-1}
= D_m (\omega^{-1})$ whenever $\omega \ne 0$, but also
\begin{equation}\label{eq3n1}
D_m(\omega)^{-1} \enspace \mbox{\rm and} \enspace
\omega^{1-m} D_m(\omega) \enspace \mbox{\rm are bounded (in fact, converge) as }|\omega|\to +\infty \, .
\end{equation}
Moreover, recall that $J_m^m=0$, and hence
$$
e^{tJ_m} = I_m + t J_m + \ldots + \frac{t^{m-1}}{(m-1)!} J_m^{m-1}
\quad \forall t \in \R \, .
$$
A simple lower bound for the size of $e^{tJ_m}x$ is as follows.

\begin{prop}\label{prop3n1}
For every $m\in \N$ and norm $\|\, \cdot \,\|$ on $\R^m$ there exists a
$\nu \in \R^+$ such that
$$
\left\|
e^{tJ_m} x
\right\| \ge \frac{\nu \|x\|}{\sqrt{1 + t^{2m-2}}} \quad \forall t \in
\R, x \in \R^m \, .
$$
\end{prop}

\noindent
Next, recall that the function $1/\Gamma$, the reciprocal of the Euler Gamma
function, is entire \cite[Ch.6]{AS}. In particular, given any $m,n\in \N$ and $\omega
\in \C$, the Toeplitz-type matrix
$$
\Delta_{m,n}^{[\omega ]} : = \left[
\begin{array}{cccc}
1/\Gamma (\omega  + 1) & 1/\Gamma (\omega  + 2) & \cdots & 1/\Gamma (\omega
+ n) \\
1/\Gamma (\omega   ) & 1/\Gamma (\omega + 1) & \cdots  & 1/\Gamma (\omega 
-1 + n ) \\
\vdots & \vdots &  & \vdots \\
1/\Gamma (\omega  - m + 2) & 1/\Gamma (\omega  - m + 3) & \cdots &
1/\Gamma (\omega  - m + n + 1)
\end{array}
\right]\in \C^{m\times n}
$$
is well-defined, each of its entries depending analytically on
$\omega $. Note that $\Delta_{m,n}^{[\omega ]}\in \R^{m \times n}$
whenever $\omega \in \R$, and $\Delta_{m,n}^{[\omega ]}$ is upper triangular
(respectively, the zero matrix) if and only if $\omega $ is an integer
$\le 0$ (an integer $\le - n$). Also, in
the case of a square matrix, the function $ \det
\Delta_{m,m}^{[\, \cdot \, ]}$ is entire and not constant, and hence
$\Delta_{m,m}^{[\omega ]}$ is invertible for most $\omega $.

\begin{prop}\label{prop3n2}
Let $m\in \N$ and $\omega \in \C$. Then
$$
\det \Delta_{m,m}^{[\omega ]} = \prod\nolimits_{j=1}^m \frac{\Gamma
  (j)}{\Gamma (\omega  + j)} \, ;
$$
in particular, $\Delta_{m,m}^{[\omega ]} $ is invertible unless $\omega $ is
a negative integer.
\end{prop}

To appreciate the usefulness of the matrices $D_m$ and $\Delta_{m,n}^{[\omega]}$
in the study of linear flows, note that 
$$
e^{tJ_m} =  \left[ 
\begin{array}{ccccc}
1 & t &   & \cdots &\displaystyle\frac{ t^{m-1}}{(m-1)!} \\
0 & \ddots & \ddots &   &  \vdots \\
&  &  & \ddots &    \\
\vdots  & & & \ddots  & t\\
0 & \cdots & & 0  & 1
\end{array}
\right] = D_m(t)^{-1} \Delta_{m,m}^{[0]} D_m(t) \quad \forall t \in \R
\setminus \{0 \}\, .
$$
More generally, for any $1\le j \le m$ and $t\ne 0$, the $m\times m$-matrix $e^{tJ_m}$ can be
partitioned as
\begin{equation}\label{eq3n2}
e^{tJ_m} = \left[ 
\begin{array}{c|c}
D_j(t)^{-1} \Delta_{j,m-j}^{[0]}  D_{m-j} (t) & t^{m-j}
D_{j}(t)^{-1}\Delta_{j,j}^{[m-j]} D_j(t) \\[1mm] \hline \\[-4mm] t^{-j} D_{m-j}(t)^{-1} \Delta_{m-j,m-j}^{[-j]} D_{m-j}(t) & t^{m-2j}
D_{m-j}(t)^{-1} \Delta_{m-j,j}^{[m-2j]} D_j(t)
\end{array}
\right]  .
\end{equation}

\begin{proof}[Proof of Lemma \ref{lem31}] For simplicity, suppress the symbols $(\Phi, X)$ in all cores, i.e., write
  $C$ instead of $C(\Phi, X)$ etc. Note that if 
$\mbox{\rm dim}\, X \le  1$ then $C_0^*  = C_0  = \{ 0
\}$, whereas $C^* = C $ equals $\{0\}$ or $X$,
depending on whether $\Phi\ne  0$ or $\Phi =  0$. Thus the lemma holds
if $\mbox{\rm dim}\, X \le 1$. Henceforth assume $\mbox{\rm
  dim}\, X \ge 2$, and let $(b_1, \ldots , b_{{\rm dim} \, X})$ be an ordered
basis of $X$ relative to which $A^{\Phi}$ is a single real Jordan
block. Throughout, no notational distinction is made between linear operators
on (respectively, elements of) $X$ on the one hand, and their coordinate matrices
(column vectors) relative to $(b_j)$ on the other hand.

Assume for the time being that $\sigma (\Phi) = \{ a \}$ with $a\in
\R$, and hence $A^{\Phi} = a I_{{\rm dim}\, X} + J_{{\rm dim} \, X}$. In
this case,
\begin{equation}\label{eq3n3}
\|\Phi_t x\| = \|e^{ta} e^{tJ_{{\rm dim} \, X}} x\| \ge e^{ta}
\frac{\nu \|x\|}{\sqrt{1 + t^{2\, {\rm dim} \, X - 2}}} \quad \forall t \in \R
, x\in X \, ,
\end{equation}
by Proposition \ref{prop3n1}. Pick any $x\in C$. If $a\ne 0$
and $(\Phi_{t_n} x_n)$ is bounded for appropriate sequences
$(t_n)$ and $(x_n)$ with $at_n \to +\infty$ and $x_n \to x$, then
(\ref{eq3n3}) implies that $x=0$. Thus $C = \{0\}$ whenever
$a\ne 0$, and only the case of $a=0$ has to be considered further.

Assume first that $\mbox{\rm dim}\, X$ is {\em odd}, say $\mbox{\rm
  dim}\, X = 2d+1$ with $d\in \N$. Letting $m=2d+1$, deduce from
(\ref{eq3n2}) with $j=d+1$ that for all $t\ne 0$,
\begin{equation}\label{eq3n4A}
\Phi_t  =  \left[ 
\begin{array}{c|c}
D_{d+1}(t)^{-1} \Delta_{d+1,d}^{[0]} D_{d} (t) & t^{d}
D_{d+1}(t)^{-1}\Delta_{d+1,d+1}^{[d]} D_{d+1}(t) \\ [1mm] \hline \\[-4mm] 0 & t^{-1}
D_{d}(t)^{-1} \Delta_{d,d+1}^{[-1]} D_{d+1}(t)
\end{array}
\right] \, ,
\end{equation}
because $\Delta_{d,d}^{[-d-1]}=0$, whereas with $j=d$,
\begin{equation}\label{eq3n4B}
\Phi_t =   \left[ 
\begin{array}{c|c}
D_{d}(t)^{-1} \Delta_{d,d+1}^{[0]} D_{d+1} (t) & t^{d+1}
D_{d}(t)^{-1}\Delta_{d,d}^{[d+1]} D_{d}(t) \\ [1mm] \hline \\[-4mm]
t^{-d} D_{d+1}(t)^{-1} \Delta_{d+1,d+1}^{[-d]} D_{d+1}(t)
 & t D_{d+1}(t)^{-1} \Delta_{d+1,d}^{[1]} D_{d}(t)
\end{array}
\right] \, .
\end{equation}
Let $V=\mbox{\rm span}\, \{b_1, \ldots , b_d\}$, pick any $x =
\left[  \begin{array}{c} v  \\  \hline \\[-5mm] 0 \end{array}
\right]\in V$ with $v 
  \in \R^d$, and consider
$$
x_t : = \left[
\begin{array}{c}
v  \\  \hline \\[-4mm] - t^{-d} D_{d+1}(t)^{-1} \bigl(\Delta_{d+1,d+1}^{[d]}\bigr)^{-1}
\Delta_{d+1,d}^{[0]} D_d(t) v  
\end{array}
\right] \quad \forall t \in \R \setminus \{ 0 \}\, .
$$
(Recall that $\Delta_{d+1,d+1}^{[d]}$ is invertible by Proposition
\ref{prop3n2}.) From (\ref{eq3n1}), it is clear that $\lim_{|t|\to
  +\infty} x_t = x$, and together with the expression
for $\Phi_t$ in (\ref{eq3n4A}) also 
$$
\Phi_t x_t = \left[
\begin{array}{c}
0 \\  \hline \\[-4mm]  - t^{-(d+1)} D_{d}(t)^{-1} \Delta_{d,d+1}^{[-1]} \bigl( \Delta_{d+1,d+1}^{[d]}\bigr)^{-1}
\Delta_{d+1,d}^{[0]} D_d(t) v  
\end{array}
\right] \stackrel{|t|\to +\infty}{\longrightarrow} 0 \, .
$$
Thus $x\in C_0^*$. Since $x\in V$ was arbitrary, $V
\subset C_0^* $. Conversely, given any $x =
\left[ \begin{array}{c} v  \\  \hline \\[-5mm]  w   \end{array} \right]\in
C_0 $, with $v  \in \R^d$, $w  \in \R^{d+1}$, there exist
sequences $(t_n)$, $(v _n)$, and $(w _n)$ with $t_n \to +\infty$,
$v _n \to v $, and $w _n \to w $ such that
\begin{equation}\label{eq3n1xx}
\Phi_{t_n} \left[  \begin{array}{c} v _n \\  \hline \\[-5mm]  w _n  \end{array}
  \right] = 
\left[
\begin{array}{c}
D_{d+1}(t_n)^{-1} \bigl( \Delta_{d+1,d}^{[0]} D_{d}(t_n)v _n + t_n^d
\Delta_{d+1,d+1}^{[d]} D_{d+1}(t_n) w _n\bigr)  \\[1mm]  \hline \\[-5mm] 
\ldots  
\end{array} 
\right] \to 0 \, .
\end{equation}
Recall from (\ref{eq3n1}) that $\bigl( t_n^{-d} D_{d+1}(t_n)\bigr)$
converges, and apply these matrices to the first component of (\ref{eq3n1xx}) to
obtain 
$$
t_n^{-d} \Delta_{d+1,d}^{[0]} D_{d}(t_n)v _n + \Delta_{d+1,d+1}^{[d]}
D_{d+1}(t_n) w _n \to 0 \, .
$$
With (\ref{eq3n1}) also $t_n^{-d} \Delta_{d+1,d}^{[0]} D_{d}(t_n)v
_n\to 0$, and hence $\Delta_{d+1,d+1}^{[d]}
D_{d+1}(t_n) w _n \to 0$. Since $\Delta_{d+1,d+1}^{[d]}$ is invertible
and $\bigl( D_{d+1}(t_n)^{-1}\bigr)$ converges, $w _n \to 0 = w $,
i.e., $x\in V$. As $x\in C_0 $ was arbitrary, $C_0 \subset V$, and
hence $C_0^*  = C_0   = V$; note that $\mbox{\rm dim}\, V = d =
\lfloor \frac12 \, \mbox{\rm dim}\, X \rfloor$.

Next, given any $x=
\left[   \begin{array}{c} w  \\  \hline \\[-5mm]  0 \end{array} \right]\in V
\oplus \mbox{\rm span}\, \{ b_{d+1}\}$, with $w  \in \R^{d+1}$,
consider
$$
x_t : = \left[
\begin{array}{c}
w  \\  \hline \\[-4mm]  - t^{-(d+1)} D_{d}(t)^{-1} \bigl( \Delta_{d,d}^{[d+1]}\bigr)^{-1}
\Delta_{d,d+1}^{[0]} D_{d+1}(t) w  
\end{array}
\right] \quad \forall t \in \R \setminus \{0\} \, ,
$$
which again is well-defined as $\Delta_{d,d}^{[d+1]}$ is
invertible. As before, (\ref{eq3n1}) implies $\lim_{|t|\to
  +\infty} x_t = x$, and together with the expression for
$\Phi_t$ in (\ref{eq3n4B}) also shows that
$$
\Phi_t x_t = \left[
\begin{array}{c}
0  \\  \hline \\[-4mm]   t^{-d} D_{d+1}(t)^{-1} \bigl( \Delta_{d+1,d+1} ^{[-d]} -
\Delta_{d+1,d}^{[1]} \bigl( \Delta_{d,d}^{[d+1]}\bigr) ^{-1}
\Delta_{d,d+1}^{[0]} \bigr) D_{d+1}(t) w  
\end{array}
\right] 
$$
converges as $|t|\to +\infty$, and hence
$x\in C^*$. Thus $ V\oplus \mbox{\rm span}\, \{ b_{d+1}\} \subset C^*$. Conversely, given any
 $x = \left[   \begin{array}{c} w \\  \hline \\[-5mm]  v    \end{array}  \right]\in
C$, there exist sequences $(t_n)$, $(w _n)$, and $(v _n)$ with $t_n \to +\infty$,
$w _n \to w $, and $v _n \to v $ such that
\begin{equation}\label{eq3n1yy}
\Phi_{t_n} \left[   \begin{array}{c} w _n \\  \hline \\[-5mm]   v
    _n   \end{array} \right] = 
\left[
\begin{array}{c}
D_{d}(t_n)^{-1} \bigl( \Delta_{d,d+1}^{[0]} D_{d+1}(t_n)w _n + t_n^{d+1}
\Delta_{d,d}^{[d+1]} D_{d}(t_n) v _n\bigr) \\[1mm]  \hline \\[-5mm] 
\ldots  
\end{array}
\right] 
\end{equation}
is bounded as $n\to \infty$. Since $t_n^{-(d+1)} D_{d}(t_n)\to 0$,
applying these matrices to the first component of (\ref{eq3n1yy}) yields
$$
t_n^{-(d+1)} \Delta_{d,d+1}^{[0]} D_{d+1}(t_n)w_n + \Delta_{d,d}^{[d+1]}
D_{d}(t_n) v _n \to 0 \, .
$$
As before, also $\Delta_{d,d}^{[d+1]}
D_{d}(t_n) v _n \to 0$, and hence $v _n \to 0 = v $,
i.e., $x\in  V \oplus \mbox{\rm span}\, \{ b_{d+1}\} $. In summary, $C^*  = C =  V\oplus \mbox{\rm span}\, \{ b_{d+1}\}$.
This establishes the lemma when $\sigma (\Phi)\subset \R$
and $\mbox{\rm dim}\, X$ is odd, as $\mbox{\rm dim}\, V \oplus \mbox{\rm
  span}\, \{ b_{d+1}\}  = d+1 = \lceil \frac12 \, \mbox{\rm dim}\, X \rceil$.

The case of $\mbox{\rm dim}\, X$ {\em even}, say $\mbox{\rm dim}\, X =
2d$, is similar but simpler: In this case, (\ref{eq3n2}) with $m=2d$, $j=d$
yields
$$
\Phi_t  =  \left[ 
\begin{array}{c|c}
D_{d}(t)^{-1} \Delta_{d,d}^{[0]} D_{d} (t) & t^{d}
D_{d}(t)^{-1}\Delta_{d,d}^{[d]} D_{d}(t) \\ [1mm] \hline \\[-4mm] 0 &
D_{d}(t)^{-1} \Delta_{d,d}^{[0]} D_{d}(t)
\end{array}
\right] \quad \forall t \in \R \setminus \{0 \} \, .
$$
On the one hand, if  $x =
\left[   \begin{array}{c} v  \\  \hline \\[-5mm]  0 \end{array} \right]\in V$ with $v 
  \in \R^d$, then
$$
x_t : = \left[
\begin{array}{c}
v  \\  \hline \\[-4mm]   - t^{-d} D_{d}(t)^{-1} \bigl( \Delta_{d,d}^{[d]}\bigr)^{-1}
\Delta_{d,d}^{[0]} D_d(t) v  
\end{array}
\right] \stackrel{|t|\to +\infty}{\longrightarrow}  x \, ,
$$
by (\ref{eq3n1}), but also
$$
\Phi_t x_t = \left[
\begin{array}{c}
0\\  \hline \\[-4mm]  - t^{-d} D_{d}(t)^{-1} \Delta_{d,d}^{[0]} \bigl( \Delta_{d,d}^{[d]}\bigr)^{-1}
\Delta_{d,d}^{[0]} D_d(t) v  
\end{array}
\right] \stackrel{|t|\to +\infty}{\longrightarrow} 0 \, ,
$$
showing that $V \subset C_0^*$. On the other hand, if $x =
\left[   \begin{array}{c} u  \\  \hline \\[-5mm]  v  \end{array} \right]\in
C$ with $u, v   \in \R^d$, then there exist
sequences $(t_n)$, $(u _n)$, and $(v_n)$ with $t_n \to +\infty$,
$u _n \to u $, and $v_n \to v$, such that
$$
\Phi_{t_n} \left[   \begin{array}{c} u _n \\  \hline \\[-5mm]
    v_n  \end{array}  \right] = 
\left[
\begin{array}{c}
D_{d}(t_n)^{-1} \bigl( \Delta_{d,d}^{[0]} D_{d}(t_n)u _n + t_n^d
\Delta_{d,d}^{[d]} D_{d}(t_n) v_n\bigr)  \\[1mm]  \hline \\[-5mm] 
\ldots  
\end{array}
\right] 
$$
is bounded as $n\to \infty$. Applying $t_n^{-d} D_d(t_n)\to 0$ to the first component
yields $v_n \to 0 = v$, as before, and hence $x\in V$. In
summary, $C_0^* = C^* = C_0 = C = V$. Noting that
$\mbox{\rm dim}\, V = d = \frac12\, \mbox{\rm dim}\, X$ establishes the lemma when $\sigma (\Phi)\subset \R$
and $\mbox{\rm dim}\, X$ is even.

Finally, it remains to consider the case of $\sigma (\Phi) = \{ a \pm
\imath b\}$ with $a\in \R$, $b\in \R^+$. Since $\mbox{\rm dim}\, X$ is
even in this case, let $m = \frac12 \, \mbox{\rm dim}\, X$. Then
$
A^{\Phi} = a I_{2m} + \left[
\begin{array}{r|r}
J_m & - bI_m \\ \hline
bI_m & J_m
\end{array}
\right]$,
which in turn yields
\begin{equation}\label{eq3nzz}
\Phi_t = e^{ta} \left[ \begin{array}{r|r}
\cos (bt) I_m & - \sin (bt) I_m \\  \hline \\[-5mm]
\sin (bt) I_m &  \cos (bt) I_m
\end{array}
\right] 
\left[
\begin{array}{c|c}
e^{tJ_m} & 0 \\ \hline  \\[-5mm]
0  & e^{tJ_m}
\end{array}
\right] \quad \forall t\in \R \, .
\end{equation}
From (\ref{eq3nzz}) and Proposition \ref{prop3n1}, it is clear that,
with an appropriate $\widetilde{\nu} \in \R^+$,
$$
\|\Phi_t x\| \ge e^{ta} \frac{\widetilde{\nu } \|x\|}{\sqrt{1 + t^{2m-2}}} \quad
\forall t\in \R , x\in X \, .
$$
As before, it follows that $C = \{0\}$ unless $a=0$, so only that case has to be analyzed further. This analysis is virtually identical to the one above, simply because the
left matrix on the right-hand side of (\ref{eq3nzz}) does not in any
way affect boundedness or
convergence to $0$ of $\Phi_t x$: On the one hand, if $m=2d+1$ then,
with $W = \mbox{\rm span}\, \{b_1, \ldots , b_d, b_{m+1}, \ldots ,
b_{m+d}\}$,
$$
C_0^* = C_0 = W \, , \quad C^* = C = W \oplus \mbox{\rm span}\,
\{b_{d+1}, b_{m+d+1}\} \, .
$$
On the other hand, if $m=2d$ then
$C_0^* = C_0 = C^* = C = W$.
In either case, $\mbox{\rm dim}\, W = 2d = 2 \lfloor \frac14 \, \mbox{\rm
dim}\, X \rfloor$ and $\mbox{\rm dim}\, W\oplus \mbox{\rm span}\,
\{b_{d+1}, b_{m+d+1}\} = 2d+2 = 2 \lceil \frac14 \, \mbox{\rm
dim}\, X \rceil$.
\end{proof}

Given any $\Phi$-invariant subspace $Z$ of $X$, denote by $\Phi_Z$ the linear
flow induced by $\Phi$ on $Z$, that is, $\Phi_Z(t,x) = \Phi_t x$ for all
$(t,x)\in \R \times Z$. Note that if $X = \bigoplus_{j=1}^{\ell} Z_j$ with
$\Phi$-invariant subspaces $Z_1, \ldots , Z_{\ell}$, then $\Phi$ is
flow equivalent to the linear flow $\bigtimes_{j=1}^{\ell}
\Phi_{Z_j}$ on $\bigtimes_{j=1}^{\ell} Z_j$, via the linear isomorphism $h(x)= (P_1 x, \ldots, P_{\ell}
x)$ and $\tau_x = \mbox{\rm id}_{\R}$ for all $x\in X$; here $P_j$
denotes the linear projection of $X$ onto $Z_j$ along $\bigoplus_{k\ne j}
Z_k$. With this, an immediate consequence of Lemma \ref{lem31}
announced earlier is 

\begin{theorem}\label{thm3xx}
Let $\Phi$ be a linear flow on $X$. Then $C^*(\Phi, X)= C (\Phi, X)$, $C_0^*(\Phi, X) = C_0(\Phi, X)$, and both sets are
$\Phi$-invariant subspaces of $X^{\Phi}_{\sf C}$.
\end{theorem}

\noindent
{\em Proof.} Let $X = \bigoplus_{j=1}^{\ell} Z_j$ be such that each flow $\Phi_{Z_j}$ is
irreducible. With $h$ as above,
\begin{align*}
C(\Phi, X) & = h^{-1}   C \left(  \bigtimes\nolimits_{j=1}^{\ell} \Phi_{Z_j} ,
    \bigtimes\nolimits_{j=1}^{\ell} Z_j\right)  \subset h^{-1} \left(
  \bigtimes\nolimits_{j=1}^{\ell} C(\Phi_{Z_j}, Z_j)\right) \\
& = h^{-1}  \left( \bigtimes\nolimits_{j=1}^{\ell} C^*(\Phi_{Z_j},
  Z_j)\right) =  h^{-1}  C^* \left(  \bigtimes\nolimits_{j=1}^{\ell} \Phi_{Z_j} ,
    \bigtimes\nolimits_{j=1}^{\ell} Z_j\right)  = C^*(\Phi, X) \, ,
\end{align*}
where, from left to right, the equalities are due to Lemmas
\ref{prop27}, \ref{lem31}, \ref{prop26}, and the fact that $\Phi$ and
$\bigtimes_{j=1}^{\ell} \Phi_{Z_j}$ are flow equivalent via $h$, respectively,
whereas the inclusion is  the $\ell $-factor analogue of (\ref{eq23}). With (\ref{eq21}), therefore, $C^*(\Phi, X) = C(\Phi,
X)$, and recalling that $h(0)=0$, also $C_0^* (\Phi, X) = C_0 (\Phi, X)$. Let $J= \{1\le
j\le \ell : \sigma (\Phi_{Z_j}) \subset \imath
\R\}$. By Lemma
\ref{lem31}, $C(\Phi_{Z_j}, Z_j)= \{0\}$ whenever $j\not \in J$, and consequently
$$
\qquad C(\Phi, X) =  h^{-1} \left(\bigtimes\nolimits_{j=1}^{\ell}
  C(\Phi_{Z_j}, Z_j)\right)  = \bigoplus\nolimits_{j\in J } \!\! C(\Phi_{Z_j}, Z_j)\subset
\bigoplus\nolimits_{j\in J} \!\! Z_j = X_{\sf C}^{\Phi} \, . \qquad \qed
$$

In light of Theorem \ref{thm3xx}, when dealing with linear flows only the symbols $C$ and
$C_0$ are used henceforth. Note that if $Z$ is a
$\Phi$-invariant subspace of $X$ then one may also consider
cores of the flow $\Phi_Z$, and this idea of restriction can be iterated. To
do so in a systematic way, given any binary sequence
$\epsilon = (\epsilon_{k })_{k \in \N_0}$, that is, $\epsilon_{k }\in \{0,1\}$
for all $k $, let $C^{\epsilon, -1} (\Phi, X)= X$ and, for every $k  \in
\N_0$, let
$$
C^{\epsilon, k } (\Phi, X) = \left\{
\begin{array}{cl}
C \bigl( \Phi_{C^{\epsilon, k  - 1}(\Phi, X)}, C^{\epsilon, k  - 1}(\Phi, X)\bigr) & \mbox{\rm
  if } \epsilon_{k } = 0 \, , \\[1mm]
C_0 \bigl( \Phi_{C^{\epsilon, k  - 1}(\Phi, X)}, C^{\epsilon, k  -
  1}(\Phi, X)\bigr) & \mbox{\rm
  if } \epsilon_{k } = 1 \, .
\end{array}
\right.
$$
Clearly $X \supset C^{\epsilon, 0}(\Phi, X) \supset   C^{\epsilon,
  1}(\Phi, X) \supset \cdots $, and hence the {\bf iterated core}
$$
C^{\epsilon} (\Phi, X) := \lim\nolimits_{k  \to \infty} C^{\epsilon,
k } (\Phi, X) = \bigcap\nolimits_{k \in \N_0} C^{\epsilon, k } (\Phi, X)
$$
is a $\Phi$-invariant subspace naturally
inheriting basic properties from $C(\Phi, X)$ and $C_0(\Phi, X)$.

\begin{lem}\label{lem35xx}
Let $\Phi, \Psi$ be linear flows on $X, Y$, respectively, and
$\epsilon$ a binary sequence.
\begin{enumerate}
\item If $\Phi$ is $(h,\tau)$-related to $\Psi$ then $h\bigl(
  C^{\epsilon} (\Phi, X)\bigr) = C^{\epsilon} (\Psi, Y)$.
\item $C^{\epsilon} (\Phi\times \Psi, X\times Y) = C^{\epsilon}(\Phi,
  X) \times C^{\epsilon} (\Psi, Y)$.
\end{enumerate}
\end{lem}

\begin{proof}
With Lemma \ref{prop27} and Proposition \ref{prop30a}, $h\bigl(
C^{\epsilon, 0} (\Phi, X)\bigr) = C^{\epsilon, 0}(\Psi, Y)$. By
induction, for every $k  \ge 1$, $\Phi_{C^{\epsilon, k  - 1}(\Phi,
  X)}$ is $(h_k, \tau_k)$-related to $\Psi_{C^{\epsilon, k  - 1}(\Psi,
  Y)}$, with $h_k$ and $\tau_k$ denoting the restrictions of $h$ and
$\tau$ to $C^{\epsilon, k-1} (\Phi, X)$ and $\R \times C^{\epsilon,k-1}
(\Phi, X)$ respectively. Hence $h\bigl(C^{\epsilon, k } (\Phi, X)\bigr) = C^{\epsilon,
  k }(\Psi, Y)$, which proves (i). Similarly, with Lemma \ref{prop26} and
Theorem \ref{thm3xx}, induction yields $C^{\epsilon, k } (\Phi\times
\Psi, X\times Y) = C^{\epsilon, k } (\Phi, X) \times C^{\epsilon,
  k } (\Psi, Y)$ for every $k\ge 1$, which establishes (ii).
\end{proof}

It is not hard to see that $C^{\epsilon} (\Phi, X) = \{0\}$ whenever
$\epsilon _{k } = 1$ for infinitely many $k $.
In what follows, therefore, only terminating binary sequences (i.e., $\epsilon_{k }
= 0$ for all large $k $) are of interest. Any such
sequence (uniquely) represents a non-negative
integer. More precisely, given any $n\in \N_0$, let $\epsilon (n)$
be the binary sequence of base-$2$ digits of $n$ in reversed (i.e., ascending) order, that is,
$$
n = \sum\nolimits_{k  = 0}^{\infty} 2^{k } \epsilon(n)_{k } 
\quad \forall n \in \N_0 \, ;
$$
thus, for instance, $\epsilon (4) = (0,0,1,0,0 ,\ldots)$ and $\epsilon
(13) = (1,0,1,1,0,0,\ldots )$. To understand the structure of
$C^{\epsilon (n)} (\Phi, X)$, first consider the case of an
irreducible flow.

\begin{lem}\label{lem35zz}
Let $\Phi$ be an irreducible linear flow on $X$.
\begin{enumerate}
\item If $\sigma(\Phi)\cap \imath \R = \varnothing$ then $C^{\epsilon
    (n)}(\Phi, X) = \{0\}$ for all $n\in \N_0$.
\item If $\sigma (\Phi) = \{0\}$ then
$
C^{\epsilon (n)} (\Phi, X) = \left\{
\begin{array}{cl}
\mbox{\rm Fix}\, \Phi & \mbox{if } n < \mbox{\rm dim}\, X \, , \\
\{0\} & \mbox{if } n \ge \mbox{\rm dim}\, X \, . \\
\end{array}
\right.
$
\item If $\sigma (\Phi) \subset \imath \R \setminus \{0\}$ then
$
C^{\epsilon (n)} (\Phi, X) = \left\{
\begin{array}{cl}
\mbox{\rm Per}\, \Phi & \mbox{if } n < \frac12 \, \mbox{\rm dim}\, X \, , \\[0.5mm]
\{0\} & \mbox{if } n \ge \frac12 \, \mbox{\rm dim}\, X \, . \\
\end{array}
\right.
$
\end{enumerate}
\end{lem}

\begin{proof}
Recall from Lemma \ref{lem31} that $C(\Phi, X) = \{0\}$ whenever
$\sigma (\Phi)\cap \imath \R = \varnothing$, and in this case $C^{\epsilon
    (n)}(\Phi, X) = \{0\}$ for every $n\in \N_0$, proving (i).

To establish (ii) and (iii), let $(b_1, \ldots , b_{{\rm dim}\, X})$
be an ordered basis of $X$, relative to which $A^{\Phi}$ is a single
real Jordan block. If $\sigma(\Phi) = \{0\}$ consider the two
increasing functions $f_0, f_1 :\R \to \R$, given by
$f_0(s) = \lceil {\textstyle \frac12} s \rceil $ and $f_1 (s) = \lfloor
{\textstyle \frac12} s \rfloor$, respectively. Let 
$m_{k } = f_{\epsilon (n)_{k }} \circ \cdots \circ f_{\epsilon
  (n)_0} (\mbox{\rm dim}\, X) $. As seen in the proof of Lemma \ref{lem31}, $C^{\epsilon(n), k }(\Phi, X)= \mbox{\rm
  span}\, \{b_1, \ldots , b_{m_{k }}\}$ for every $k  \in \N_0$,
provided that $m_{k  }\ge 1$, and $C^{\epsilon(n), k }(\Phi,
X)=\{0\}$ otherwise. Note that
\begin{equation}\label{eq3p55x}
\lim\nolimits_{k  \to \infty} \underbrace{f_0 \circ \cdots \circ
  f_0}_{\mbox{\rm {\small $k$ times}}} (s) =
\left\{
\begin{array}{cl}
1 & \mbox{\rm if } s > 0 \, , \\
0 & \mbox{\rm if } s \le 0 \, .
\end{array}
\right.
\end{equation}
Consequently, $\epsilon(0) = (0,0,\ldots)$, and
$\lim_{k  \to \infty}m_{k } = 1$, so $C^{\epsilon(0) }(\Phi,
X)=\mbox{\rm span}\, \{b_1\}$. Henceforth, assume 
$n = \epsilon(n)_0 + 2\epsilon(n)_1 + \ldots + 2^{\ell}\epsilon(n)_{\ell}\ge 1$,
with $\ell \in \N_0$ and $\epsilon(n)_{\ell} = 1$. Notice that
$$
f_{\epsilon (n)_{k }} \circ \cdots \circ f_{\epsilon (n)_0} (n) =
\epsilon (n)_{k  + 1} + 2 \epsilon (n)_{k  + 2} + \ldots + 2^{\ell  -
  k  - 1} \epsilon (n)_{\ell} \quad \forall k = 0 , \ldots ,   \ell  -1 \, ,
$$
hence in particular $f_{\epsilon (n)_{\ell -1} }\circ \cdots \circ f_{\epsilon (n)_0}
  (n) =\epsilon (n)_{\ell} = 1$, which implies $f_{\epsilon (n)_{\ell } }\circ
    \cdots \circ f_{\epsilon (n)_0} (n) =0$. Since $f_0, f_1$ are
    increasing, $f_{\epsilon (n)_{\ell }} \circ \cdots \circ f_{\epsilon
      (n)_0} (i) \le 0$ for all $i\le n$, and since $\epsilon
    (n)_{k }=0$ for all $k  > \ell $, it follows from (\ref{eq3p55x})
    that
$\lim_{k  \to \infty} f_{\epsilon (n)_{k }} \circ \cdots
\circ f_{\epsilon (n)_0} (i) = 0$.
In particular, $C^{\epsilon (n)}(\Phi, X) = \{0\}$ whenever $\mbox{\rm dim}\, X
\le n$. Next, notice that
$$
f_{\epsilon (n)_{k }} \circ \cdots \circ f_{\epsilon (n)_0} (n+1) 
= 1 + f_{\epsilon (n)_{k }} \circ \cdots \circ f_{\epsilon (n)_0}
(n) \quad \forall  k = 0, \ldots , \ell  \, ,
$$
hence in particular $f_{\epsilon (n)_{\ell}} \circ \cdots \circ
  f_{\epsilon (n)_0} (n+1) =1$. Again by monotonicity, $f_{\epsilon
    (n)_{\ell }} \circ \cdots \circ f_{\epsilon (n)_0} (i) \ge 1$ for all
    $i\ge n+1$, and with (\ref{eq3p55x}) $\lim_{k  \to \infty} f_{\epsilon (n)_{k }} \circ \cdots
\circ f_{\epsilon (n)_0} (i) = 1$. This shows
that $C^{\epsilon (n)}(\Phi, X) = \mbox{\rm span}\, \{b_1\} =
\mbox{\rm Fix}\, \Phi$ whenever $\mbox{\rm dim}\, X \ge n+1$, proving (ii).

Finally, to prove (iii) recall from the proof of Lemma \ref{lem31} that
$$
C^{\epsilon
(n),k}(\Phi, X) = \mbox{\rm span}\, \{b_1, \ldots , b_{\widetilde{m}_{k }}, b_{\frac12 {\rm dim}\, X + 1}, \ldots , b_{\frac12 {\rm
  dim}\, X + {\widetilde m}_{k }}\}\, , 
$$
provided that $\widetilde{m}_{k } = f_{\epsilon
(n)_{k }} \circ \cdots \circ f_{\epsilon (n)_0} (\frac12\, \mbox{\rm dim}\,
X)\ge 1$, and $C^{\epsilon (n) , k}(\Phi, X) = \{0 \}$
otherwise. Again, $\lim_{k\to \infty}\widetilde{m}_k$ equals $1$ if
$\frac12 \, \mbox{\rm dim}\, X \ge n+1$, and equals $0$ if
$\frac12 \, \mbox{\rm dim}\, X \le n$. This proves (iii) since $\mbox{\rm span}\, \{b_1, b_{\frac12 {\rm
    dim}\, X + 1}\}= \mbox{\rm Per}\, \Phi$.
\end{proof}

Given an arbitrary linear flow $\Phi$ on $X$, let $X =
\bigoplus_{j=1}^{\ell}Z_j$ be such that $\Phi_{Z_j}$ is irreducible for
every $j=1,\ldots , \ell $. By combining Lemmas \ref{lem35xx} and
\ref{lem35zz}, it is clear that $C^{\epsilon(0)} (\Phi, X) = \mbox{\rm
  Bnd}\, \Phi$, and that $\bigl(
C^{\epsilon (n)} (\Phi, X)\bigr)_{n\in \N_0}$ is a decreasing sequence
of nested spaces, with $C^{\epsilon(n) }(\Phi, X)= \{0\}$ for all $n
\ge \max_{j=1}^{\ell} \mbox{\rm dim}\, Z_j$. Moreover, for every $n\in \N_0$,
\begin{align}\label{eq3p6}
\mbox{\rm dim}\, C^{\epsilon (n)} (\Phi, X) \: = \: \: & \#\Bigl\{ 1\le
j \le \ell  :
\sigma (\Phi_{Z_j}) =\{0\}, \mbox{\rm dim}\, Z_j > n\Bigr\} \\[1mm]
&
+ \, 2 \, \# \Bigl\{ 1\le j \le \ell  :
\sigma (\Phi_{Z_j}) \subset \imath \R \setminus \{0\}, \mbox{\rm
  dim}\, Z_j > 2n\Bigr\} \, . \nonumber
\end{align}
By Lemma \ref{lem35xx}, these numbers are preserved under orbit equivalence. Thus, iterated cores, and
especially their dimensions, provide crucial information regarding the
numbers and sizes of blocks in the real Jordan normal form of
$A^{\Phi}$. However, these cores do not per se distinguish between
different eigenvalues of $A^{\Phi_{\sf C}}$. To distinguish blocks corresponding to
different elements of $\sigma (\Phi)\cap \imath \R$, ideally in a way that is preserved under orbit
equivalence, a finer analysis of $\mbox{\rm Bnd}\, \Phi$ is
needed. 

\section{Bounded linear flows}\label{sec3}

Call a linear flow $\Phi$ on $X$ {\bf bounded} if $\mbox{\rm Bnd}\,
\Phi = X$. (Recall that $X$ is a finite-dimensional normed space over
$\R$.) Clearly, every bounded linear flow is central, i.e.,
$X_{\sf C}^{\Phi} = X$; see also Section \ref{sec4}. Unless explicitly stated
otherwise, every linear flow considered in this section is bounded. 
Note that $\Phi$ is bounded precisely if $\sigma (\Phi) \subset
\imath \R$ and $A^{\Phi}$ is diagonalisable (over $\C$), in which case
Theorem \ref{thmB} takes a particularly simple form.

\begin{theorem}\label{thm4n1}
Two bounded linear flows $\Phi, \Psi$ are $C^0$-orbit equivalent if and only if $A^{\Phi},
\alpha A^{\Psi}$ are similar for some $\alpha \in \R^+$.
\end{theorem}

\noindent
The main purpose of this section is to provide a proof of Theorem
\ref{thm4n1}, divided into several steps for the reader's
convenience. Given a non-empty set $\Omega \subset \C$, refer to any element of
$\Omega_{\Q}:= \{ \omega \Q : \omega \in \Omega \}$ as a {\bf rational class}
generated by $\Omega$. Note that for every $\omega, \widetilde{\omega} \in \C$
either $\omega\Q = \widetilde{\omega} \Q$ or $\omega \Q \cap \widetilde{\omega} \Q = \{0\}$. Given $\omega \in\C$ and a bounded
linear flow $\Phi$ on $X$, associate with $\omega \Q$ the
$\Phi$-invariant subspace
$$
X_{\omega \Q}^{\Phi} := \mbox{\rm ker}\, A^{\Phi} \, \oplus \, 
\bigoplus\nolimits_{s\in \R^+ :  \imath s \in \omega \Q}
\mbox{\rm ker} \, \bigl( (A^{\Phi})^2 + s^2 \, \mbox{\rm id}_X \bigr) \:
\supset \: \mbox{\rm Fix}\, \Phi \, .
$$
A few basic properties of such spaces follow immediately from this definition.

\begin{prop}\label{prop4n1a}
Let $\Phi$ be a bounded linear flow on $X$, and $\omega, \widetilde{\omega} \in
\C$. Then:
\begin{enumerate}
\item $X_{\omega\Q}^{\Phi} \cap X_{\widetilde{\omega} \Q}^{\Phi} \ne \mbox{\rm Fix}\,
  \Phi$ if and only if $ \omega \Q = \widetilde{\omega} \Q =\lambda \Q$ for some
  $\lambda \in \sigma (\Phi)\setminus \{0\}$, and hence $X_{\omega \Q}^{\Phi} = \mbox{\rm Fix}\,
  \Phi$ precisely if $\omega  \Q \cap \sigma (\Phi) \subset \{0\}$;
\item For $\lambda ,\widetilde{ \lambda} \in \sigma (\Phi)$, $X_{\lambda \Q}^{\Phi} =
  X_{\widetilde{\lambda} \Q}^{\Phi}$ if and only if $\lambda \Q = \widetilde{\lambda} \Q$;
\item $\sum_{\lambda \in \sigma (\Phi)} X_{\lambda\Q}^{\Phi} = X$;
\item $X_{\{0\}}^{\Phi} = \mbox{\rm Fix}\, \Phi$, and
  $\bigcup_{\lambda \in \sigma(\Phi)} X_{\lambda \Q}^{\Phi} = \mbox{\rm Per}\,
  \Phi$;
\item For every $\lambda \in \sigma (\Phi)\setminus \{0\}$,
$X_{\lambda \Q}^{\Phi} = \mbox{\rm Per}_{T} \Phi$,
with 
$$
T= T_{\lambda \Q}^{\Phi} := \min \bigcap\nolimits_{s\in \R^+
  :\{-\imath s, \imath s\} \cap 
  \lambda \Q\cap \sigma (\Phi)\ne \varnothing}
\frac{2\pi}{s}\N \, ,
$$
and $\{ x \in X_{\lambda \Q}^{\Phi} : T_x^{\Phi} = T_{\lambda \Q}^{\Phi}\}$ is
open and dense in $X_{\lambda \Q}^{\Phi}$.
\end{enumerate}
\end{prop}

\noindent
Recall from Section \ref{secorb} that if $\varphi$ is
$(h,\tau)$-related to $\psi$ then $h(\mbox{\rm Per}\, \varphi) =
\mbox{\rm Per}\, \psi$, and yet $h(\mbox{\rm Per}_T \varphi)$ may not be
contained in $\mbox{\rm Per}_S \psi$ for any $S\in \R^+$. Taken together, the
following two lemmas show that such a situation cannot occur for linear flows.

\begin{lem}\label{lem4n2}
Let $\Phi, \Psi$ be bounded linear flows on $X,Y$, respectively, and
assume that $\mbox{\rm Per}\, \Phi = X$. If $\Phi$ is
$(h,\tau)$-related to $\Psi$ then there exists an $\alpha \in \R^+$ with the
following properties:
\begin{enumerate}
\item $T_{h(x)}^{\Psi} = \alpha T_x^{\Phi}$ for every $x\in X$;
\item $h(\mbox{\rm Per}_T \Phi) = \mbox{\rm Per}_{\alpha T} \Psi$ for
  every $T\in \R^+$;
\item $A^{\Phi}, \alpha A^{\Psi}$ are similar.
\end{enumerate}
\end{lem}

\begin{proof}
By Proposition \ref{prop4n1a}(iv), $\mbox{\rm Per}\, \Phi = X$ if and
only if $X_{\lambda \Q}^{\Phi} = X$ for some $\lambda \in \sigma (\Phi)$, and since $\mbox{\rm Per}\, \Psi = h (\mbox{\rm
  Per}\, \Phi)= h(X) = Y$, also $Y_{\mu \Q}^{\Psi} = Y$ for some
$\mu \in \sigma (\Psi)$. Clearly, if $\lambda = 0$ then
$\mu = 0$, in which case every $\alpha \in \R^+$ has all the desired
properties. Henceforth assume that $\lambda \ne 0$, or equivalently that
$\mbox{\rm Fix}\, \Phi \ne X$, and hence $\mu \ne 0$ as well.

To prove (i), pick any $x\in X \setminus \mbox{\rm Fix}\, \Phi$. By
Proposition \ref{prop4n1a}(v), there exists a sequence $(x_n)$ with
$x_n \to x$ and $T_{x_n}^{\Phi} = T_{\lambda \Q}^{\Phi}$ for all $n$, so $\langle
x \rangle^{\Phi} = T_{\lambda \Q}^{\Phi}/ T_x^{\Phi}$, and similarly
$\langle h(x) \rangle^{\Psi} = T_{\mu \Q}^{\Psi} /
T_{h(x)}^{\Psi}$. By Proposition \ref{prop25}(ii), therefore,
$T_{h(x)}^{\Psi} / T_x^{\Phi} = T_{\mu \Q}^{\Psi}/ T_{\lambda
  \Q}^{\Phi}$, that is, (i) holds with $\alpha = T_{\mu \Q}^{\Psi}/
T_{\lambda  \Q}^{\Phi}$.

To prove (ii), pick any $T\in \R^+$ and $x\in \mbox{\rm Per}_T \Phi
\setminus \mbox{\rm Fix}\, \Phi$. Then $T/T_x^{\Phi} = m$ for some
$m\in \N$, and (i) yields $\alpha T /T_{h(x)}^{\Psi} =m$, that is,
$h(x)\in \mbox{\rm Per}_{\alpha T}\Psi$. Thus $h(\mbox{\rm Per}_T
\Phi)\subset \mbox{\rm Per}_{\alpha T}\Psi$, and reversing the roles
of $\Phi$ and $\Psi$ yields $h(\mbox{\rm Per}_T
\Phi) = \mbox{\rm Per}_{\alpha T}\Psi$.

To prove (iii), denote $A^{\Phi}, A^{\Psi}$ simply by $A,B$,
respectively, and let $\sigma (\Phi)\setminus \{ 0 \} = \{\pm \imath a_1, \ldots , \pm
\imath a_m\}$ and $\sigma (\Psi) \setminus \{0 \}= \{\pm \imath b_1, \ldots , \pm
\imath b_n\}$ with appropriate $m,n\in \N$ and real numbers
$a_1 > \ldots > a_m>0$ and $b_1 > \ldots > b_n>0$; for convenience,
$a_0 := b_0 := 0$. Also, let $X_0 = \mbox{\rm ker}\, A$, $Y_0 =
\mbox{\rm ker}\, B$, as well as $X_s = \mbox{\rm ker}\, (A^2 + s^2 \,
\mbox{\rm id}_X)$, $Y_s = \mbox{\rm ker}\, (B^2 + s^2 \, \mbox{\rm id}_Y)$ for every
$s\in \R^+$. Since $A,B$ are diagonalisable (over $\C$), to establish
(iii) it suffices to show that in fact $m=n$, and that moreover
\begin{equation}\label{eq4l1p0}
a_k = \alpha b_k \quad \mbox{\rm and} \quad \mbox{\rm dim}\, X_{a_k} =
\mbox{\rm dim}\, Y_{b_k} \quad \forall k = 0, 1, \ldots , m \, .
\end{equation}
To this end, notice first that $\mbox{\rm Per}_{2\pi / s} \Phi =
\bigoplus_{k\in \N_0} X_{ks}$, and similarly
$\mbox{\rm Per}_{2\pi/s} \Psi = \bigoplus_{k\in \N_0} Y_{k s}$. 
For the purpose of induction, assume that, for some integer $0\le \ell < \min \{m,n\}$,
\begin{equation}\label{eq4l1p1}
a_k = \alpha b_k \quad \mbox{\rm and} \quad \mbox{\rm dim}\, X_{a_k} =
\mbox{\rm dim}\, Y_{b_k} \quad \forall k = 0, 1, \ldots , \ell \, .
\end{equation}
Now, recall that $h(X_0) = h(\mbox{\rm Fix}\, \Phi) = \mbox{\rm Fix}\, \Psi = Y_0 $, and hence $\mbox{\rm
  dim}\, X_0 = \mbox{\rm dim}\, Y_0$ by the topological invariance of
dimension \cite[ch.2]{hatch}. In other words, (\ref{eq4l1p1}) holds for $\ell
=0$. Next, let $K_{\ell} = \bigl\{k\in \N_0 : k
a_{\ell +1} \in \{a_0, a_1 , \ldots a_{\ell} \}\bigr\}$, and note that $K_{\ell}\subset \N_0$
is finite with $0\in K_{\ell}$ and $1\not \in K_{\ell}$. Moreover, since
$a_{\ell +1}>0$,
$$
\mbox{\rm Per}_{2\pi/a_{\ell +1}} \Phi =\bigoplus\nolimits_{k \in
  \N_0} \!\! \! \! X_{k a_{\ell +1}} = \bigoplus\nolimits_{k\in\N \setminus
  K_{\ell}}\!\! \! \! X_{k a_{\ell +1}}
\oplus \bigoplus\nolimits_{k\in K_{\ell}} \!\!\! \! X_{k a_{\ell +1}} =
 X_{a_{\ell +1}} \oplus \bigoplus\nolimits_{k \in K_{\ell}} \!\!\! \!
 X_{k a_{\ell +1}} \, ,
$$
whereas by (ii),
$$
h(\mbox{\rm Per}_{2\pi/a_{\ell +1}} \Phi)  = \mbox{\rm
  Per}_{2\pi\alpha/a_{\ell +1}} \Psi  =  \bigoplus\nolimits_{k \in \N_0}\!\!\!\!
Y_{k a_{\ell +1}/\alpha}  = \bigoplus\nolimits_{k \in \N \setminus K_{\ell}} \!\!\!\! Y_{k
  a_{\ell +1}/\alpha} \oplus \bigoplus\nolimits_{k \in  K_{\ell}} \!\!\! \!
Y_{k   a_{\ell +1}/\alpha} \, . 
$$
By assumption (\ref{eq4l1p1}), $\mbox{\rm dim}\, X_{k a_{\ell +1}} \! =
\mbox{\rm dim}\, Y_{k a_{\ell +1}/\alpha}$ for every $k\in K_{\ell}$. 
Since $\mbox{\rm dim}\, \mbox{\rm Per}_{2\pi/a_{\ell +1}} \! \Phi = \mbox{\rm
  dim}\, \mbox{\rm Per}_{2\pi\alpha /a_{\ell +1}} \! \Psi$, again by the
topological invariance of dimension, clearly
$\mbox{\rm dim}\, X_{a_{\ell +1}} = \sum_{k \in \N \setminus K_{\ell}}
\mbox{\rm dim}\, Y_{k a_{\ell +1}/\alpha}>0$. This shows that $\imath
k a_{\ell +1}/\alpha \in \sigma (\Psi)$ for some $k \in \N \setminus K_{\ell}$,
and also $\mbox{\rm dim}\, X_{a_{\ell +1}} \ge \mbox{\rm dim}\,
Y_{a_{\ell +1}/\alpha}$ because $1\in \N \setminus K_{\ell}$. Note
that $k a_{\ell +1}/\alpha \not
\in \{b_0, b_1, \ldots , b_{\ell} \}$ whenever $k\in \N\setminus K_{\ell}$. Thus
$k a_{\ell +1}/\alpha \le b_{\ell +1}$, and in particular $a_{\ell +1}\le \alpha
b_{\ell +1}$. The same argument with the roles of $\Phi $ and $\Psi$
reversed yields $a_{\ell +1}\ge \alpha b_{\ell +1}$ and $\mbox{\rm dim}\,
X_{b_{\ell +1}/\alpha} \le \mbox{\rm dim}\, Y_{b_{\ell +1}}$. Consequently,
(\ref{eq4l1p1}) holds with $\ell +1$ instead of $\ell $, and in
fact for all $\ell \le \min \{m,n\}$ by induction. Since $\Phi, \Psi$ are bounded, 
$X = \bigoplus_{\ell =0}^m X_{a_{\ell}}$, $Y = \bigoplus_{\ell =0}^n
Y_{b_{\ell}}$, from which it is clear that $m=n$, showing in turn that (\ref{eq4l1p0})
holds. As observed earlier, this proves that $A^{\Phi} =A$ and $\alpha A^{\Psi}= \alpha B $ are similar.
\end{proof}

As seen in the above proof, the assumption $\mbox{\rm Per}\, \Phi = X$ in Lemma
\ref{lem4n2} simply means that $X_{\lambda \Q}^{\Phi} = X$ for some
$\lambda \in \sigma (\Phi)$. Thus $\sigma (\Phi)$ generates
at most one rational class other than $\{0\}$. Even when $\sigma
(\Phi)$ does generate several rational classes, however, it turns out that if
$\Phi$ is $(h,\tau)$-related to $\Psi$ then $h(X_{\lambda \Q}^{\Phi})$
always equals $Y_{\mu \Q}^{\Psi}$ with an appropriate $\mu$. This way the homeomorphism $h$ induces a
bijection between the rational classes generated by $\sigma (\Phi)$ and $\sigma (\Psi)$.

\begin{lem}\label{lem4n3}
Let $\Phi, \Psi$ be bounded linear flows on $X,Y$, respectively. If
$\Phi$ is $(h,\tau)$-related to $\Psi$ then there exists a (unique)
bijection $h_{\Q} : \sigma (\Phi)_{\Q} \to \sigma (\Psi)_{\Q}$ with
$h(X_{\lambda \Q}^{\Phi}) = Y_{h_{\Q} (\lambda \Q)}^{\Psi}$ for every
$\lambda \in \sigma (\Phi)$; in particular, $\sigma (\Phi)$ and
$\sigma (\Psi)$ generate the same number of rational classes.
\end{lem}

\noindent
The proof of Lemma \ref{lem4n3} is facilitated by a simple topological
observation \cite{Wsupp}.

\begin{prop}\label{prop4npx}
Let $Z_1, \ldots , Z_{\ell}$ be subspaces of $X$, with $\ell \in \N$. If $\mbox{\rm
  dim}\, X/Z_{j} \ge 2$ for every $j =1,\ldots, \ell $ then $X \setminus \bigcup_{j=1}^{\ell} Z_j$ is connected.
\end{prop}

\begin{rem}
Proposition \ref{prop4npx} remains valid when $\mbox{\rm dim}\, X = \infty$, provided that each $Z_j$
is closed. It also holds when $X$ is a normed space over $\C$, in which case it suffices to require that
$Z_j \ne X$ for every $j$.
\end{rem}

\begin{proof}[Proof of Lemma \ref{lem4n3}]
Assume that $\sigma (\Phi)$ and $\sigma (\Psi)$ both generate at least two
different rational classes other than $\{0\}$. (Otherwise, the lemma
trivially is correct.) Fix any $\lambda \in \sigma (\Phi) \setminus
\{0\}$. Given $x\in X_{\lambda \Q}^{\Phi}\subset \mbox{\rm Per}\,
\Phi$, Propositions \ref{prop24} and \ref{prop4n1a}(iv)
guarantee that $h(x)\in Y_{\mu \Q}^{\Psi}$ for an appropriate,
possibly $x$-dependent $\mu \in \sigma (\Psi)$. Thus the family of
closed, connected sets $\bigl\{ h^{-1} (Y_{\mu \Q} ^{\Psi}) : \mu \in
\sigma (\Psi)\setminus \{0\} \bigr\}$ constitutes a finite cover of
$X_{\lambda \Q}^{\Phi}\setminus \mbox{\rm Fix}\, \Phi$; by Proposition
\ref{prop4npx}, the latter set is connected. If $X_{\lambda
  \Q}^{\Phi}\setminus \mbox{\rm Fix}\, \Phi$ was not entirely
contained in $h^{-1} (Y_{\mu \Q}^{\Psi})$ for some $\mu$, then one
could choose $\mu_1, \mu_2\in \sigma (\Psi)\setminus \{0\}$ with $\mu_1 \Q \ne \mu_2
\Q$ such that
\begin{align*}
\varnothing \ne h^{-1} (Y_{\mu_1 \Q}^{\Psi}) \cap  h^{-1} (Y_{\mu_2
  \Q}^{\Psi}) \cap  ( X_{\lambda \Q}^{\Phi} \setminus \mbox{\rm Fix}\,
\Phi )
& = h^{-1} (Y_{\mu_1 \Q}^{\Psi} \cap   Y_{\mu_2  \Q}^{\Psi} ) \cap 
X_{\lambda \Q}^{\Phi} \setminus \mbox{\rm Fix}\, \Phi \\
& \subset h^{-1} (\mbox{\rm Fix}\, \Psi) \setminus \mbox{\rm Fix}\, \Phi  = \varnothing \, ,
\end{align*}
an obvious contradiction. Hence indeed $h(X_{\lambda \Q}^{\Phi})\subset
Y_{\mu \Q}^{\Psi}$ for some $\mu \in \sigma (\Psi)$, and reversing the
roles of $\Phi $ and $\Psi$ yields $h(X_{\lambda \Q}^{\Phi})= Y_{\mu
  \Q}^{\Psi}$. Note that the rational class $\mu \Q$ is uniquely
determined by $\lambda \Q$, due to Proposition
\ref{prop4n1a}(ii). Letting $h_{\Q} (\lambda \Q) = \mu \Q$ precisely
when $h(X_{\lambda \Q}^{\Phi})= Y_{\mu
  \Q}^{\Psi}$ therefore (uniquely) defines a map $h_{\Q} : \sigma(\Phi)_{\Q} \to
\sigma (\Psi)_{\Q}$. Since $h$ is one-to-one, so is $h_{\Q}$, and
hence $\# \sigma (\Phi)_{\Q} \le \# \sigma (\Psi)_{\Q}$. Again, reversing
the roles of $\Phi $ and $\Psi$ yields $\# \sigma (\Phi)_{\Q} = \#
\sigma (\Psi)_{\Q}$, and $h_{\Q}$ is a bijection.
\end{proof}

Combining Lemmas \ref{lem4n2} and \ref{lem4n3}, notice that
if $\lambda \in \sigma (\Phi)\setminus \{0\}$ and $\Phi, \Psi$ are
$C^0$-orbit equivalent, then the respective (linear) flows induced on
$X_{\lambda \Q}^{\Phi}$ and $Y_{h_{\Q} (\lambda \Q)}^{\Psi}$ are
linearly flow equivalent with $\tau_x =
\alpha_{{\lambda \Q}} \mbox{\rm id}_{\R}$ for every $x\in X_{\lambda\Q}^{\Phi}$, where
$\alpha_{ {\lambda \Q}} = T_{h_{\Q} (\lambda \Q)}^{\Psi}/
T_{\lambda\Q}^{\Phi}$. As it turns out, Theorem \ref{thm4n1} is but
a direct consequence of the fact that $\alpha_{{\lambda
  \Q}}$ does not actually depend on $\lambda \Q$.

\begin{lem}\label{lem4ny}
Let $\Phi, \Psi$ be bounded linear flows. If
$\Phi$ is $(h,\tau)$-related to $\Psi$ then 
$$
\frac{T_{h_{\Q} (\lambda \Q)}^{\Psi}}{T_{\lambda \Q}^{\Phi}} =
\frac{T_{h_{\Q}( \widetilde{\lambda} \Q)}^{\Psi}}{T_{\widetilde{\lambda} \Q}^{\Phi}} \quad
\forall \lambda , \widetilde{\lambda} \in \sigma (\Phi) \setminus \{0\} \, ;
$$
here $h_{\Q}$ denotes the bijection of Lemma \ref{lem4n3}.
\end{lem}

The proof of Lemma \ref{lem4ny} given below is somewhat subtle. It
makes use of a few elementary facts regarding maps of the $2$-torus
$\T := \R^2/ \Z^2$. Specifically, recall that with every continuous map $f:\T \to \T$ one can
associate a continuous function $F_f :\R^2 \to \R^2$ with $f(x+\Z^2) =
F_f(x) + \Z^2$ for all $x\in \R^2$, as well as $\sup_{x\in \R^2} \| F_f(x) -
L_f x\|<+\infty$ for a unique $L_f
\in \Z^{2\times 2}$. Two continuous maps $f,\widetilde{f}:\T \to \T$ are homotopic if
and only if $L_f = L_{\widetilde{f}}$; moreover, $L_{f\circ \widetilde{f}} = L_f L_{\widetilde{f}}$.
Also, if $f_n \to f$ uniformly on $\T$, then $L_{f_n} = L_f$ for all
sufficiently large $n$.

Given any $u\in \R^2$, let $\kappa_u (t, z) = z + ut $ for all $(t, z) \in \R \times \T$. Thus $\kappa_u$ simply
is the Kronecker (or parallel) flow on $\T$ generated by the differential equation $\dot z =
u$. Recall that for every $z\in \T$, the $\kappa_u$-orbit $\kappa_u (\R, z)$ is
either a singleton (if $u=0$), homeomorphic to a circle (if $au \in
\Z^2\setminus \{0\}$ for some $a\in \R$), or dense in $\T$. Variants of the following simple rigidity property of
Kronecker flows appear to have long been part of dynamical systems folklore; cf.\ \cite[Thm.2]{AM}
and \cite[Lem.6]{Ladis}.

\begin{prop}\label{prop4nx}
Let $u,\widetilde{u}\in \R^2$. If $f:\T \to
\T$ is continuous and maps some $\kappa_u$-orbit into a
$\kappa_{\widetilde{u}}$-orbit, i.e., $f\circ \kappa_u(\R , z) \subset \kappa_{\widetilde{u}}
(\R , \widetilde{z})$ for some $z, \widetilde{z} \in \T$, then $L_f u, \widetilde
{u}$ are linearly dependent.
\end{prop}

\begin{rem}
All concepts regarding $\T$ recalled above have precise analogues on the
$m$-torus $\R^m / \Z^m$ for all $m\in \N$, and Proposition
\ref{prop4nx} carries over verbatim with $u,\widetilde{u}\in \R^m$ and their
associated $m$-dimensional Kronecker flows. Only the special case of $m=2$,
however, plays a role in what follows.
\end{rem}

\begin{proof}[Proof of Lemma \ref{lem4ny}]
As in the proof of Lemma \ref{lem4n2}, denote $A^{\Phi}, A^{\Psi}$
simply by $A,B$. Also, let $\lambda_1 \Q, \ldots , \lambda_{\ell} \Q$ and $\mu_1\Q , \ldots,
\mu_{\ell }\Q$, with $\ell \in \N_0$, be the distinct rational classes other
than $\{0\}$ generated by $\sigma (\Phi)$ and $\sigma (\Psi)$
respectively, and $h_{\Q} (\lambda_j \Q) = \mu_j \Q$ for $j=1, \ldots
, \ell $. As there is nothing to prove otherwise, assume $\ell \ge
2$, and let $\lambda_1 = \lambda$, $\lambda_2 = \widetilde{\lambda}$. 
For the reader's convenience, the proof is carried out in several separate steps.

\smallskip

\noindent
{\em Step I -- Topological preliminaries.}
Let $X_{j,k} =\sum_{\lambda \in \sigma (\Phi) \cap (\lambda_j \Q + \lambda_k \Q)}
X_{\lambda \Q}^{\Phi}$ for every $1\le j \le k \le \ell $, and
similarly let $Y_{j,k} = \sum_{\mu \in \sigma (\Psi) \cap (\mu_j \Q + \mu_k \Q)}
Y_{\mu \Q}^{\Psi}$. Clearly, $X_{j,k}$ is $\Phi$-invariant and
contains both $X_{\lambda_j\Q}^{\Phi}$ ($=X_{j,j}$) and $X_{\lambda_k
  \Q}^{\Phi}$. Moreover, if $\{ j_1,k_1\} \ne \{ j_2,k_2\}$ then
$X_{j_1, k_1} \cap X_{j_2, k_2} \subset X_{\lambda \Q}^{\Phi} \subset
\mbox{\rm Per}\, \Phi$, with an appropriate $\lambda \in \sigma
(\Phi)$. Also, note that $x\in X_{j,k} \setminus \mbox{\rm Per}\,
\Phi$ for some $j,k$ if and only if $\overline{ \Phi (\R,  x ) }$ is
homeomorphic to $\T$. Since this property is
preserved under orbit equivalence, given any $x\in X_{1,2}$, there
exist $j,k$, possibly depending on $x$, such that $h(x)\in
Y_{j,k}$. Thus the closed, connected sets $\{ h^{-1} (Y_{j,k}): 1 \le
j\le k \le \ell \}$ cover $X_{1,2} \setminus \mbox{\rm Per}\, \Phi$. Since
the latter set is connected, and $h^{-1} (Y_{j_1,k_1} \cap
Y_{j_2,k_2}) \subset h^{-1} (\mbox{\rm Per}\, \Psi) = \mbox{\rm Per}\,
\Phi$, the same argument as in the proof of Lemma \ref{lem4n3} demonstrates
that $h(X_{1,2}) \subset Y_{j,k}$ for some $j,k$, and since $Y_{\mu_i
  \Q}^{\Psi} = h(X_{\lambda_i\Q}^{\Phi})$ for $i=1,2$, it is clear
that in fact $h(X_{1,2})\subset Y_{1,2}$. Reversing the roles of
$\Phi$ and $\Psi$ yields $h(X_{1,2}) = Y_{1,2}$. Henceforth, assume
w.l.o.g.\ that $X_{1,2} = X$ and $Y_{1,2} = Y$. (Otherwise, all
topological notions employed in Steps III to V below have to be interpreted relative to $X_{1,2}$ and
$Y_{1,2}$, respectively.)

\smallskip

\noindent
{\em Step II -- Arithmetical preliminaries.}
For convenience, let $Z_0= \mbox{\rm ker}\, A = \mbox{\rm Fix}\,
\Phi$, and for every $j = 1, \ldots, \ell$ let $Z_j = \bigoplus_{s\in \R^+ :\imath s \in \lambda_j \Q} \mbox{\rm
  ker}\, (A^2 + s^2 \, \mbox{\rm id}_X)$, and also let $T_j = T_{\lambda_j
  \Q}^{\Phi}$. With this, $X = \bigoplus_{j=0}^{\ell} Z_j$, and for each
$j=1, \ldots, \ell$ the eigenvalue $\lambda_j$ is a rational multiple of
$2\pi \imath/ T_j$. Since $X = X_{1,2}$ by assumption, there
exist unique $k_{j,1}, k_{j,2}\in \Z$, $k_{j}\in \N$ with
$\mbox{\rm gcd}\, (k_{j,1}, k_{j,2}, k_{j})=1$ and
$$
k_j / T_j = k_{j,1} / T_1 + k_{j,2} / T_2  \quad \forall j = 1 , \ldots ,
\ell \, .
$$
Let $\cLL_{\R}$ be the subspace of $\R^{\ell}$ given by
$$
\cLL_{\R} = \{  x\in \R^{\ell} : k_{j,1} x_{,1} + k_{j,2} x_{,2} -
k_j x_{,j} = 0 \enspace  \forall j = 1, \ldots, \ell  \} \, .
$$
Note that $\cLL_{\R}$ is two-dimensional and contains two linearly
independent integer vectors. (If $\ell = 2$ then simply $\cLL_{\R} = \R^2$.) Hence $\cLL_{\Z}:= \cLL_{\R} \cap \Z^{\ell}$ is
a two-dimensional lattice, that is, a discrete additive subgroup of
$\cLL_{\R}$. Let $b_1, b_2\in \Z^{\ell}$ be a basis of this lattice,
i.e., $\cLL_{\Z} = b_1 \Z + b_2 \Z$. Though not unique per se, the basis $b_1, b_2$ is
uniquely determined under the additional assumption that
\begin{equation}\label{eq4n43}
b_{1,1}>0 \, , \quad b_{2,1}= 0 \, , \quad 0 \le b_{1,2} < b_{2,2} \, .
\end{equation}
(Note that if $\ell=2$ then simply $b_{1,1} = b_{2,2} = 1$,
$b_{2,1} = b_{1,2} = 0$.) Since clearly $[\, 1/T_1, \ldots ,
1/T_{\ell }\, ]^{\top} \in \cLL_{\R}\setminus \{0\}$, there exists a
unique $u\in \R^2\setminus \{0\}$ such that
\begin{equation}\label{eq4n44}
[\, b_1 \, | \, b_2\, ] \, u =  [\, 1/ T_1, \ldots ,
1/ T_{\ell } \, ]^{\top}  \, .
\end{equation}
Notice in particular that $u_{,1} \Q + u_{,2} \Q = 1/T_1 \Q +
1/T_2 \Q$, and hence $u_{,1}, u_{,2}$ are rationally
independent because $1/T_1, 1/T_2$ are. 

A completely analogous construction can be carried out in $Y$: Let
$W_0 = \mbox{\rm ker}\, B = \mbox{\rm Fix}\, \Psi $, and let $W_j =
\bigoplus_{s\in \R^+:\imath s \in \mu_j \Q} \mbox{\rm ker}\, (B^2 +
s^2 \, \mbox{\rm id}_Y)$ for $j = 1, \ldots , \ell$, as well as $S_j =
T_{\mu_j \Q}^{\Psi}$. Then $Y = \bigoplus_{j=0}^{\ell} W_j$, and the
same procedure as above yields unique $c_1, c_2 \in \Z^{\ell}$
with
\begin{equation}\label{eq4n45}
c_{1,1}>0 \, , \quad c_{2,1}= 0 \, , \quad 0 \le c_{1,2} < c_{2,2} \, ,
\end{equation}
(and in fact $c_{1,1} = c_{2,2} =1$, $c_{2,1} = c_{1,2}=0$ in case $\ell =2$),
together with a unique $\widetilde{u}\in \R^2 \setminus \{0\}$ such that
\begin{equation}\label{eq4n46}
[\, c_1 \, | \, c_2\, ] \, \widetilde{u}  =  [\, 1/ S_1, \ldots , 1/S_{\ell } \, ]^{\top}  \, ;
\end{equation}
again, $\widetilde{u}_{,1}, \widetilde{u}_{,2}$ are rationally independent.

\smallskip

\noindent
{\em Step III -- Construction of maps on $\T$.}
Denote by $P_0, \ldots, P_{\ell}$ the complementary
linear projections associated with the decomposition $X=\bigoplus_{j=0}^{\ell}
Z_j$, i.e., $P_0$ is the projection of $X$ onto $Z_0$ along $\bigoplus_{j=1}^{\ell}
Z_j$ etc. Note that $P_j \Phi_t = \Phi_t P_j$ for all $j=0, 1, \ldots, \ell$ and $t\in \R$, due to the $\Phi$-invariance of $Z_j$. Given any $x\in X$, define $p_x : \T \to X$ as
$$
p_x (z) = P_0 x + \sum\nolimits_{j=1}^{\ell} \Phi_{(b_{1,j} z_{,1} +
  b_{2,j} z_{,2}) T_j} P_j x \quad \forall z \in \T \, ,
$$
with $b_1, b_2\in \Z^{\ell}$ as in Step II. Clearly, $p_x$ is
continuous, $p_x(0+\Z^2) = x$, and with an appropriate constant $\nu \in \R^+$,
\begin{equation}\label{eq4n46a}
\| p_x (z) - p_{\widetilde{x}} (z)\| \le \nu \| x - \widetilde{x}\|
\quad \forall x, \widetilde{x}\in X , z \in \T \, .
\end{equation}
Thus $p_{x_n}\to p_x$ uniformly on $\T$ whenever $x_n\to x$.
Also, with the unique $u$ from (\ref{eq4n44})
$$
p_x(ut +\Z^2) = P_0 x + \sum\nolimits_{j=1}^{\ell} \Phi_{t(b_{1,j} u_{,1}
  + b_{2,j} u_{,2})T_j} P_j x = P_0 x + \sum\nolimits_{j=1}^{\ell} \Phi_t
P_j x = \Phi_t x \quad \forall t \in \R \, .
$$
In terms of the Kronecker flow $\kappa_u$ on $\T$, this simply means
that
\begin{equation}\label{eq4n47}
p_x \circ \kappa_u (t , 0+\Z^2) = \Phi_t x \quad \forall t \in \R \, .
\end{equation}
Since $u_{,1}, u_{,2}$ are rationally independent, the $\kappa_u$-orbit $\kappa_u
(\R, 0+\Z^2)$ is dense in $\T$, and hence $p_x (\T) = \overline{ \Phi (\R, x)
}$. Thus, $p_x$ maps $\T$ continuously onto the closure
of the $\Phi$-orbit of $x$, for every $x\in X$.

Next, consider $U: = \{x\in X: T_{P_j x}^{\Phi} = T_j \enspace \forall j = 1,
\ldots , \ell \}$, an open, dense, and connected subset of $X$ by
Propositions \ref{prop4n1a} and \ref{prop4npx}. Whenever $x\in U$,
note that $p_x (z) = p_x (\widetilde{z})$ implies $z-\widetilde{z} \in
\Z^2$, i.e., $p_x$ is one-to-one and hence a homeomorphism from
$\T$ onto $\overline{ \Phi (\R , x)}$. Moreover, $p_x^{-1}$
depends continuously on $x\in U$ in the following sense: If $x_n \to x$
in $U$, and if $(\widetilde{x}_n)$ converges to some $\widetilde{x}$
with $\widetilde{x}_n \in p_{x_n} (\T)$ for every $n$, then 
$\widetilde{x}\in p_x (\T)$ and $p_{x_n}^{-1} (\widetilde{x}_n) \to
p_x^{-1} (\widetilde{x})$ in $\T$. To see this, let $\widetilde{x}_n
= p_{x_n} (z_n)$ with the appropriate $z_n \in \T$, and note that
every subsequence $(z_{n_k})$ contains a subsequence that converges in
$\T$ to some $z$ with $\widetilde{x} = p_x (z)$. Since $p_x$ is
one-to-one, $z$ is uniquely determined by this property, and so $(z_n)
= \bigl( p_{x_n}^{-1}(\widetilde{x}_n)\bigr)$ converges to $z =
p_x^{-1} (\widetilde{x})$.

Again, a completely analogous construction can be carried out in $Y$:
Denote by $Q_0, \ldots , Q_{\ell}$ the projections
associated with the decomposition $Y = \bigoplus_{j=0}^{\ell} W_j$
and, given any $y\in Y$, define $q_y :\T \to Y$ as
$$
q_y (z) = Q_0 y + \sum\nolimits_{j=1}^{\ell} \Psi_{(c_{1,j} z_{,1} +
  c_{2,j} z_{,2}) S_j} Q_j y \quad \forall z  \in \T \, ,
$$
with $c_1, c_2 \in \Z^{\ell}$ as in Step II. As before, $q_y$ is
continuous, $q_y(0+\Z^2)=y$, and $q_{y_n} \to q_y$ uniformly on $\T$
whenever $y_n \to y$. In analogy to (\ref{eq4n47}), with the unique $\widetilde{u}$ from
(\ref{eq4n46}), 
\begin{equation}\label{eq4n48}
q_y \circ \kappa_{\widetilde{u}} (t , 0+\Z^2) = \Psi_t y \quad \forall t \in \R \, ,
\end{equation}
and $q_y$ maps $\T$ continuously onto
$\overline{\Psi (\R, y)}$. With the
open, dense, and connected subset $V := \{y\in Y: T_{Q_j y}^{\Psi} =
S_j \enspace \forall j = 1, \ldots , \ell\}$ of $Y$, the map $q_y$ is
one-to-one whenever $y\in V$, and $q_y^{-1}$ depends continuously on
$y\in V$, in the sense made precise earlier.

Combining the homeomorphism $h$ with the maps introduced so far yields
a continuous map $f_x:\T \to \T$, given by
$$
f_x (z) := q_{h(x)}^{-1} \circ h \circ p_x (z) \quad \forall
z  \in \T \, ,
$$
with $f_x(0+\Z^2) = 0+\Z^2$, provided that $x\in h^{-1}(V)$; see also Figure \ref{fig4}. 
\begin{figure}[!ht] 
\psfrag{tx1}[]{$\overline{ \Phi (\R, x)}$}
\psfrag{tx3}[]{$\overline{ \Psi \bigl( \R, h(x) \bigr)}$}
\psfrag{tx2}[]{$h$}
\psfrag{tx4}[]{$\T$}
\psfrag{tx6}[]{$\T$}
\psfrag{tx5}[]{$f_x = q_{h(x)}^{-1} \circ h \circ p_x$}
\psfrag{tx7}[r]{$p_x$}
\psfrag{tx8}[l]{$q_{h(x)}$}
\psfrag{t0}[]{$0+\Z^2$}
\psfrag{tu}[]{$u$}
\psfrag{tv}[]{$\widetilde{u}$}
\psfrag{txx}[]{$x$}
\psfrag{txsp}[]{$X$}
\psfrag{tyy}[]{$h(x)$}
\psfrag{tysp}[]{$Y$}
%
%
%
\begin{center}
\includegraphics{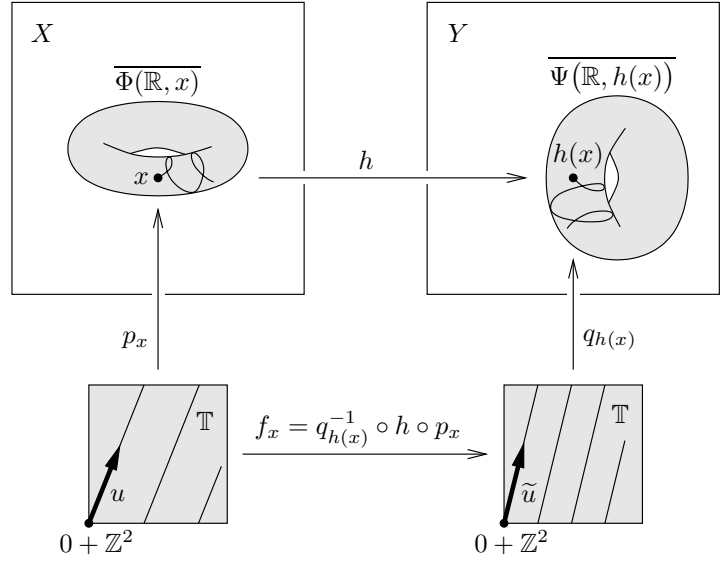}
\end{center}
\caption{The map $f_x : \T \to \T$ is well-defined and continuous
  provided that $x\in h^{-1} (V)\subset X$ and is a homeomorphism whenever
  $x\in U\cap h^{-1} (V)$.}\label{fig4}
\end{figure}
Notice that $h^{-1}(V)\subset X$ is
open, dense, and connected. As seen earlier, if $x_n \to x$ in
$h^{-1}(V)$ then $f_{x_n}  \to f_x $ pointwise. In fact, using the
analogue for $q_{h(x)}$ of (\ref{eq4n46a}), it is readily seen that 
$f_{x_n}\to f_x$ uniformly on $\T$. Thus $x\mapsto L_{f_x}$ is
continuous on $h^{-1}(V)$, and indeed constant because $h^{-1}(V)$ is
connected. In other words, $L_{f_x} = L$ for a unique $L\in
\Z^{2\times 2}$ and every $x\in h^{-1}(V)$. Recall that $p_x$, and
hence also $f_x$, is a homeomorphism whenever $x\in U \cap
h^{-1}(V)$. This set, though perhaps not connected, is open and
dense in $X$, so certainly not empty. Thus $L$ is
invertible over $\Z$, or equivalently $|\mbox{\rm det}\, L| = 1$.

\smallskip

\noindent
{\em Step IV -- Properties of $L$ and $[\, b_1 \, | \, b_2 \, ], [\,
  c_1 \, | \, c_2 \, ]$.} The scene is now set for recognizing some finer
properties of the matrices $L\in \Z^{2\times 2}$ and $[\, b_1 \, | \, b_2 \, ], [\,
  c_1 \, | \, c_2 \, ]\in \Z^{\ell \times 2}$, which truly is the crux of
  this proof. Concretely, it will be shown both that $L = I_2$ that and
  the first two rows of $[\, c_1 \, | \, c_2 \, ]$ are positive
  integer multiples of the corresponding rows of $[\, b_1 \, | \, b_2
  \, ]$. To this end, for $i=1,2$ fix $x_i\in
  Z_i $ so that $T_{x_i}^{\Phi} = T_i$. Then $y_i := h(x_i) \in W_0
  \oplus W_i$ and $T_{y_i}^{\Psi} = S_i$. Also, by (\ref{eq4n46a})
  and its analogue for $q_y$, picking $\nu \in \R^+$ large enough
  ensures that
$$
\|p_x(z) - p_{\widetilde{x} } (z)\| \le \nu \|x - \widetilde {x}\| \,
, \quad
\|q_y(z) - q_{\widetilde{y} } (z)\| \le \nu \|y - \widetilde {y}\| 
\quad
\forall x, \widetilde{x}\in X , y , \widetilde{y}\in Y , z \in \T \, .
$$
Since $S_i$ is the minimal $\Psi$-period of $y_i$, given any
$\varepsilon>0$, there exists a $\delta_1(\varepsilon) > 0$ such that
\begin{equation}\label{eq4n450}
\|\Psi_t y_i - \Psi_{\widetilde{t}} y_i\| < \delta_1 (\varepsilon) \enspace
\mbox{\rm for some } t, \widetilde{t} \in \R \quad
\Longrightarrow \quad
\min\nolimits_{k\in \Z}|(t - \widetilde{t}\, ) /S_i - k | <
\varepsilon \, .
\end{equation}
By the continuity of $h$ and the periodicity of $x_i$,
there also exists a $\delta_2(\varepsilon)>0$ such that
\begin{equation}\label{eq4n451}
\|x  - \Phi_t x_i\| < \delta_2 (\varepsilon) \enspace
\mbox{\rm for some } t \in \R \quad
\Longrightarrow \quad
\|h(x) - h(\Phi_t x_i)\| < \frac{\delta_1 (\varepsilon)}{2(1+\nu)} \, .
\end{equation}
Moreover, notice the simple estimate, valid for $x\in h^{-1}(V)$ and
$i=1,2$,
\begin{equation}\label{eq4n452}
\|q_{y_i} \circ f_x (z) - h \circ p_{x_i} (z)\| \le \nu \| h(x) - y_i \| + \| h\circ p_x (z) - h \circ p_{x_i} (z)\| \quad \forall z \in
\T \, .
\end{equation}
Finally, let $z_s, \widetilde{z}_s\in \T$ be given by
$$
z_s = \left[ \! \begin{array}{c} 0 \\ s \end{array} \! \right]+ \Z^2 \,
,\quad
\widetilde{z}_s = \left[ \! \begin{array}{r} -b_{2,2} s \\ b_{1,2}
    s \end{array} \! \right]+ \Z^2 \quad
\forall s\in \R \, ,
$$
and observe that, for $i=1,2$,
\begin{equation}\label{eq4n453}
p_{x_i} (z_s) = \Phi_{b_{2,i}s T_i} x_i \, , \quad
q_{y_i} \circ f_x (z_s) = \Psi_{\gamma_i (s)S_i} y_i \quad \forall s\in
\R \, ,
\end{equation}
with $\gamma_i (s) = [\, c_{1,i}, c_{2,i}\, ] F_{f_x} (z_s)$. Similarly
\begin{equation}\label{eq4n454}
p_{x_i} (\widetilde{z}_s) = \Phi_{(- b_{1,i} b_{2,2} + b_{1,2} b_{2,i})s T_i} x_i \, , \quad
q_{y_i} \circ f_x (\widetilde{z}_s) = \Psi_{\widetilde{\gamma}_i (s)S_i}
y_i \quad \forall s\in \R \, ,
\end{equation}
with $\widetilde{\gamma}_i(s) = [\, c_{1,i}, c_{2,i}\, ] F_{f_x}
(\widetilde{z}_s)$. With these preparations, it is possible to
analyze $f_x$ for $x$ close to $x_1$ or $x_2$. For the reader's convenience, the analysis is carried
out in two separate sub-steps.

\smallskip

\noindent
{\em Sub-step IVa -- Analysis of $f_x$ for $x$ close to $x_1$.} 
Given any $0<\varepsilon < \frac14$, let
$\delta= \delta_2(\varepsilon)/(1+\nu)$ for convenience, and assume
that $x\in h^{-1} (V)$ with $\|x - x_1\|< \delta$. Then $\|h(x) - y_1\|< \frac12
\delta_1(\varepsilon)/(1+\nu)$ by (\ref{eq4n451}), and using
(\ref{eq4n453}) with $i=1$, recalling that $b_{2,1}=0$,
$$
\| p_x (z_s) - p_{x_1} (z_s)\|  = \| p_x (z_s) - x_1\| <
\frac{\nu}{1+\nu} \delta_2 (\varepsilon) < \delta_2 (\varepsilon) \, ,
$$
and hence $\| h \circ p_x (z_s) - h \circ p_{x_1} (z_s)\|< \frac12
\delta_1 (\varepsilon)/(1+\nu)$ as well. With (\ref{eq4n452}),
therefore, $\|\Psi_{\gamma_1(s) S_1} y_1 - y_1\|<\frac12
\delta_1(\varepsilon)$, and (\ref{eq4n450}) yields $\min_{k\in \Z}|
\gamma_1(s) - k|<\varepsilon$ for all $s\in \R$. Since $\gamma_1$ is
continuous, there exists a unique $k\in \Z$ such that $|\gamma_1(s)-k
|< \varepsilon$ for all $s$. Recall that $\gamma_1 (s) = [\, c_{1,1},
0\, ]F_{f_x}(z_s)$ and that $\sup_{s\in \R} |F_{f_x} (z_s) - L_{f_x} z_s|<+\infty$. Consequently,
$$
 \sup\nolimits_{s\in \R}
|c_{1,1} L_{1,2} s| = \sup\nolimits_{s\in \R} \big|[\, c_{1,1}, 0\, ]
L z_s\big|  < + \infty \, ,
$$
and since $c_{1,1}>0$, it follows that $L_{1,2}=0$, which in turn
implies $|L_{1,1}| = |L_{2,2}|=1$, because $|\mbox{\rm det}\, L|=1$.

Similarly, using (\ref{eq4n454}) with $i=1$,
$$
\|p_x (\widetilde{z}_s) - \Phi_{- b_{1,1} b_{2,2} s T_1} x_1\| =
\|p_x(\widetilde{z}_s) - p_{x_1} (\widetilde{z}_s)\| <
\frac{\nu}{1+\nu} \delta_2 (\varepsilon) \, ,
$$
and again $\| h\circ p_x (\widetilde{z}_s) - h \circ p_{x_1}
(\widetilde{z}_s)\|<\frac12 \delta_1 (\varepsilon)/(1+\nu)$, so that
(\ref{eq4n452}) now yields
$$
\|\Psi_{\widetilde{\gamma}_1 (s) S_1} y_1 - \Psi_{\tau_{x_1} (-b_{1,1}
  b_{2,2} s T_1)} y_1\| = \| q_{y_1} \circ f_x (\widetilde{z}_s) - h
\circ p_{x_1} (\widetilde{z}_s)\| < {\textstyle \frac12}  \delta_1
(\varepsilon)  \, .
$$
Hence $\min_{k\in \Z}| \widetilde{\gamma}_1 (s) - \tau_{x_1} (-
b_{1,1} b_{2,2} s T_1)/S_1 - k | < \varepsilon$ for all $s\in
\R$. Similarly to before, and since $L_{1,2}=0$, this implies that
$$
\sup\nolimits_{s\in \R} |b_{2,2}
c_{1,1} L_{1,1} s -  \tau_{x_1} (b_{1,1} b_{2,2} s T_1)/ S_1 | < + \infty \, .
$$
As $b_{1,1}, b_{2,2}, c_{1,1}$ all are positive, and $\tau_{x_1}$ is
increasing, $L_{1,1}\ge  0$, and so in fact $L_{1,1}=1$.

Finally, let $r = 1/(b_{1,1} b_{2,2})$ and note
that $p_{x_1} (\widetilde{z}_{s+r}) = p_{x_1} (\widetilde{z}_s)$ for
all $s$, but also
\begin{align*}
\|  \Psi_{\widetilde{\gamma}_1 (s+r)S_1} & y_1 - \Psi_{\widetilde{\gamma}_1
  (s)S_1} y_1\|  = \| q_{y_1} \circ f_x (\widetilde{z}_{s+r}) -
q_{y_1} \circ f_x (\widetilde{z}_s) \| \\[1mm]
& \le 2 \nu \| h(x) - y_1 \| + \| h \circ p_x (\widetilde{z}_{s+r}) - h
\circ p_{x_1} (\widetilde{z}_{s+r})\| +  \| h \circ p_x (\widetilde{z}_{s}) - h
\circ p_{x_1} (\widetilde{z}_{s})\|\\
& < 2\nu \frac{\delta_1(\varepsilon)}{2(1+\nu)} +
\frac{\delta_1(\varepsilon)}{2(1+\nu)} +
\frac{\delta_1(\varepsilon)}{2(1+\nu)} \\
& = \delta_1 (\varepsilon) \quad
\forall s\in \R \, .
\end{align*}
Deduce from (\ref{eq4n450}) that, with a unique $k\in \Z$,
\begin{equation}\label{eq4n455}
|\widetilde{\gamma}_1 (s+r) - \widetilde{\gamma}_1 (s) + k| < \varepsilon
\quad \forall s \in \R \, .
\end{equation}
Adding (\ref{eq4n455}) with $s=0,r, \ldots, (n-1)r$ yields
$|\widetilde{\gamma}_1 (nr) - \widetilde{\gamma}_1 (0) + nk| <
n\varepsilon$ for every $n\in \N$.
Since the difference between $\widetilde{\gamma}_1(nr) = [\, c_{1,1} ,
0\, ] F_{f_x} (\widetilde{z}_{nr})$ and $[\, c_{1,1}, 0\,
]L\widetilde{z}_{nr}= - c_{1,1}n/b_{1,1}$ remains bounded as $n\to
\infty$, it follows that $| c_{1,1} /b_{1,1}- k|\le \varepsilon$. Moreover, since $\varepsilon >0$
was arbitrary and $b_{1,1}, c_{1,1}$ are positive, in fact $c_{1,1}/b_{1,1} = k\in \N$. In summary, the analysis for $x$ being sufficiently
close to $x_1$ shows that $L_{1,1} =1$, $L_{1,2} =0$, and
$c_{1,1}/b_{1,1} \in \N$.

\smallskip

\noindent
{\em Sub-step IVb -- Analysis of $f_x$ for $x$ close to $x_2$.} 
A completely analogous analysis can be carried out for $x$ being close
to $x_2$. Specifically, given any $0 < \varepsilon < \frac14$, assume
that $x\in h^{-1} (V)$ with $\|x - x_2\|< \delta$. Similarly to before, (\ref{eq4n452}) and
(\ref{eq4n453}) now yield
$$
\|\Psi_{\gamma_2 (s) S_2} y_2 - \Psi_{\tau_{x_2} (b_{2,2}sT_2)}y_2\| =
\|q_{y_2} \circ f_x (z_s) - h \circ p_{x_2} (z_s)\| < {\textstyle
  \frac12} \delta_1 (\varepsilon)  \quad \forall s \in \R \, ,
$$
and consequently $\min_{k\in \Z}| \gamma_2(s) - \tau_{x_2} (b_{2,2}s
T_2) /S_2 - k | <\varepsilon$. As $\gamma_2(s) = [\, c_{1,2},
c_{2,2}\, ]F_{f_x}(z_s)$, this implies that
$$
\sup\nolimits_{s\in \R} |c_{2,2} L_{2,2} s - \tau_{x_2} (b_{2,2} s
T_2) /S_2 | < + \infty \, ,
$$
and hence $L_{2,2}\ge 0$, so in fact $L_{2,2} = 1$. As well, $p_{x_2}
(z_{s+1/b_{2,2}}) = p_{x_2} (z_s)$ for all $s$, but also
\begin{align*}
\| & \Psi_{\gamma_2 (s + 1/b_{2,2})S_2}  y_2 - \Psi_{\gamma_2(s) S_2} y_2\|
= \| q_{y_2} \circ f_x (z_{s + 1/b_{2,2}}) - q_{y_2} \circ f_x
(z_s)\| \\
& \le 2 \nu \|h(x) - y_2 \| + \| h \circ p_x (z_{s + 1/b_{2,2}}) - h
\circ p_{x_2} (z_{s + 1/b_{2,2}})\| + \| h \circ p_x (z_{s }) - h
\circ p_{x_2} (z_{s })\| \\
& < \delta_1 (\varepsilon)  \quad \forall s \in \R \, ,
\end{align*}
implying that $|\gamma_2 (s + 1/b_{2,2}) - \gamma_2 (s) - k| <
\varepsilon$ for a unique $k\in \Z$ and all $s\in \R$. 
By adding these inequalities for $s=0, 1/b_{2,2}, \ldots, (n-1)/b_{2,2}$, similarly
to before, it follows that
$$
| c_{2,2}/b_{2,2} - k| = \limsup\nolimits_{n\to \infty} | \gamma_2
(n /b_{2,2})/ n  - k | \le \varepsilon \, ,
$$
and since $\varepsilon > 0$ was arbitrary, $c_{2,2}/b_{2,2}\in
\N$. Finally, utilizing (\ref{eq4n452}) and (\ref{eq4n454}) with
$i=2$,
$$
\|\Psi_{\widetilde{\gamma}_2 (s) S_2} y_2 - y_2\| = \| q_{y_2} \circ f_x
(\widetilde{z}_s) - h \circ p_{x_2} (\widetilde{z}_s)\| <
{\textstyle \frac12} \delta_1 (\varepsilon)  \quad \forall s \in \R \, ,
$$
yields $\min_{k\in \Z}|\widetilde{\gamma}_2(s)- k| < \varepsilon$
for all $s$. Consequently, as $L_{1,1} = L_{2,2} = 1$ and $L_{1,2}=0$,
$$
\sup\nolimits_{s\in \R} |s (b_{1,2} c_{2,2} - b_{2,2} c_{1,2} -
L_{2,1} b_{2,2} c_{2,2})|  = \sup\nolimits_{s\in \R} \big|[\, c_{1,2}, c_{2,2}\, ] L \widetilde{z}_s\big| < + \infty \, .
$$
Thus necessarily $L_{2,1} = b_{1,2}/b_{2,2} - c_{1,2}/c_{2,2}$.
By (\ref{eq4n43}) and (\ref{eq4n45}), both ratios $b_{1,2}/b_{2,2}$
and $c_{1,2}/c_{2,2}$ are non-negative and strictly less than $1$. Thus $L_{2,1}=0$ and
$b_{1,2}/b_{2,2} = c_{1,2} /c_{2,2}$. In summary, the analysis for $x$
being sufficiently close to $x_2$ shows that $L_{2,1} = 0$,
$L_{2,2}=1$, and hence $L=I_2$, as well as $c_{2,2}/b_{2,2}\in \N$ and $b_{1,2}/
b_{2,2} = c_{1,2}/ c_{2,2}$. 

\smallskip

\noindent
{\em Step V -- Concluding the proof.}  For every $x\in U \cap h^{-1} (V)$ the map
$g_x: \T \to \T$ given by $g_x = p_x^{-1} \circ h^{-1} \circ
q_{h(x)}$ is a homeomorphism of $\T$, with $g_x = f_x^{-1}$, and
carrying out Step IV with the roles of $\Phi$ and $\Psi$ reversed
yields $L_{g_x} = L^{-1} = I_2$, as well as $b_{1,1}/c_{1,1},
b_{2,2} /c_{2,2} \in \N$. This shows that in fact $b_{1,1} =
c_{1,1}$, $b_{2,2} = c_{2,2}$, and hence also $b_{1,2} =
c_{1,2}$. With this, the proof is readily completed:
Combine (\ref{eq4n47}), (\ref{eq4n48}), the definition of $f_x$, and the fact that $\Phi$ is $(h,
\tau)$-related to $\Psi$, to deduce that for every $x \in h^{-1}(V)$,
$$
f_x \circ \kappa_u (t, 0+\Z^2) = q_{h(x)}^{-1} \circ h (\Phi_t x) =
q_{h(x)}^{-1} \bigl( \Psi_{\tau_x (t)} h(x)\bigr) = \kappa_{\widetilde{u}}
(\tau_x(t), 0+\Z^2) \quad \forall t\in \R \, ,
$$
where $u,\widetilde{u}\in \R^2\setminus \{0\}$ are determined by (\ref{eq4n44})
and (\ref{eq4n46}) respectively. In particular
\begin{equation}\label{eq4n410}
\left[
\begin{array}{cc}
b_{1,1} & 0 \\
b_{1,2} & b_{2,2} 
\end{array}
\right] u =
\left[\!
\begin{array}{c}
1/T_1 \\ 1/T_2
\end{array}
\! \right]  \, , \quad
\left[
\begin{array}{cc}
c_{1,1} & 0 \\
c_{1,2} & c_{2,2} 
\end{array}
\right] \widetilde{u} =
\left[ \!
\begin{array}{c}
1/S_1 \\ 1/S_2
\end{array} \!
\right]  \, .
\end{equation}
By Proposition \ref{prop4nx}, the vectors $L_{f_x} u,\widetilde{u}$ are linearly
dependent. Since $L_{f_x} = I_2$ and the two
matrices in (\ref{eq4n410}) are identical, linear
dependence of $L_{f_x} u,\widetilde{u}$ implies linear dependence of
$[\, 1/ T_1, 1/T_{2} \, ]^{\top}, [\, 1/ S_1, 1/S_{2}\, ]^{\top}$, that
is, 
$$
0 = \left|
\begin{array}{cc}
1/T_1 & 1/S_1 \\
1/T_2 & 1/S_2 
\end{array}
\right| = \frac1{S_1 S_2} \left( \frac{S_1}{T_1} -
  \frac{S_2}{T_2}\right) \, . 
$$
Thus, $T_{\mu_1 \Q}^{\Psi} /T_{\lambda_1 \Q}^{\Phi} = T_{\mu_2
  \Q}^{\Psi} /T_{\lambda_2 \Q}^{\Phi}$, as claimed.
\end{proof}

As alluded to earlier, by combining Lemmas \ref{lem4n2}, \ref{lem4n3},
and \ref{lem4ny} it is now easy to establish the ``only if'' part of
Theorem \ref{thm4n1}. (The ``if'' part is obvious.)

\begin{proof}[Proof of Theorem \ref{thm4n1}]
As in the proof of Lemma \ref{lem4ny}, let $\lambda_1\Q , \ldots ,
\lambda_{\ell} \Q$, with $\ell \in \N_0$, be the distinct rational classes
other than $\{0\}$ generated by $\sigma (\Phi)$; again there is
nothing to prove unless $\ell \ge 2$. For convenience, denote the
generators of the linear flows induced on $X_{\lambda_j\Q}^{\Phi}$ and
$Y_{h_{\Q} (\lambda _j \Q)}^{\Psi}$ by $A_j$ and $B_j$ respectively,
and let $X_{\lambda_j \Q}^{\Phi} = \mbox{\rm Fix}\, \Phi \oplus
\bigoplus_{k=1}^{m_j} X_{a_{j,k}}$, $Y_{h_{\Q} (\lambda_j \Q)}^{\Psi}
= \mbox{\rm Fix}\, \Psi \oplus \bigoplus_{k=1}^{m_j}
Y_{a_{j,k}/\alpha_j}$, in accordance with the proof of Lemma \ref{lem4n2}. As seen
in that proof, $H_j A_j = \alpha_j B_j H_j$, with $\alpha_j =
T_{h_{\Q} (\lambda_j \Q)}^{\Psi}/ T_{\lambda_j \Q}^{\Phi}$ and an
isomorphism $H_j : X_{\lambda_j \Q}^{\Phi} \to Y_{h_{\Q} (\lambda_j
  \Q)}^{\Psi}$ satisfying $H_j \mbox{\rm Fix}\, \Phi = \mbox{\rm
  Fix}\, \Psi$ as well as $H_j X_{a_{j,k}} = Y_{a_{j,k}/\alpha_j}$
for $k=1, \ldots , m_j$. By Lemma \ref{lem4ny}, $\alpha_j = \alpha_1$
for all $j=1,\ldots , \ell$. Since $X =\sum_{j=1}^{\ell} X_{\lambda_j
  \Q}^{\Phi}$ and $Y= \sum_{j=1}^{\ell} Y_{h_{\Q} (\lambda_j \Q)}^{\Psi}$,
letting $Hx = H_j x$ for $x\in \bigoplus_{k=1}^{m_j} X_{a_{j,k}}$
and $Hx = H_1x$ for $x\in \mbox{\rm Fix}\, \Phi$, yields a linear isomorphism $H:X\to Y$ with $HA^{\Phi} = \alpha_1 A^{\Psi} H$.
\end{proof}

\section{Proof of the classification theorems}\label{sec4}

Let $\Phi$ be a linear flow on $X$, a finite-dimensional normed space
over $\R$. The subspaces
\begin{align*}
X_{\sf S}^{\Phi} & := \{x \in X : \lim\nolimits_{t\to +\infty} \Phi_t x = 0
\} \,  , \\
X_{\sf C}^{\Phi} & := \{x \in X : \lim\nolimits_{|t|\to +\infty}
e^{-\varepsilon |t|}\Phi_t x = 0 \enspace \forall \varepsilon > 0
\} \,  , \\
X_{\sf U}^{\Phi} & := \{x \in X : \lim\nolimits_{t\to -\infty} \Phi_t x = 0
\} \,  ,
\end{align*}
referred to as the {\bf stable}, {\bf central}, and {\bf unstable}
space of $\Phi$, respectively, are $\Phi$-invariant, and $X =
X_{\sf S}^{\Phi} \oplus X_{\sf C}^{\Phi} \oplus X_{\sf U}^{\Phi}$;
see, e.g., \cite{CK} for an authoritative account on linear dynamical systems. Call $\Phi$ {\bf
  hyperbolic} if $X_{\sf C}^{\Phi} = \{0\}$, and {\bf central} if
$X_{\sf C}^{\Phi} = X$. For $\bullet = {\sf S}, {\sf C}, {\sf U}$, let $P_{\bullet}^{\Phi}$ be
the linear projection onto $X_{\bullet}^{\Phi}$ along
$\bigoplus_{\circ \ne \bullet} X_{\circ}^{\Phi}$. With this and
$\Phi_{\bullet}:= \Phi_{X_{\bullet}^{\Phi}}$, clearly $\Phi$ is linearly
flow equivalent to the product flow $\bigtimes_{\bullet}
\Phi_{\bullet}$, via the isomorphism $\bigtimes_{\bullet}
P_{\bullet}^{\Phi}$ and with $\tau_x = \mbox{\rm id}_{\R}$ for all $x\in
X$. By invariance, $P_{\bullet}^{\Phi} \Phi_t = \Phi_t
P_{\bullet}^{\Phi}$ for all $t\in \R$, and hence also
$P_{\bullet}^{\Phi} A^{\Phi} = A^{\Phi} P_{\bullet}^{\Phi}$.
Notice that if $\Phi$ is $(h,\tau)$-related to $\Psi$ then
$h(X_{\sf S}^{\Phi}) = Y_{\sf S}^{\Psi}$, $h(X_{\sf U}^{\Phi}) = Y_{\sf U}^{\Psi}$,
whereas it is possible that $h(X_{\sf C}^{\Phi})\ne Y_{\sf C}^{\Psi}$.

\begin{proof}[Proof of Theorem \ref{thmB}]
To establish that (i)$\Rightarrow$(iv), assume that $\Phi$ is
$(h,\tau)$-related to $\Psi$. Then $h(X_{\sf S}^{\Phi}) = Y_{\sf
  S}^{\Psi}$, $h(X_{\sf U}^{\Phi}) = Y_{\sf U}^{\Psi}$, hence $\mbox{\rm
  dim}\, X_{\sf S}^{\Phi} = \mbox{\rm dim}\, Y_{\sf S}^{\Psi}$, $\mbox{\rm
  dim}\, X_{\sf U}^{\Phi} = \mbox{\rm dim}\, Y_{\sf U}^{\Psi}$,
and it only remains to prove the assertion regarding $\Phi_{\sf C},
\Psi_{\sf C}$. To this end, in analogy to the proofs in Section \ref{sec3}, denote
$A^{\Phi_{\sf C}}, A^{\Psi_{\sf C}}$ by $A,B$ respectively, and let $X_0 =
\mbox{\rm ker}\, A$, $Y_0 = \mbox{\rm ker}\, B$, as well as $X_s =
\mbox{\rm ker}\, (A^2 + s^2 \, \mbox{\rm id}_{X_{\sf C}^{\Phi}})$, $Y_s =
\mbox{\rm ker}\, (B^2 + s^2 \, \mbox{\rm id}_{Y_{\sf C}^{\Psi}})$ for every $s\in
\R^+$. For $s\ge 0$ and $n\in \N_0$, let $c_n^{\Phi}(s) =
\mbox{\rm dim}\, \bigl( X_s \cap C^{\epsilon (n)} (\Phi, X)\bigr)$.
Recall from Section \ref{sec2} that $\bigl( c_n^{\Phi} (s)\bigr)$ is a
decreasing sequence of integers, with $c_0^{\Phi} (s) = \mbox{\rm dim}\,
X_s$, as well as $c_n^{\Phi}(s) = 0$ for all large $n$. With this, consider
non-negative integers
$d_n^{\Phi} (s):= c_{n-1}^{\Phi} (s) - c_n^{\Phi} (s)$, with any $ n \in \N$.
As a consequence of (\ref{eq3p6}), $d_n^{\Phi}(0)$ 
simply equals the number of blocks $J_n$ in the real Jordan normal form of $A$,
whereas $\frac12 d_n^{\Phi} (s)$ equals, for every $s\in \R^+$, the number of
blocks $\left[ \begin{array}{c|r} J_n & - sI_n \\ \hline  \\[-5mm]  sI_n &
    J_n \end{array} \right]$.
Recall first that $h(X_0) = Y_0$, by Proposition \ref{prop24}, and
that $h \bigl(  C^{\epsilon (n)} (\Phi, X)\bigr) = C^{\epsilon (n)}(\Psi, Y)$
for every $n\in \N_0$, by Lemma \ref{lem35xx}. It follows that
$c_n^{\Phi} (0) = c_n^{\Psi} (0)$ for all $n\in \N_0$, and hence also
$d_n^{\Phi} (0) = d_n^{\Psi} (0)$ for all $n\in \N$. Thus, $A,B$ (and in fact $\alpha B$ for any $\alpha \in \R^+$) contain
the same number (possibly, zero) of blocks $J_n$ in their respective
real Jordan normal forms, for each $n\in \N$. Since this clearly  proves (iv) in case
$\sigma (\Phi)\cap \imath \R \subset \{0\}$, henceforth assume that
$\sigma (\Phi) \cap \imath \R \setminus \{ 0\}\ne \varnothing$.

Pick any $\lambda \in \sigma (\Phi) \cap \imath \R \setminus \{ 0 \}$,
and recall that $X_{\lambda \Q}^{\Phi} \subset \mbox{\rm Bnd}\, \Phi$
as well as $h(\mbox{\rm Bnd}\, \Phi) = \mbox{\rm Bnd}\, \Psi$. Thus
$h(X_{\lambda \Q}^{\Phi}) = Y_{h_{\Q} (\lambda \Q)}^{\Psi}$, by Lemma
\ref{lem4n3}. As in the proof of Lemma \ref{lem4n2}, for convenience
let $\lambda \Q \cap \sigma (\Phi) \setminus \{ 0 \} = \{ \pm \imath
a_1, \ldots , \pm \imath a_m\}$ and $h_{\Q} (\lambda \Q ) \cap \sigma (\Psi) \setminus \{ 0 \} = \{ \pm \imath
b_1, \ldots , \pm \imath b_m\}$, with $m\in \N$ and real
numbers $a_1 > \ldots > a_m>0$ and $b_1 > \ldots > b_m > 0$; again,
$a_0:=b_0:= 0$. As seen in that proof, $a_k = \alpha b_k$ for every
$k=0,1,\ldots, m$, with $\alpha = T_{h_{\Q} (\lambda \Q)}^{\Psi}/ T_{\lambda
  \Q}^{\Phi} \in \R^+$, but also, with the sets $K_{\ell}\subset \N_0$
defined there,
\begin{equation}\label{eq5p101}
h \left( 
X_{a_{\ell + 1}} \oplus \bigoplus\nolimits_{k\in K_{\ell}} X_{k
  a_{\ell + 1}}
\right) = Y_{b_{\ell + 1}} \oplus \bigoplus\nolimits_{k\in K_{\ell}} Y_{k
  b_{\ell + 1}} \quad \forall \ell = 0, 1, \ldots , m-1 \, .
\end{equation}
Now, assume that, for some $0\le \ell < m$,
\begin{equation}\label{eq5p102}
d_n^{\Phi} (a_k) = d_n^{\Psi} (b_k) \quad \forall n\in \N ,  k=0, 1, \ldots,
\ell  \, ;
\end{equation}
as seen earlier, (\ref{eq5p102}) holds for $\ell = 0$. With
(\ref{eq5p101}) and Lemma \ref{lem35xx}, for any $n\in \N_0$,
\begin{align*}
c_n^{\Phi} (a_{\ell + 1}) & + \! \! \sum\nolimits_{k\in
  K_{\ell}} \!\! c_n^{\Phi} (k a_{\ell + 1} ) = \mbox{\rm dim} \left( \left(X_{a_{\ell + 1}} \oplus \bigoplus\nolimits_{k \in K_{\ell}} \!\! X_{k
  a_{\ell + 1}} \right) \cap C^{\epsilon (n)} (\Phi, X )\right)\\
& =\mbox{\rm dim} \left(
\left( Y_{b_{\ell + 1}} \oplus \bigoplus\nolimits_{k \in K_{\ell}} \!\! Y_{k
  b_{\ell + 1}} \right) \cap C^{\epsilon (n)} (\Psi, Y )\right) = 
c_n^{\Psi} (b_{\ell + 1} ) + \! \! \sum\nolimits_{k\in
  K_{\ell}} \!\! c_n^{\Psi} (kb_{\ell + 1} ) \, .
\end{align*}
Together with (\ref{eq5p102}), this implies that $d_n^{\Phi}
(a_{\ell+1}) = d_n^{\Psi} (b_{\ell + 1})$ for every $n\in \N$, i.e., (\ref{eq5p102}) holds
with $\ell + 1$ instead of $\ell$, and by induction (on $\ell$) in fact for $\ell =
m$ as well. Thus, $A, \alpha B$ contain the same number of blocks
$\left[ \begin{array}{c|r} J_n & - a_k I_n \\ \hline  \\[-5mm]  a_k I_n &
    J_n \end{array} \right]$ in their respective real Jordan normal
forms, for each $n\in \N$ and $k=1, \ldots , m$.
The same argument can be applied to every rational class $\lambda \Q$
with $\lambda \in \sigma (\Phi)\cap \imath \R \setminus \{ 0 \}$. By
Lemma \ref{lem4ny}, the resulting value of $\alpha$ is independent of
$\lambda $. Thus $ A^{\Phi_{\sf C}}=A$ and $\alpha A^{\Psi_{\sf C}}= \alpha B$
are similar, as claimed.

Showing that (iv)$\Rightarrow$(iii)$\Rightarrow$(ii) 
requires straightforward, mostly routine arguments. Since details of
the latter can be found in many textbooks, e.g., \cite{Amann, Arnold,
  CK, Irwin, Meiss}, only a brief outline is included here for completeness.
To prove that (iv)$\Rightarrow$(iii), note first that
$\|x\|^{\Phi}_{\sf S} := \int_0^{+\infty} \|\Phi_t P^{\Phi}_{\sf S} x\| \,
{\rm d}t$ and its counterpart $\| \cdot  \|_{\sf S}^{\Psi}$ on $Y$
define norms on $X_{\sf S}^{\Phi}$ and $Y_{\sf S}^{\Psi}$
respectively, for which $ \|\Phi_{\bf \cdot} x\|^{\Phi}_{\sf S}$ and
$\|\Psi_{\bf \cdot} y\|^{\Psi}_{\sf S}$ are strictly decreasing to $0$ as
$t\to +\infty$ whenever
$x\ne 0$, $y\ne 0$. Consequently, given $x\in X_{\sf S}^{\Phi}\setminus
\{0\}$, there exists a unique $t_x \in \R$ with $\| \Phi_{t_x} x
\|_{\sf S} = 1$. Also, by assumption, there exists a linear isomorphism
$H_{\sf S} : X_{\sf S}^{\Phi} \to Y_{\sf S}^{\Psi}$. It is readily
confirmed that $h_{\sf S}: X_{\sf S}^{\Phi} \to Y_{\sf S}^{\Psi}$,
given by
$$
h_{\sf S}(x) = \left\{
\begin{array}{ll}
\displaystyle \frac{\Psi_{-\alpha t_x}   H_{\sf S} \Phi_{t_x} x}{\| H_{\sf S}
  \Phi_{t_x}x\|^{\Phi}_{\sf S}} & \mbox{\rm if } x\in X_{\sf S}^{\Phi}
\setminus \{0 \} \, , \\[2mm]
0 & \mbox{\rm if } x = 0 \, , 
\end{array}
\right.
$$
is a homeomorphism, and
\begin{equation}\label{eq5p103}
h_{\sf S} (\Phi_t P_{\sf S}^{\Phi} x) = \Psi_{\alpha t} h_{\sf S}
(P_{\sf S}^{\Phi} x) \quad \forall (t,x)\in \R \times X \, .
\end{equation}
A completely analogous argument, utilizing $\|x\|^{\Phi}_{\sf U} :=
\int_{-\infty}^{0} \| \Phi_t P_{\sf U}^{\Phi} x\|\, {\rm d}t$, its
counterpart $\| \cdot \|_{\sf U}^{\Psi}$ on $Y$, and a linear isomorphism $H_{\sf U} : X_{\sf
  U}^{\Phi} \to Y_{\sf U}^{\Psi}$, yields a homeomorphism $h_{\sf U}:
X_{\sf U}^{\Phi} \to Y_{\sf U}^{\Psi}$ for which (\ref{eq5p103}) holds
with $\sf U$ instead of $\sf S$. With this, clearly $\Phi_{\sf
  S}\times \Phi_{\sf U}$, $\Psi_{\sf  S}\times \Psi_{\sf U}$ are
$C^0$-flow equivalent via the homeomorphism $h_{\sf S}\times h_{\sf
  U}$ and with $\tau_x = \alpha \, \mbox{\rm id}_{\R}$ for all $x\in X_{\sf
  S}^{\Phi} \times X_{\sf U}^{\Phi}$. Since $HA^{\Phi_{\sf C}} =
\alpha A^{\Psi_{\sf C}} H$ by assumption, $H(\Phi_{\sf C})_t =
(\Psi_{\sf C})_{\alpha t} H$ for all $t\in \R$, that is, $\Phi_{\sf C}$,
$\Psi_{\sf C}$ are linearly flow equivalent.

To prove that (iii)$\Rightarrow$(ii), assume that $\Phi_{\sf
  S}\times \Phi_{\sf U}$, $\Psi_{\sf  S}\times \Psi_{\sf U}$ are
$C^0$-flow equivalent and $H_{\sf C}(\Phi_{\sf C})_t =
(\Psi_{\sf C})_{\alpha t} H_{\sf C}$ for all $t\in \R$, with some linear
isomorphism $H_{\sf C}:  X_{\sf C}^{\Phi} \to Y_{\sf C}^{\Psi}$ and $\alpha \in \R^+$. By the implication (i)$\Rightarrow$(iv)
already proved, $\mbox{\rm dim}\, X_{\sf S}^{\Phi} = \mbox{\rm dim}\, Y_{\sf S}^{\Psi}$, $\mbox{\rm
  dim}\, X_{\sf U}^{\Phi} = \mbox{\rm dim}\, Y_{\sf U}^{\Psi}$, and
the argument used above to prove that (iv)$\Rightarrow$(iii) yields a
homeomorphism $h_{\sf S}: X_{\sf S}^{\Phi} \to Y_{\sf S}^{\Psi}$
satisfying (\ref{eq5p103}), as well as its counterpart $h_{\sf U}:
X_{\sf U}^{\Phi} \to Y_{\sf U}^{\Psi}$. Combining these ingredients, 
$$
h(x) := h_{\sf S} (P_{\sf S}^{\Phi} x) + H_{\sf C} P_{\sf C}^{\Phi} x
+ h_{\sf U} (P_{\sf U}^{\Phi} x) \quad \forall x \in X \, ,
$$
defines a homeomorphism $h: X \to Y$ with $h(\Phi_t x) = \Psi_{\alpha
  t} h(x)$ for all $(t,x)\in \R \times X$. Thus $\Phi, \Psi$ are
$C^0$-flow equivalent.
The implication (ii)$\Rightarrow$(i) is trivial.
\end{proof}

The proof of Theorem \ref{thmA} given below relies on two simple
observations, both of which are straightforward linear algebra
exercises \cite{Wsupp}; recall that $X,Y$ are finite-dimensional
linear spaces over $\R$.

\begin{prop}\label{prop4a1}
Let $A,\widetilde{A}:X\to Y$ be linear, and assume that $Z\ne X$ is a subspace
of $X$ with $Z \supset \mbox{\rm ker}\, A + \mbox{\rm
ker}\, \widetilde{A}$. Then the following are equivalent:
\begin{enumerate}
\item $Ax,\widetilde{A}x$ are linearly dependent for each $x\in X\setminus Z$;
\item $\widetilde{A} = \alpha A$ for some $\alpha \in \R \setminus \{0\}$.
\end{enumerate}
\end{prop}

\begin{prop}\label{prop4a2}
Let $A:X\to X$ be linear. Then the following are equivalent:
\begin{enumerate}
\item $A$ is nilpotent, i.e., $A^n = 0$ for some $n\in \N$;
\item $A,\alpha A$ are similar for every $\alpha \in \R \setminus
  \{0\}$;
\item $A,\alpha A$ are similar for some $\alpha > 1$.
\end{enumerate}
\end{prop}

\begin{rem}
While the non-trivial implication (i)$\Rightarrow$(ii) in Proposition
\ref{prop4a1} may fail if $Z \not \supset \mbox{\rm ker}\, A + \mbox{\rm
ker}\, \widetilde{A}$, even when $\mbox{\rm dim}\, X = 1$, finite-dimensionality of
$X$ (or $Y$) is irrelevant for the result. By contrast, although
(i)$\Rightarrow$(ii)$\Rightarrow$(iii) remains valid in Proposition
\ref{prop4a2} when $\mbox{\rm dim}\, X = \infty$, every other implication
may fail in this case. Provided that $\R$ is replaced with $\C$ in (ii), Propositions \ref{prop4a1}
and \ref{prop4a2} also hold when $X,Y$ are
linear spaces over $\C$.
\end{rem}

\begin{proof}[Proof of Theorem \ref{thmA}]
Clearly (iv)$\Rightarrow$(iii)$\Rightarrow$(ii)$\Rightarrow$(i), so only the
implication (i)$\Rightarrow$(iv) requires proof.
To prepare for the argument, assume $\Phi$ is $(h,\tau)$-related to
$\Psi$ with a $C^1$-diffeomorphism $h:X\to Y$. For convenience, denote the linear
isomorphism $D_0h$ by $H$, the generators $A^{\Phi}, A^{\Psi}$ by
$A, B$, and the projections $P_{\bullet}^{\Phi},
P_{\bullet}^{\Psi}$ by $P_{\bullet}, Q_{\bullet}$, respectively. As seen earlier,
$h(X_{\sf S}^{\Phi}) = Y_{\sf S}^{\Psi}$ and hence $HX_{\sf S}^{\Phi}
= Y_{\sf S}^{\Psi}$, and similarly for $X_{\sf U}^{\Phi}$. It is possible, however, that $HX_{\sf C}^{\Phi} \ne
Y_{\sf C}^{\Psi}$, and this in turn necessitates usage of one additional pair of invariant
subspaces as follows: Recall that $X_{\sf C}^{\Phi} \supset \mbox{\rm Bnd}\, \Phi \supset \mbox{\rm ker} \, A$ and $Y_{\sf C}^{\Psi} \supset
\mbox{\rm Bnd}\, \Psi \supset \mbox{\rm ker} \, B$. By Proposition \ref{prop24},
$h(\mbox{\rm Bnd}\, \Phi) = \mbox{\rm Bnd}\, \Psi$, and hence
$H\, \mbox{\rm Bnd}\, \Phi = \mbox{\rm Bnd}\, \Psi$,
but also $A \, \mbox{\rm Bnd}\, \Phi \subset \mbox{\rm Bnd}\, \Phi $
and $B\, \mbox{\rm Bnd}\, \Psi \subset
\mbox{\rm Bnd}\, \Psi $, due to invariance. With this, let $X_{\sf HB} =
X_{\sf S}^{\Phi} \oplus \mbox{\rm Bnd}\, \Phi \oplus X_{\sf U}^{\Phi}$
and $Y_{\sf HB} =
Y_{\sf S}^{\Psi} \oplus \mbox{\rm Bnd}\, \Psi \oplus Y_{\sf
  U}^{\Psi}$. Plainly, $HX_{\sf HB} =
Y_{\sf HB}$, and crucially,
$$
Q_{\bullet} H x = H P_{\bullet} x \quad \forall x\in X_{\sf HB} , \,
{\bullet} = {\sf S}, {\sf C}, {\sf {U} }\, .
$$
By Theorem \ref{thmB}, there is nothing to
prove if $X_{\sf C}^{\Phi} = X$, or equivalently if $X_{\sf HB} \setminus
X_{\sf C}^{\Phi} = \varnothing$. Thus, henceforth assume that $X_{\sf HB} \setminus
X_{\sf C}^{\Phi} \ne  \varnothing$; notice that this in particular includes
the possibility of $X_{\sf C}^{\Phi} =\{0\}$, i.e., the case of a hyperbolic
flow $\Phi$.

With the notations introduced above, pick any $x\in X_{\sf HB} \setminus
X_{\sf C}^{\Phi}$ and $t\in \R^+$. Note that if $\tau$ was
differentiable, then differentiating the
identity $h(e^{tA} x) = e^{\tau_x (t) B} h(x)$ at $(0,0)$
would immediately yield $HA = \tau_0'(0) BH$; cf.\
\cite[p.233]{Wiggins}. The following argument mimics this process of
differentiation for arbitrary $\tau$. First observe that, for every $\varepsilon > 0$,
\begin{equation}\label{eq4t2p2}
h(e^{tA} \varepsilon x)/\varepsilon = e^{\tau_{\varepsilon x} (t) B}
h(\varepsilon x)/\varepsilon \, .
\end{equation}
Suppose that $\lim_{\varepsilon \downarrow 0} \tau_{\varepsilon x} (t)
= +\infty$. If so, $\lim_{n\to \infty} \tau_{\varepsilon_n x} (t) =
+\infty$ for every strictly decreasing sequence $(\varepsilon_n)$ with
$\lim_{n\to \infty} \varepsilon_n = 0$. In this case, applying $Q_{\sf S}$
to (\ref{eq4t2p2}) yields
$$
H e^{tA} P_{\sf S} x = Q_{\sf S} H e^{tA}x = \lim\nolimits_{n\to \infty}
e^{\tau_{\varepsilon_n x} (t) B} Q_{\sf S} h(\varepsilon_n x)/\varepsilon_n = 0 \, ,
$$
and hence $P_{\sf S} x = 0$, whereas applying $Q_{\sf U}$ yields
$$
0 = \lim\nolimits_{n\to \infty} e^{- \tau_{\varepsilon_n x} (t) B} Q_{\sf U}
h(e^{tA} \varepsilon_n x)/\varepsilon_n = Q_{\sf U} H x = H P_{\sf U} x \, ,
$$
and hence $P_{\sf U} x = 0$. Taken together, 
$x\in \mbox{\rm ker}\, (P_{\sf S} + P_{\sf U}) = X_{\sf C}^{\Phi}$,
contradicting the fact that $x\in
X_{\sf HB} \setminus X_{\sf C}^{\Phi}$. Consequently, $\rho_0 (t,x) :=
\liminf_{\varepsilon\downarrow
0} \tau_{\varepsilon x}(t) < +\infty$ and 
\begin{equation}\label{eq4t2p3}
He^{tA} x = e^{\rho_0 (t,x)B} Hx \, .
\end{equation}
Since $\rho_0 (t,x) = 0$ would imply $x\in \mbox{\rm Per}\, \Phi
\subset X_{\sf C}^{\Phi}$, clearly $\rho_0 (t,x)\in \R^+$. Also, notice that if
$\limsup_{t\downarrow 0} \rho_0 (t,x)$ was positive, possibly $+\infty$,
then $P_{\sf S} x = 0$ and $P_{\sf U} x = 0$ would follow from applying $Q_{\sf S}$ and
$Q_{\sf U}$ respectively to (\ref{eq4t2p3}), again contradicting the
fact that $x\in X_{\sf HB}
\setminus X_{\sf C}^{\Phi}$. Thus $\lim_{t\downarrow 0} \rho_0 (t,x) = 0$.

Next, deduce from (\ref{eq4t2p3}) that
\begin{equation}\label{eq4t2p4}
HAx = \lim\nolimits_{t\downarrow 0} H \frac{e^{tA} - \mbox{\rm id}_X}{t} x =
\lim\nolimits_{t\downarrow 0} \frac{\rho_0 (t,x)}{t} \cdot
\frac{e^{\rho_0 (t,x) B} - \mbox{\rm id}_Y}{\rho_0 (t,x)} Hx =
\lim\nolimits_{t\downarrow 0} \frac{\rho_0 (t,x)}{t}  BHx  \, ,
\end{equation}
and so $\rho_{0,0} (x):= \lim_{t \downarrow 0} \rho_0 (t,x)/t$ exists
because $BHx \ne 0$. Clearly, $\rho_{0,0} (x)\ge 0$. In summary, for every $x\in X_{\sf HB} \setminus X_{\sf C}^{\Phi}$ there exists
$\rho_{0,0} (x)\ge 0$ such that $HA x = \rho_{0,0} (x) BH x$; in
particular, $BHx, HA$ are linearly dependent for each $x \in
X_{\sf HB} \setminus X_{\sf C}^{\Phi}$.
Notice that $\mbox{\rm ker}\, BH = H^{-1} \mbox{\rm ker}\, B \subset
H^{-1} \mbox{\rm Bnd}\, \Psi = \mbox{\rm Bnd}\, \Phi$, as well as $\mbox{\rm ker}\, HA = \mbox{\rm
  ker}\, A \subset \mbox{\rm Bnd}\, \Phi$, and hence $\mbox{\rm ker}\, BH +
\mbox{\rm ker}\, HA \subset \mbox{\rm Bnd}\, \Phi \ne X_{\sf HB}$. Proposition
\ref{prop4a1}, applied to $BH,HA:X_{\sf HB} \to Y_{\sf HB}$ and $Z =
\mbox{\rm Bnd}\, \Phi  = X_{\sf HB} \cap X_{\sf C}^{\Phi} $, guarantees the existence of $\alpha \in \R \setminus
\{0\}$ such that 
\begin{equation}\label{eq4t2p5}
HAx = \alpha BHx \quad \forall x \in X_{\sf HB} \, ,
\end{equation}
and from (\ref{eq4t2p4}) it is clear that in fact $\alpha \in \R^{+}$. Thus
the proof is complete in case $X_{\sf HB} = X$, or equivalently whenever
$\mbox{\rm Bnd}\, \Phi = X_{\sf C}^{\Phi}$. (This, for instance, includes the case of
a hyperbolic flow $\Phi$.)

It remains to consider the case of $\mbox{\rm Bnd}\, \Phi$ being a proper
subspace of $X_{\sf C}^{\Phi}$, where necessarily $\mbox{\rm Bnd}\,
\Phi \ne \{0\}$. Deduce from Theorem \ref{thmB} that there exists a
linear isomorphism $K:X\to Y$, with $KX_{\bullet}^{\Phi} =
Y_{\bullet}^{\Psi}$ for each ${\bullet}
= {\sf S}, {\sf C}, {\sf U}$, and a $\beta \in \R^+$ such that
\begin{equation}\label{eq4t2p6}
K Ax = \beta B K x \quad \forall x \in X_{\sf C}^{\Phi} \, .
\end{equation}
Notice that (\ref{eq4t2p6}) implies $K \, \mbox{\rm Bnd}\, \Phi =
\mbox{\rm Bnd}\, \Psi$. Combine (\ref{eq4t2p5}) and (\ref{eq4t2p6}) to obtain
\begin{equation}\label{eq4t2p7}
\alpha H^{-1} BH x = \beta K^{-1} B K x \quad \forall x \in \mbox{\rm
  Bnd}\, \Phi \, .
\end{equation}
For convenience, denote the generators of $\Psi_{\mbox{{\scriptsize \rm Bnd}}\Psi}$ and $\Psi_{Y_{\sf C}^{\Psi}}$
by $B_{\sf B}$ and $B_{\sf C}$ respectively. Since $H \, \mbox{\rm
  Bnd}\, \Phi = \mbox{\rm Bnd}\, \Psi = K \, \mbox{\rm Bnd}\, \Phi$,
(\ref{eq4t2p7}) simply asserts that $\alpha B_{\sf B}$, $\beta B_{\sf B}$ are
similar. It is now helpful to distinguish two cases: On the one hand,
if $B_{\sf B}$ is not nilpotent, then $\alpha = \beta$ by Proposition
\ref{prop4a2}. In this case, $L: X \to Y$ with
\begin{equation}\label{eq4t2p8}
L = H P_{\sf S}  + KP_{\sf C} + HP_{\sf U}
\end{equation}
is a linear isomorphism, and $LA x = \alpha BLx$ for all $x\in X$. On
the other hand, if $B_{\sf B}$ is nilpotent then so is $B_{\sf C}$, and
Proposition \ref{prop4a2} shows that $\alpha B_{\sf C}$, $\beta B_{\sf C}$ are
similar. Consequently, there exists a linear isomorphism
$\widetilde{K}: X \to Y$, with $\widetilde{K} X_{\bullet}^{\Phi} = Y_{\bullet}^{\Phi}$
for each ${\bullet}={\sf S}, {\sf C}, {\sf U}$, such that
$\widetilde{K} Ax = \alpha B \widetilde{K} x$ for all $x \in X_{\sf C}^{\Phi}$.
The same argument as in the non-nilpotent case then applies, with
$\widetilde{K}$ in place of $K$ in (\ref{eq4t2p8}). In either case,
therefore, $LA = \alpha BL$, that is, $A^{\Phi}=A$ and $\alpha A^{\Psi}=\alpha B$ are similar, and the proof is complete.
\end{proof}

With the main results established, the remainder of this
section provides a brief discussion relating them to the existing literature.

In the case of hyperbolic flows, Theorem \ref{thmB} is classical
\cite{Amann, Arnold, CK, Irwin, Meiss}. What makes the result more challenging in general, then, is the
presence of a non-trivial central space. On this matter, two
key references are \cite{Kuiper, Ladis}. In \cite{Kuiper}, the equivalence (ii)$\Leftrightarrow$(iv) of
Theorem \ref{thmB} is proved utilizing a version of flow
equivalence (termed {\bf homeomorphy}, also allowing for negative
$\alpha$ in (iv), that is, for time-reversal). To put this in perspective,
notice that insisting on {\em flow\/} (rather than mere {\em orbit\/})
equivalence greatly simplifies the arguments in the present
article as well. For instance, Proposition \ref{prop25}(i) simply reads $T_{h(x)}^{\psi} = \alpha
T_x^{\varphi}$ in this case, and Lemma \ref{lem4ny}
(the proof of which required considerable effort) trivially holds. 
Consequently, to decide whether two bounded real linear flows
are $C^0$-flow equivalent, all that is needed is an elementary
analysis of periodic points, as developed in Section \ref{sec3}. In
particular, one may bypass the topological considerations
of \cite[\S3-4]{Kuiper} which the authors found
unduly hard to grasp. To deal with
non-semisimple eigenvalues on $\imath \R$, \cite[\S5]{Kuiper}
introduces a proximality relation $\Re_{\varphi}$: Specifically,
$x\Re_{\varphi} \widetilde{x}$ if, given any neighbourhoods $U,\widetilde{U}$ of $x,\widetilde{x}\in X$
respectively, there exists a $v \in X$ such that
$\varphi(\R, v)\cap U \ne \varnothing$ and $\varphi(\R,
v)\cap \widetilde{U}  \ne \varnothing$. Plainly, $\Re_{\varphi}$ is
reflexive and symmetric, but not, in general, transitive, and
if $\varphi$ is $(h,\tau)$-related to $\psi$, then $x\Re_{\varphi}
\widetilde{x}$ is equivalent to $h(x) \Re_{\psi} h(\widetilde{x} )$. 
Moreover, if $x\in C_{0} (\varphi, X)$, then $x\Re_{\varphi} 0$, and for irreducible linear flows the
converse is true also. While the usage of $\Re_{\varphi}$ in \cite{Kuiper}
thus resembles the usage of $C_0$ (and $C$) in the present article, recall from
Section \ref{secorb} that these non-uniform cores may be ill-behaved under
products --- and so may be $\Re_{\varphi}$. In fact, as per Example
\ref{ex25} with $u, \widetilde{u}$ as in (\ref{eqex1d}), it is
readily seen that $u \Re_{\varphi}0$ and $\widetilde{u}
\Re_{\varphi}0$, yet $(u,\widetilde{u})\cancel{\Re_{\varphi \times
  \varphi}} (0,0)$. Good behaviour of $\Re_{\varphi}$ under products,
which even for linear flows may or may not occur in general, appears to have been
taken for granted throughout \cite{Kuiper} without proper justification. For
comparison, recall from Section \ref{secorb} that using {\em
  uniform\/} cores allows one to avoid this difficulty altogether; see also \cite{He, WMSc}.

The focus in \cite{Ladis} is on $C^0$-orbit equivalence for linear
flows, real or complex, for which (i)$\Leftrightarrow$(iv) of Theorem
\ref{thmB} and, in essence, a version of Theorem
\ref{thm6n1} below are established. In the process, the following
terminology is employed (cf.\ also \cite[sec.II.4]{BS}): For every $x\in X$, consider the $\varphi$-invariant
closed sets
$$
D^-_{\varphi} (x) = \bigcap\nolimits_{t, \varepsilon \in \R^+}
\overline{\varphi \bigl(]-\infty, - t] , B_{\varepsilon} (x)\bigr)} \,
, \quad
D^+_{\varphi} (x) = \bigcap\nolimits_{t, \varepsilon \in \R^+}
\overline{\varphi \bigl([t, +\infty[  , B_{\varepsilon} (x)\bigr)} \, ,
$$
where $B_{\varepsilon} (x) $ denotes the open $\varepsilon$-ball
centered at $x$. With this, $D_{\varphi} (x) := D^-_{\varphi} (x)\cap
D^+_{\varphi} (x)$ and $S_{\varphi}:= \{x\in X : D_{\varphi}^- (x) \ne
\varnothing , D^+_{\varphi} (x) \ne \varnothing \}$ are called the $\varphi$-{\bf
  prolongation} of $x$ and the $\varphi$-{\bf separatrix},
respectively. Note that, in the parlance of Section \ref{secorb},
simply $D_{\varphi} (x) = C_{x,x} (\varphi, X)$ and $S_{\varphi} =
C(\varphi, X)$. A crucial lemma \cite[Lem.7]{Ladis} asserts that these sets are
well-behaved under products, in that, for instance, $S_{\varphi \times
\psi} = S_{\varphi} \times S_{\psi}$. As demonstrated by Example
\ref{ex25}, this is incorrect in general. Another crucial lemma
\cite[Lem.8]{Ladis} asserts that prolongations and separatrices are
well-behaved under orbit equivalence. Although this assertion is correct
(and a special case of Lemma \ref{prop27}), its proof in
\cite{Ladis} assumes $\tau : \R \times X \to \R$ in (\ref{eq1}) to be
continuous. The reader will have no difficulty constructing
examples of $C^0$-orbit equivalent flows on $X = \R^2$ for which $\tau$ is
not even measurable, let alone continuous. Sometimes $\tau$ can be
replaced by a continuous modification, but simple examples
show that this may not always be the case. Obviously, by Theorem
\ref{thmB}, a continuous modification of $\tau$ always exists between {\em
linear\/} flows, but surely this should be a consequence, rather than
an assumption, of any topological
classification theorem --- as it is in the present article, where
no regularity whatsoever is assumed for $\tau$ beyond the requirement
that $\tau_x$ be strictly increasing for each $x\in X$. One
observation regarding a counterpart of Lemma \ref{lem4ny} is
worth mentioning also: \cite[Prop.3]{Ladis} implicitly
assumes that no more than two different rational classes have to be
considered simultaneously. In the notation of the proof of Lemma
\ref{lem4ny}, this amounts to assuming that $X_{1,2} = X_{\lambda_1
  \Q}^{\Phi} \oplus X_{\lambda_2 \Q}^{\Phi}$. As the reader may want
to check, this drastically simplifies the
proof of that lemma, since Step II and much of Step IV become obsolete. In general,
however, such an assumption is unfounded, as it is quite possible for
three or more rational classes to be rationally dependent, and hence
for $X_{1,2}$ to be strictly larger than $ X_{\lambda_1   \Q}^{\Phi} \oplus
X_{\lambda_2 \Q}^{\Phi}$.

As far as the smooth classification of linear flows is concerned, most
textbooks mention the special case (ii)$\Leftrightarrow$(iv) of
Theorem \ref{thmA} which, of course, can be established immediately by differentiating $h(e^{tA} x) = e^{\alpha t B}
h(x)$ w.r.t.\ $x$ and $t$; see, e.g., \cite{Amann, CK, Meiss, Perko}. However, if one only assumes $C^1$-{\em orbit\/} equivalence,
where $\tau$ may depend on $x$ in a potentially very rough way, 
differentiation clearly is not available, and a finer analysis is needed.
A substantial literature exists of further classification results for linear flows (considering, e.g., Lipschitz \cite{KS} and H\"older
\cite{MM} equivalence) as well as non-autonomous \cite{Cong} and
control systems \cite{ACK2,LZ, Willems}, and also for non-linear flows derived from
them \cite{AK, KRS}.

Finally, it is worth pointing out that a similar classification
problem presents itself in discrete time, i.e., for linear operators
$A:X\to X$, $B:Y\to Y$ which are $C^{\ell}$-{\bf equivalent} if $h(Ax) =
Bh(x)$ for all $x\in X$. While for $\ell \ge 1$ this problem is 
easier than its continuous-time analogue, for $\ell = 0$ it is
significantly more difficult and, to some extent, still unresolved; see, e.g., \cite{CappSh, CSSW, Cruz, HP,
KR} and the references therein for the long history of the problem and its many ramifications.

\section{Equivalence of complex linear flows}\label{secom}

So far, the classification of finite-dimensional
linear flows developed in this article has focussed
entirely on {\em real\/} flows. Such focus is warranted by the fact that the main result, Theorem \ref{thmB}, is a
truly real theorem, whereas Theorem \ref{thmA} carries over verbatim
to {\em complex\/} flows. The goal of this concluding section is to
make these two assertions precise, via Theorems
\ref{thm6n1} and \ref{thm6n2} below.

Throughout, let $X$ be a finite-dimensional normed space
over $\K = \R$ or $\K = \C$; to avoid notational conflicts with
previous sections, the field of scalars is indicated explicitly
wherever appropriate. Further, let $X_{\R}$ be the {\bf realification}
of $X$, i.e., the linear space $X_{\R}$ equals $X$ as a set, but with the field of scalars being $\R$, and
define $\iota_X : X \to X_{\R}$ as $\iota_X (x) = x$. Thus, if $\K =
\C$, then $\iota_X$ is a homeomorphism as well as an $\R$-linear
bijection, and $\mbox{\rm dim}\,  X_{\R} = 2 \, \mbox{\rm dim}\, 
X$. (Trivially, if $\K = \R$ then $X_{\R}$ equals $X$ as a linear space, and
$\iota_X =\mbox{\rm id}_X$.) Every map $h:X\to Y$ induces a
map $h_{\R}  = \iota_Y \circ h \circ \iota_X^{-1}:X_{\R}
\to Y_{\R}$ which is continuous (one-to-one, onto) if and only
if $h$ is. If $h$ is $C^{\ell}$ or linear then so is $h_{\R}$,
but the converse is not true in general when $\K = \C$. In particular,
an $\R$-linear map $h:X\to Y$ is $\C$-linear precisely if $h_{\R}
J_X = J_Y h_{\R}$ where $J_X : X_{\R}\to X_{\R}$ is the unique
linear operator with $J_X (\, \cdot \, ) = \iota_X \bigl( \imath \, 
\iota_X^{-1} (\, \cdot \, )\bigr)$. Given any (smooth) flow $\varphi$ on $X$, its
realification $\varphi_{\R}$ on $X_{\R}$ is defined via $(\varphi_{\R})_t =
(\varphi_t)_{\R}$ for all $t\in \R$. Clearly, if $\varphi, \psi$ are
$C^{\ell}$-orbit (or -flow) equivalent then so are $\varphi_{\R},
\psi_{\R}$, and for $\ell = 0$ the converse also holds. For a $\K$-linear flow $\Phi$ on $X$, it is
readily confirmed that all fundamental dynamical objects associated
with $\Phi$ are well-behaved under
realification in that, for instance, $A^{\Phi_{\R}} = A^{\Phi}_{\R}$
and also $X_{\R \bullet}^{\Phi_{\R}} = \iota_X (X_{\bullet}^{\Phi})$ for $\bullet =
{\sf S}, {\sf C}, {\sf U}$.  With this, the topological classification
theorem for $\K$-linear flows, a generalization and immediate
consequence of Theorem \ref{thmB}, presents itself as a truly {\em real\/} result in that topological equivalence is determined completely by
the associated realifications. (The reader familiar with \cite{Ladis}
will notice how usage of realifications avoids the somewhat cumbersome
notion of $c$-{\em analog}.)

\begin{theorem}\label{thm6n1}
Let $\Phi, \Psi$ be $\K$-linear flows on $X, Y$,
respectively. Then each of the following five statements implies the other four:
\begin{enumerate}
\item $\Phi, \Psi$ are $C^0$-orbit equivalent;
\item $\Phi, \Psi$ are $C^0$-flow equivalent;
\item $\Phi_{\R}, \Psi_{\R}$ are $C^0$-orbit equivalent;
\item $\Phi_{\R}, \Psi_{\R}$ are $C^0$-flow equivalent;
\item $\mbox{\rm dim}\, X_{\sf S }^{\Phi} = \mbox{\rm dim}\,
  Y_{\sf S}^{\Psi}$, $\mbox{\rm dim}\, X_{\sf U}^{\Phi} = \mbox{\rm dim}\,
  Y_{\sf U}^{\Psi}$, and $A_{\R}^{\Phi_{\sf C}}, \alpha A^{\Psi_{\sf
      C}}_{\R}$ are $\R$-similar for some $\alpha \in \R^+$.
\end{enumerate}
\end{theorem}

\begin{proof}
For $\K = \R$, this is part of Theorem \ref{thmB}, so assume $\K =
\C$. Since $h_{\R}:X_{\R} \to Y_{\R}$ is a homeomorphism if and only if
$h:X\to Y$ is, clearly (i)$\Leftrightarrow$(iii) and (ii)$\Leftrightarrow$(iv). By Theorem
\ref{thmB}, (iii)$\Leftrightarrow$(iv)$\Leftrightarrow$(v).
\end{proof}

By contrast, smooth equivalence of $\C$-linear flows is {\em
  not\/} determined by the associated realifications. To appreciate this
basic difference, consider the $\C$-linear flows $\Phi, \Psi$
generated by $[\, \imath \,], [-\imath]$, respectively: While
$[\, \imath\, ]_{\R}, [-\imath]_{\R} $ are $\R$-similar, and hence
$\Phi_{\R}, \Psi_{\R}$ are $C^1$- (in fact, linearly) flow
equivalent, $[\, \imath\, ],
\alpha [-\imath]$ are not $\C$-similar for any $\alpha \in \R^+$, and
correspondingly $\Phi, \Psi$ are not $C^1$-orbit equivalent --- though,
of course, they are $C^0$-flow equivalent by Theorem \ref{thm6n1}. The
following generalization of Theorem \ref{thmA} shows that, just as in
this simple example, smooth equivalence always is determined by the
$\K$-similarity of generators (and not by the $\R$-similarity of realified generators).

\begin{theorem}\label{thm6n2}
Let $\Phi, \Psi$ be $\K$-linear flows. Then each of the following four statements implies the other three:
\begin{enumerate}
\item $\Phi, \Psi$ are $C^1$-orbit equivalent;
\item $\Phi, \Psi$ are $C^1$-flow equivalent;
\item $\Phi, \Psi$ are $\K$-linearly flow equivalent;
\item $A^{\Phi}, \alpha  A^{\Psi}$ are $\K$-similar for some $\alpha \in \R^+$.
\end{enumerate}
\end{theorem}

\noindent
Apart from a few simple but crucial modifications, the proof of
Theorem \ref{thm6n2} closely follows the arguments in previous
sections and only is outlined here, with most details left to the
interested reader. A noteworthy stepping stone is the following
extension of Theorem \ref{thm4n1}; note that
the increased smoothness is irrelevant when $\K = \R$, but is
essential (for the ``only if'' part) when $\K = \C$, as demonstrated
by the simple example considered earlier.

\begin{lem}\label{lem6n3}
Two bounded $\K$-linear flows $\Phi, \Psi$ are $C^1$-orbit equivalent
if and only if $A^{\Phi}, \alpha A^{\Psi}$ are $\K$-similar for some
$\alpha \in \R^+$. 
\end{lem}

\begin{proof} Only the case of
$\K = \C$ needs to be considered. Note that the definition of
$X_{\omega \Q}^{\Phi}$ makes sense in this case, in fact,
$X_{\omega\Q}^{\Phi} = \bigoplus_{s\in \R : \imath s \in \omega \Q}
\mbox{\rm ker}\, (A^{\Phi} - \imath s \, \mbox{\rm id}_X)$, and Proposition \ref{prop4n1a} carries over
verbatim. A crucial step, then, is to show that Lemma \ref{lem4n2},
with similarity in (iii) understood to mean $\C$-similarity, also remains
valid provided that $h:X\to Y$ is a $C^1$-diffeomorphism. For
assertions (i) and (ii), this is obvious, even when $h$ is only a
homeomorphism. Differentiability of $h$, however, in addition yields 
$H \, \mbox{\rm Per}_T \Phi = \mbox{\rm Per}_{\alpha T} \Psi$
for every $T\in \R^+$, where $H= D_0h$ for convenience. To establish
(iii), analogously to the proof of Lemma \ref{lem4n2}, denote
$A^{\Phi}, A^{\Psi}$ by $A,B$ respectively, and let $\sigma (\Phi)
\setminus \{0\} = \{\imath a_1, \ldots , \imath a_m\}$ with the
appropriate $m\in \N_0$ as well as real numbers $a_j$ such that
$|a_1|\ge \ldots \ge |a_m|>0$, and $a_j>a_{j+1}$ in case $|a_j| =
|a_{j+1}|$. Similarly, $\sigma (\Psi) \setminus \{0\} = \{\imath b_1,
\ldots , \imath b_n\}$ with $n\in \N_0$ as well as $|b_1|\ge \ldots
\ge |b_n|>0$, and $b_{j}> b_{j+1}$ whenever $|b_j| = |b_{j+1}|$. For
convenience, $a_0= b_0 = 0$, and $X_s = \mbox{\rm ker}\, (A - \imath
s \, \mbox{\rm id}_X)$, $Y_s = \mbox{\rm ker}\, (B - \imath
s \, \mbox{\rm id}_Y)$ for every $s\in \R$. Since $A,B$ are
diagonalisable, it suffices to prove that $m=n$, and moreover that
\begin{equation}\label{eq6n1}
a_k = \alpha b_k \quad \mbox{\rm and} \quad H X_{a_k} = Y_{b_k} \quad
\forall k = 0, 1, \ldots , m \, .
\end{equation}
To this end, notice that $\mbox{\rm Per}_{2\pi/ |s|} \Phi =
\bigoplus_{k\in \Z} X_{ks}$ and $\mbox{\rm Per}_{2\pi/ |s|} \Psi =
\bigoplus_{k\in \Z} Y_{ks}$ for every $s\in \R \setminus
\{0\}$. Clearly, if $mn=0$ then $m=n=0$, and (\ref{eq6n1})
holds. Henceforth, let $m,n\ge 1$, and assume that, for some integer $0\le \ell < \min \{m,n\}$,
\begin{equation}\label{eq6n2}
a_k = \alpha b_k \quad \mbox{\rm and} \quad H X_{a_k} = Y_{b_k} \quad
\forall k = 0, 1, \ldots , \ell \, ;
\end{equation}
since $HX_0 = Y_0$, this is clearly correct for $\ell = 0$. Letting
$$
K_{\ell} = \bigl\{ k\in \Z \setminus \{-1, 1\} : k |a_{\ell + 1}| \in
\{ a_0, a_1 , \ldots , a_{\ell}\}\bigr\} \, , 
$$
note that $K_{\ell}$ is finite, and $0\in K_{\ell}$. Deduce from
$$
\mbox{\rm Per}_{2\pi/|a_{\ell + 1}|} \Phi = \bigoplus\nolimits_{k\in
  \Z} X_{k |a_{\ell + 1}|} = X_{- |a_{\ell + 1}|} \oplus X_{|a_{\ell +
  1}|} \oplus \bigoplus\nolimits_{k \in K_{\ell}} X_{k |a_{\ell + 1}|}
\, ,
$$
together with (\ref{eq6n2}) and
\begin{align}\label{eq6n3}
HX_{-|a_{\ell + 1}|} \oplus H_{|a_{\ell + 1}|} \oplus
\bigoplus\nolimits_{k\in K_{\ell}} HX_{k |a_{\ell + 1}|} & = H\,
\mbox{\rm Per}_{2\pi / |a_{\ell + 1}|} \Phi = \mbox{\rm Per}_{2\pi
  \alpha / |a_{\ell + 1}|} \Psi \\
& = \bigoplus\nolimits_{k\in \Z \setminus K_{\ell}} Y_{k |a_{\ell +
    1}|/\alpha} \oplus \bigoplus\nolimits_{k \in K_{\ell}} Y_{k
  |a_{\ell + 1}|/\alpha} \, , \nonumber
\end{align}
that $\mbox{\rm dim}\, (X_{-|a_{\ell + 1}|} \oplus X_{|a_{\ell + 1}|})
= \sum_{k \in \Z \setminus K_{\ell}} \mbox{\rm dim}\, Y_{k |a_{\ell +
    1}|/\alpha} > 0$. Hence $\imath k |a_{\ell + 1}|/\alpha \in
\sigma (\Psi)$ for some $k\in \Z \setminus K_{\ell}$, and so in fact
$|a_{\ell + 1}| \le \alpha |b_{\ell + 1}|$, but also $\mbox{\rm dim}\,
(Y_{-|a_{\ell + 1}|/\alpha} \oplus Y_{|a_{\ell + 1}|/\alpha}) \le \mbox{\rm dim}\, (X_{-|a_{\ell + 1}|} \oplus X_{|a_{\ell + 1}|})$
because $\{-1,1\}\subset \Z \setminus K_{\ell}$. Reversing the roles
of $\Phi $ and $ \Psi$ yields that $|a_{\ell + 1}| = \alpha |b_{\ell + 1}|$ and
$\mbox{\rm dim}\, (X_{-|a_{\ell + 1}|} \oplus X_{|a_{\ell + 1}|}) = \mbox{\rm dim}\,
(Y_{-|b_{\ell + 1}|} \oplus Y_{|b_{\ell + 1}|})$. Consequently, (\ref{eq6n3}) becomes
\begin{equation}\label{eq6n4}
HX_{- |a_{\ell + 1}|} \oplus HX_{|a_{\ell + 1}|} \oplus
\bigoplus\nolimits_{k \in K_{\ell}} HX_{k |a_{\ell + 1}|} = Y_{-
  |b_{\ell + 1}|} \oplus Y_{|b_{\ell + 1}|} \oplus
\bigoplus\nolimits_{k \in K_{\ell}} Y_{k |b_{\ell +1}|} \, ,
\end{equation}
and the goal now is to show that (\ref{eq6n2}) holds with $\ell + 1$
instead of $\ell$. To this end, begin by assuming that
$X_{|a_{\ell + 1}|}\ne \{0\}$, and pick any $x\in
X_{|a_{\ell + 1}|}\setminus \{0\}$. Then $\varepsilon x \in \mbox{\rm
  Per}_{2\pi/|a_{\ell + 1}|}\Phi$ and $h(\varepsilon x) \in \mbox{\rm
Per}_{2\pi / |b_{\ell + 1}|}\Psi$ for every $\varepsilon > 0$, as well
as
\begin{equation}\label{eq6n5}
h (e^{tA} \varepsilon x)/\varepsilon = h(e^{\imath t |a_{\ell + 1}|} \varepsilon x)/ \varepsilon =
e^{\tau_{\varepsilon x} (t) B} h (\varepsilon x) / \varepsilon \, .
\end{equation}
Note that $0\le \tau_{\varepsilon x} (t) \le 2 \pi /|b_{\ell + 1}|$
for every $0\le t\le  2\pi / |a_{\ell + 1}|$, and $\tau_{\varepsilon
  x} (\, \cdot \, )$
is increasing. By the Helly selection theorem, there exists a strictly decreasing
sequence $(\varepsilon_n)$ with $\lim_{n \to \infty} \varepsilon_n =
0$, along with an increasing function $\rho$ with $\rho (0) = 0$, $\rho (2\pi/
|a_{\ell + 1}|) = 2\pi / |b_{\ell + 1}|$ such that $\lim_{n\to \infty}
\tau_{\varepsilon_n x} (t) = \rho (t)$ for almost all (in fact, all
but countably many) $0\le t \le 2\pi / |a_{\ell + 1}|$. With this,
(\ref{eq6n5}) yields
$$
H e^{\imath t |a_{\ell + 1}|} x = e^{\rho (t) B} H x \quad \mbox{\rm
  for almost all } 0\le t \le 2\pi/ |a_{\ell + 1}| \, .
$$
Note that $0 < \rho (t) < 2\pi / |b_{\ell + 1}|$ for all $0 < t < 2\pi
/ |a_{\ell + 1}|$. By monotonicity,  $\rho_0 := \lim_{t\downarrow 0}\rho
(t)$ exists, with $0 \le \rho_0 < 2\pi /|b_{\ell + 1}|$. If $\rho_0 > 0$
then $H x \in \mbox{\rm Per}_{\rho_0} \Psi$, and hence $\rho_0
|b_{\ell + 1}| \in 2\pi \N$, which is impossible. Thus $\rho_0 =
0$, and
$$
\imath |a_{\ell + 1}| H x = \lim\nolimits_{t \downarrow 0} H
\frac{e^{\imath t |a_{\ell + 1}|} - 1}{t} x =
\lim\nolimits_{t\downarrow 0} \frac{\rho (t)}{t} \cdot
\frac{e^{\rho(t) B} - \mbox{\rm id}_Y}{\rho (t)} Hx =   \lim\nolimits_{t\downarrow
  0} \frac{\rho (t)}{t} BHx \, ,
$$
showing that $\rho_{0,0}:= \lim_{t\downarrow 0} \rho(t)/t$ exists, with
$\imath |a_{\ell + 1}| Hx = \rho_{0,0} BHx$. Clearly $\rho_{0,0}\ge
0$, in fact, $\rho_{0,0}>0$ since $Hx \ne 0$, and hence $Hx \in Y_{|a_{\ell + 1}|/\rho_{0,0}}$. In other words, if
$x\in X_{|a_{\ell + 1}|}$ then $Hx \in Y_b$ for some
$b\in \R^+$. Completely analogous reasoning yields $Hx \in Y_{-b}$
for some $b\in \R^+$ whenever $x\in X_{-|a_{\ell + 1}|}$.

Recall that the goal is to establish (\ref{eq6n2}) with $\ell+ 1$
instead of $\ell$. To this end, assume first that $|a_{\ell + 1}| = |a_{\ell}|$, and hence
$a_{\ell + 1} = - a_{\ell}< 0$, but also $b_{\ell + 1} = - b_{\ell} <
0$. In this case $X_{a_{\ell + 1}} = X_{- |a_{\ell + 1}|}\ne \{0\}$,
and utilizing the preceding considerations, together with
(\ref{eq6n2}) and (\ref{eq6n4}), it follows that $H
X_{a_{\ell + 1}} \subset Y_{- |b_{\ell + 1}|} = Y_{b_{\ell +
    1}}$. Reversing the roles of $\Phi $ and $\Psi$ yields $HX_{a_{\ell + 1}}
= Y_{b_{\ell + 1}}$. Since $a_{\ell + 1} = \alpha b_{\ell +
  1}$ in this case, (\ref{eq6n2}) holds with $\ell + 1$ instead of $\ell$. 

It remains to consider the case of $|a_{\ell + 1}|< |a_{\ell}|$. Here it is convenient
to distinguish two possibilities: On the one hand, if $\ell = m-1$ or
$|a_{\ell +2}|< |a_{\ell + 1}|$ then exactly one of the spaces $X_{\pm
|a_{\ell + 1}|}$ is different from $\{0\}$. As before, it is readily seen
that $a_{\ell + 1}, b_{\ell + 1}$ have the same sign, hence $a_{\ell + 1} = \alpha b_{\ell + 1}$,
and $H X_{a_{\ell + 1}} =  Y_{b_{\ell + 1}}$, so again
(\ref{eq6n2}) holds with $\ell + 1$ instead of $\ell$. On the other
hand, if $|a_{\ell + 2}| = |a_{\ell + 1}|$ then $a_{\ell + 2} = -
a_{\ell + 1} < 0$, and the argument immediately following (\ref{eq6n3}) shows
that $|a_{\ell + 2}| = \alpha |b_{\ell + 2}|$ also. Thus $a_{\ell +
1} = \alpha b_{\ell + 1} > 0$ and $a_{\ell + 2} = \alpha b_{\ell + 2}
< 0$, and analogous reasoning as before results in $H X_{a_{\ell +1}} = Y_{b_{\ell + 1}}$, $H X_{a_{\ell +
  2}} = Y_{b_{\ell + 2}}$. Again, (\ref{eq6n2}) holds
with $\ell + 1$ (in fact, $\ell + 2$) instead of $\ell$. Induction now
proves (\ref{eq6n1}), and since $X =
\bigoplus_{\ell = 0}^m X_{a_{\ell}}$, $Y = \bigoplus_{\ell = 0}^n
Y_{b_{\ell}}$, clearly $m=n$. As indicated earlier, this establishes
Lemma \ref{lem4n2}(iii) in the case of $\K = \C$ and under the
assumption that $h$ is a $C^1$-diffeomorphism. 

With Lemma \ref{lem4n2} thus extended to complex linear flows, the
remainder of the proof proceeds exactly as in Section \ref{sec3},
since Lemmas \ref{lem4n3} and \ref{lem4ny} carry over without any
modifications, and so does the proof of Theorem \ref{thm4n1}. (In
fact, with the notation used in that proof, the linear isomorphism
$H_j$ can be taken to be the restriction of $D_0h$ to $X_{\lambda_j
  \Q}^{\Phi}$. Thus, instead of being defined abstractly by $A^{\Phi}, \alpha A^{\Psi}$ both
being diagonalisable and having the same eigenvalues with matching
geometric multiplicities, $H$ now simply equals $D_0h$.)
\end{proof}

\begin{proof}[Outline of proof of Theorem \ref{thm6n2}]
Again, one only needs to consider the case of $\K = \C$ and establish
(i)$\Rightarrow$(iv), as in the proof of Theorem \ref{thmA}. The
crucial step is to extend Lemma \ref{lem6n3} from {\em bounded\/} to
{\em central\/} $\K$-linear flows, i.e., to show that (i) implies
$\C$-similarity of $A^{\Phi_{\sf C}}, \alpha A^{\Psi_{\sf C}}$ for
some $\alpha \in \R^+$. To prove the latter along the lines of the
proof of Theorem \ref{thmB}, with $X_s = \mbox{\rm ker}\,
(A^{\Phi_{\sf C}} - \imath s \, \mbox{\rm
id}_{X_{\sf C}^{\Phi}})$, $Y_s = \mbox{\rm ker}\, (A^{\Psi_{\sf C}} - \imath s \,
\mbox{\rm id}_{Y_{\sf C}^{\Psi}})$ for
every $s\in \R$, it is necessary to first adjust the auxiliary
results of Section \ref{sec2}, notably Lemmas \ref{lem31} and \ref{lem35zz}, for complex linear flows. With the
details of these routine adjustments left to the reader, the
non-negative integer $d_n^{\Phi} (s)$ now equals, for each $n\in \N$
and $s\in \R$, the number of blocks $\imath s I_n + J_n$ in the (complex) Jordan
normal form of $A^{\Phi}$. Utilizing the proof of
Lemma \ref{lem6n3}, deduce that $m=n$, as well as $a_k = \alpha b_k$
for $k=0,1 \ldots, m$ and an appropriate $\alpha \in \R^+$, and that
moreover  $d_n^{\Phi} (a_k) = d_n^{\Psi} (b_k)$ for all $n,k$. Again,
the differentiability of $h, h^{-1}$ is essential here, unlike in the
proof of Theorem \ref{thmB}. Thus, $A^{\Phi_{\sf C}}, \alpha
A^{\Psi_{\sf C}}$ indeed are $\C$-similar, which in turn
proves that (i)$\Rightarrow$(iv) in case $X_{\sf C}^{\Phi} =
X$. Apart from the fact that this latter extension of Lemma
\ref{lem6n3}, rather than Theorem \ref{thmB}, has to be used to
establish (\ref{eq4t2p6}), the remaining argument now is identical
to the one proving Theorem \ref{thmA} in Section \ref{sec4}.
\end{proof}

To finally illustrate the difference between real and complex linear flows in
dimensions $1$ and $2$, recall that on $X=\R$ there are
exactly three ($C^0$- or $C^1$-) equivalence classes of $\R$-linear
flows, represented by $\Phi(t,x) = e^{t a}x$ with $a\in \{-1,0,1\}$. By
contrast, on $X=\C$ there are four $C^0$-equivalence classes of
$\C$-linear flows, represented by $\Phi (t,x) = e^{tc}x$ with $c\in
\{-1,0,1,\imath\}$, but infinitely many $C^1$-equivalence classes,
corresponding to $c\in \{\omega \in \C :
|\omega|=1\}\cup \{0\}$. Similarly, on $X= \R^2$ there are exactly eight
$C^0$-equivalence classes of $\R$-linear flows, listed in (\ref{eq3}),
whereas for $\C$-linear flows on $X=\C^2$, all $C^0$-equivalence
classes are given by all the matrices in (\ref{eq3}) except
for the left-most, together with
$$
\pm \left[\begin{array}{cc}
1 & 0 \\0 & \imath
\end{array}
\right]  , \enspace
\left[\begin{array}{cc}
\imath  & 1 \\0 & \imath 
\end{array}
\right]  , \enspace
\left[\begin{array}{cc}
\imath  & 0 \\ 0 & \imath a
\end{array}
\right] \quad (0\le a \le 1) \, ,
$$
and all $C^1$-equivalence classes are given by
$$
\left[\begin{array}{cc}
0 & 0 \\0 & 0
\end{array}
\right]  , \enspace
\left[\begin{array}{cc}
0 & 1 \\0 & 0
\end{array}
\right]  , \enspace
\left[\begin{array}{cc}
c & 1 \\0 & c
\end{array}
\right]  , \enspace
\left[\begin{array}{cc}
c & 0 \\0 & \omega
\end{array}
\right]   \quad (c,\omega \in \C, |c|=1, |\omega|\le 1) \, .
$$
The reader may want to compare the latter to the seven singleton classes
and five infinite families that make up all $C^1$-equivalence classes
of $\R$-linear flows on $X = \R^2$, as listed in the Introduction;
cf.\ also \cite[Ex.1]{Markus}.

\subsubsection*{Acknowledgements}

The first author was partially supported by an {\sc Nserc} Discovery
Grant. The authors gratefully acknowledge helpful comments and
suggestions made by C.\ Kawan, T.\ Oertel-J\"ager, G.\ Peschke, C.\
P\"otzsche, and V.\ Troitsky.


\begin{thebibliography}{99}

\bibitem{AS} M.\ Abramowitz and I.A.\ Stegun, {\em Handbook of
    Mathematical Functions}, National Bureau of Standards Applied
  Mathematics Series {\bf 55}, 1964.

\bibitem{AM} A.\ Aeppli and L.\ Markus, Integral equivalence of vector fields on manifolds and
              bifurcation of differential systems, {\em Amer.\ J.\
                Math.\/} {\bf 85}(1963), 633--654.

\bibitem{Amann} H.\ Amann, {\em Ordinary differential equations: an
    introduction to non-linear analysis}, deGruyter, 1990.

\bibitem{Arnold} V.I.\ Arnold, {\em Ordinary differential
    equations}, 3rd edition, Springer, 1992.

\bibitem{ArrP} D.K.\ Arrowsmith and C.M.\ Place, {\em An Introduction to
    Dynamical Systems}, Cambridge University Press, 1990.

\bibitem{ACK1} V.\ Ayala, F.\ Colonius, and W.\ Kliemann, Wolfgang, Dynamical characterization of the Lyapunov form of matrices,
{\em Linear Algebra Appl.} {\bf 402}(2005), 272--290.

\bibitem{ACK2} V.\ Ayala, F.\ Colonius, and W,\ Kliemann, On
  topological equivalence of linear flows with applications to
  bilinear control systems, {\em J.\ Dyn.\ Control Syst.} {\bf
    13}(2007), 337--362.

\bibitem{AK} V.\ Ayala and C.\ Kawan, Topological conjugacy of real projective flows,
{\em J. Lond. Math. Soc. (2)\/}
  {\bf 90}(2014), 49--66.

\bibitem{BS} N.P.\ Bathia and G.P.\ Szeg\"o, {\em Stability Theory of
    Dynamical Systems}, Grundlehren der mathematischen Wissenshaften
  {\bf 161}, Springer, 1970.

\bibitem{CappSh} S.E.\ Cappell and J.L.\ Shaneson, Nonlinear
  similarity, {\em Ann.\ of Math.\ (2)} {\bf 113}(1981), 315--355.

\bibitem{CSSW} S.E.\ Cappell, J.L.\ Shaneson, M.\ Steinberger, and
  J.E.\ West, Nonlinear similarity begins in dimension six, {\em
    Amer. J. Math.} {\bf 111}(1989), 717--752.

\bibitem{CK} F.\ Colonius and W.\ Kliemann, {\em Dynamical systems and linear algebra},
   Graduate Studies in Mathematics {\bf 158}, American Mathematical Society, 2014.

\bibitem{CS} F.\ Colonius and A.J.\ Santana, Topological conjugacy for affine-linear flows and control
              systems, {\em Commun.\ Pure Appl.\ Anal.} {\bf 10}(2011), 847--857.

\bibitem{Cong} N.D.\ Cong, Topological classification of linear
  hyperbolic cocycles, {\em J.\ Dynam.\ Differential Equations} {\bf
    8}(1996), 427--467.

\bibitem{Cruz} R.N.\ Cruz, Linear and Lipschitz similarity,
  {\em Linear Algebra Appl.\/} {\bf 151}(1991), 17--25.

\bibitem{DSS} A.\ Da Silva, A.J.\ Santana, and S.N.\ Stelmastchuk, Topological conjugacy of linear systems on Lie groups,
  {\em Discrete Contin. Dyn. Syst.} {\bf 37}(2017), 3411--3421.

\bibitem{hatch} A.\ Hatcher, {\em Algebraic Topology}, Cambridge
  University Press, 2002.

\bibitem{He} T.\ He, Topological Conjugacy of Non-hyperbolic Linear
  Flows, preprint (2017), arXiv:1703.4413.

\bibitem{HP} W.C.\ Hsiang and W.\ Pardon, When are topologically equivalent orthogonal transformations
              linearly equivalent?, {\em Invent. Math.} {\bf
                68}(1982), 275--316.

\bibitem{Irwin} M.C.\ Irwin, {\em Smooth dynamical systems}, Advanced
  Series in Nonlinear Dynamics {\bf 17}, World Scientific, 2001.

\bibitem{KH} A.\ Katok and B.\ Hasselblatt, {\em Introduction to the Modern Theory of Dynamical Systems},
  Encyclopedia of Mathematics and its Applications {\bf 54}, Cambridge University Press, 1995.

\bibitem{KS} C.\ Kawan and T.\ Stender, Lipschitz conjugacy of linear flows,
{\em J. Lond. Math. Soc. (2)\/} {\bf 80}(2009),  699--715.

\bibitem{KRS} C.\ Kawan,  O.G.\ Roc\'\i o,  and A.J.\ Santana, On topological conjugacy of left invariant flows on semisimple
              and affine Lie groups, {\em Proyecciones} {\bf 30}(2011), 175--188.

\bibitem{Kuiper} N.H.\ Kuiper, The topology of the solutions of a linear differential
              equation on {$R^{n}$}, p.\ 195--203 in: 
{\em Manifolds--Tokyo 1973 (Proc.\ Internat.\ Conf., Tokyo, 1973)},
Univ.\ Tokyo Press, 1975.

\bibitem{KR} N.H.\ Kuiper and J.W.\ Robbin,  Topological
  classification of linear endomorphisms, {\em Invent. Math.} {\bf 19}
 (1973), 83--106.

\bibitem{Ladis} N.N.\ Ladis, Topological equivalence of linear flows,
  {\em Differencialnye Uravnenija\/} {\bf 9} (1973),
  1222--1235.

\bibitem{LZ} J.\ Li, and Z.\ Zhang, Topological classification of linear control systems---an
              elementary analytic approach, {\em J. Math. Anal. Appl.}
              {\bf 402}(2013), 84--102.

\bibitem{Markus} L.\ Markus, Topological types of polynomial
  differential equations, {\em Trans. Amer. Math. Soc.\/} {\bf
    171}(1972), 157--178.

\bibitem{MM} P.D.\ McSwiggen and K.R.\ Meyer, Conjugate phase portraits
  of linear systems, {\em Amer. Math. Monthly\/} {\bf 115}(2008), 596--614.

\bibitem{Meiss} J.D.\ Meiss, {\em Differential dynamical systems},
   Mathematical Modeling and Computation {\bf 14}, SIAM, 2007.
        
\bibitem{PdM} J.\ Palis and W.\ de Melo, {\em Geometric Theory of
    Dynamical Systems}, Springer, 1982.

\bibitem{Perko} L.\ Perko, {\em Differential Equations and Dynamical
    Systems}, Texts in Applied Mathematics {\bf 7}, Springer, 2001.

\bibitem{R} C.\ Robinson, {\em Dynamical Systems. Stability, Symbolic
    Dynamics, and Chaos}, CRC Press, 1995.

\bibitem{Wiggins} S.\ Wiggins, {\em Introduction to Applied Nonlinear
    Dynamical Systems and Chaos}, Texts in Applied Mathematics {\bf
    2}, Springer, 1990.

\bibitem{Willems} J.C.\ Willems, Topological classification and
  structural stability of linear systems, {\em J.\ Differential
    Equations} {\bf 35}(1980), 306--318.

\bibitem{WMSc} A.\ Wynne, {\em Classification of Linear Flows}, MSc
  thesis, University of Alberta, 2016.

\bibitem{Wsupp} A.\ Wynne,  The classification of linear flows --
  Supplementary material, in preparation.

\end{thebibliography}
\end{document}